\documentclass[a4paper, leqno, 10pt]{amsart}

\usepackage[T1]{fontenc}	
\usepackage{graphicx}
\usepackage{amssymb,centernot}
\usepackage{upgreek}
\usepackage{mathrsfs}
\usepackage{pdfsync}
\usepackage[usenames,dvipsnames]{xcolor}
\usepackage[inline,shortlabels]{enumitem}
\usepackage[sort,numbers]{natbib}						
\usepackage{verbatim}								
\usepackage{dsfont}									
\usepackage{microtype}

\usepackage{pdflscape}
\usepackage{afterpage}

\usepackage[bookmarks,
						colorlinks,				
						breaklinks,unicode,
						hypertexnames=false,
						citecolor=OliveGreen,
						linkcolor=Maroon
						]{hyperref} 

\usepackage{multirow}

\usepackage{xypic}
\usepackage{mathtools}
\usepackage{xspace}

\usepackage[a4paper,top=3.2cm,bottom=2.7cm,left=3cm,right=3cm, bindingoffset=0mm]{geometry}
\usepackage{booktabs}

\usepackage{tikz}
\usetikzlibrary{calc, positioning, shapes.geometric, patterns, cd,decorations.pathreplacing}
\usepackage{pgfplots}

\makeatletter
\newcommand*{\centerfloat}{%
  \parindent \z@
  \leftskip \z@ \@plus 1fil \@minus \textwidth
  \rightskip\leftskip
  \parfillskip \z@skip}
\makeatother

\newcommand{\boldpi}{{\boldsymbol \pi}}

\usepackage{xparse}

\ExplSyntaxOn
\NewDocumentCommand{\makeabbrev}{mmm}
 {
  \yoruk_makeabbrev:nnn { #1 } { #2 } { #3 }
 }

\cs_new_protected:Npn \yoruk_makeabbrev:nnn #1 #2 #3
 {
  \clist_map_inline:nn { #3 }
   {
    \cs_new_protected:cpn { #2 } { #1 { ##1 } }
   }
 }
 \ExplSyntaxOff

\makeabbrev{\textbf}{tbf#1}{a,b,c,d,e,f,g,h,i,j,k,l,m,n,o,p,q,r,s,t,u,v,w,x,y,z,A,B,C,D,E,F,G,H,I,J,K,L,M,N,O,P,Q,R,S,T,U,V,W,X,Y,Z}

\makeabbrev{\textbf}{bf#1}{a,b,c,d,e,f,g,h,i,j,k,l,m,n,o,p,q,r,s,t,u,v,w,x,y,z,A,B,C,D,E,F,G,H,I,J,K,L,M,N,O,P,Q,R,S,T,U,V,W,X,Y,Z}

\makeabbrev{\textsf}{tsf#1}{a,b,c,d,e,f,g,h,i,j,k,l,m,n,o,p,q,r,s,t,u,v,w,x,y,z,A,B,C,D,E,F,G,H,I,J,K,L,M,N,O,P,Q,R,S,T,U,V,W,X,Y,Z}

\makeabbrev{\mathsf}{mss#1}{a,b,c,d,e,f,g,h,i,j,k,l,m,n,o,p,q,r,s,t,u,v,w,x,y,z,A,B,C,D,E,F,G,H,I,J,K,L,M,N,O,P,Q,R,S,T,U,V,W,X,Y,Z}

\makeabbrev{\mathfrak}{mf#1}{a,b,c,d,e,f,g,h,i,j,k,l,m,n,o,p,q,r,s,t,u,v,w,x,y,z,A,B,C,D,E,F,G,H,I,J,K,L,M,N,O,P,Q,R,S,T,U,V,W,X,Y,Z}

\makeabbrev{\mathrm}{mrm#1}{a,b,c,d,e,f,g,h,i,j,k,l,m,n,o,p,q,r,s,t,u,v,w,x,y,z,A,B,C,D,E,F,G,H,I,J,K,L,M,N,O,P,Q,R,S,T,U,V,W,X,Y,Z}

\makeabbrev{\mathbf}{mbf#1}{a,b,c,d,e,f,g,h,i,j,k,l,m,n,o,p,q,r,s,t,u,v,w,x,y,z,A,B,C,D,E,F,G,H,I,J,K,L,M,N,O,P,Q,R,S,T,U,V,W,X,Y,Z}

\makeabbrev{\mathcal}{mc#1}{A,B,C,D,E,F,G,H,I,J,K,L,M,N,O,P,Q,R,S,T,U,V,W,X,Y,Z}

\makeabbrev{\mathbb}{mbb#1}{A,B,C,D,E,F,G,H,I,J,K,L,M,N,O,P,Q,R,S,T,U,V,W,X,Y,Z}

\makeabbrev{\mathscr}{ms#1}{A,B,C,D,E,F,G,H,I,J,K,L,M,N,O,P,Q,R,S,T,U,V,W,X,Y,Z}

\makeabbrev{\mathrm}{#1}{
Id,id,ran,rk,diag,stab,ann,conv,pr,ev,tr,End,Hom,sgn,im,op,can,fin,ext,red,tot,
rot,usc,lsc,Lip,LocLip,lip,bSymLip,osc,AC,loc,uloc,spec,coz,z,ul,
supp,Opt,Adm,Cpl,Geo,GeoOpt,GeoAdm,GeoCpl,reg,
bd,co,Ric,Exp,dExp,dist,seg,Seg,cut,fcut,Cut,SDiff,Iso,Isom,diam,cl,Homeo,Diff,Der,vol,dvol,inj,relint, Graph,sub,length,eucl,
var,law,Var,Gam,pa,so,iso,fs,inv,pqi,mix,inn,
TestF,
hyp,epi,
}

\makeabbrev{\mathsf}{#1}{CD,BE,MCP,Ent,wMTW,MTW,RCD,QCD,EVI,Irr,SC,wFe,VA,UP,Curv,Alex,CAT}

\newcommand{\bLip}{\mathrm{Lip}_b}

\newcommand{\T}{\tau} 
\newcommand{\A}{\Sigma} 
\newcommand{\Bo}[1]{\msB_{#1}} 

\newcommand{\eps}{\varepsilon}

\newcommand{\defeq}{\eqqcolon}
\renewcommand{\complement}{\mathrm{c}}

\newcommand{\mathsc}[1]{\text{\textsc{#1}}}
\newcommand{\emparg}{{\,\cdot\,}}

\newcommand{\slo}[2][]{\abs{\mathrm{D}#2}_{#1}}
\newcommand{\tslo}[2][]{\tabs{\mathrm{D}#2}_{#1}}
\newcommand{\wslo}[2][]{\abs{\mathrm{D}#2}_{w,\, #1}}
\newcommand{\twslo}[2][]{\tabs{\mathrm{D}#2}_{w,\, #1}}

\newcommand{\dint}[2][]{\sideset{^{#1}\!\!\!}{_{#2}^{\scriptstyle\oplus}}\int}

\newcommand{\Ch}[1][]{\mathsf{Ch}_{#1}}
\newcommand{\CE}[1][]{\mathsf{CE}_{#1}}

\newcommand{\forallae}[1]{{\textrm{\,for ${#1}$-a.e.\,}}}
\newcommand{\as}[1]{\quad #1\text{-a.e.}}
\newcommand{\Leb}{{\mathrm{Leb}}}
\newcommand{\dLeb}{{\mathrm{dLeb}}}

\newcommand{\dom}{\mcF}
\newcommand{\domb}{\mcF_b}

\newcommand{\domain}[1]{\msD({#1})}
\newcommand{\dotloc}[1]{{{#1}^\bullet_{\loc}}}

\newcommand{\domext}{\mcF_e}
\newcommand{\domextb}{\mcF_{eb}}

\newcommand{\DzLoc}[1]{\mbbL^{#1}_{\loc}}
\newcommand{\DzLocE}[1]{\mbbL^{\!\!E\ #1}_{\loc}}
\newcommand{\DzLocBE}[1]{\mbbL^{\!\!E\ #1}_{\loc,b}}

\newcommand{\DzLocB}[1]{\mbbL^{#1}_{\loc,b}}
\newcommand{\DzLocBprime}[1]{\mbbL^{\prime\, #1}_{\loc,b}}
\newcommand{\DzB}[1]{\mbbL^{#1}_b}
\newcommand{\DzE}[1]{\mbbL^{\!\!E\ #1}}
\newcommand{\DzBE}[1]{\mbbL^{\!\!E\ #1}_b}

\newcommand{\Lipu}{\mathrm{Lip}^1}
\newcommand{\bLipu}{\mathrm{Lip}^1_b}

\newcommand{\aLoc}[1]{\mathsf{Loc}^{\times}_{#1}}
\newcommand{\Loc}[1]{\mathsf{Loc}_{#1}}

\newcommand{\Rad}[2]{\mathsf{Rad}_{#1,#2}}
\newcommand{\Poi}[2]{\mathsf{P}_{#1,#2}}
\newcommand{\Dou}[2]{\mathsf{D}_{#1,#2}}
\newcommand{\IH}[2]{\mathsf{IH}_{#1,#2}}
\newcommand{\DP}[2]{\mathsf{DP}_{#1,#2}}
\newcommand{\dRad}[2]{{#1}\textrm{-}\mathsf{Rad}_{#2}}
\newcommand{\ScL}[3]{\mathsf{ScL}_{#1,#2,#3}}
\newcommand{\cSL}[3]{\mathsf{cSL}_{#1,#2,#3}}
\newcommand{\SL}[2]{\mathsf{SL}_{#1,#2}}
\newcommand{\dcSL}[3]{{#1}\textrm{-}\mathsf{cSL}_{#2,#3}}
\newcommand{\dSL}[2]{{#1}\textrm{-}\mathsf{SL}_{#2}}

\newcommand{\sqf}[1]{\boldsymbol\Gamma_{#1}}
\newcommand{\sq}[1]{\mu_{#1}}
\newcommand{\sqE}[1]{\mu^E_{#1}}

\newcommand{\qe}[1]{\;{#1}\text{-q.e.}}

\DeclareMathOperator{\eqdef}{\coloneqq}

\let\epsilon\varepsilon

\let\temp\phi
\let\phi\varphi
\let\varphi\temp

\newcommand{\longrar}{\longrightarrow}
\newcommand{\rar}{\rightarrow}

\newcommand{\nlim}{\lim_{n}}								

\newcommand{\nliminf}{\liminf_{n }}

\newcommand{\nlimsup}{\limsup_{n }\,}

\newcommand{\diff}{\mathop{}\!\mathrm{d}}						

\newcommand{\tabs}[1]{\big\lvert#1\big\rvert}	
\newcommand{\abs}[1]{\left\lvert#1\right\rvert}						

\newcommand{\norm}[1]{\left\lVert#1\right\rVert}					
\newcommand{\set}[1]{\left\{#1\right\}}							
\newcommand{\tset}[1]{\big\{#1\big\}}							
\newcommand{\ttset}[1]{\{#1\}}									
\newcommand{\paren}[1]{\left(#1\right)}							
\newcommand{\tparen}[1]{\big({#1}\big)}
\newcommand{\braket}[1]{\left[#1\right]}							
\newcommand{\tbraket}[1]{\big[#1\big]}							
\newcommand{\tclass}[2][]{\big [#2\big]_{#1}}						

\newcommand{\Li}[2][]{\mathrm{L}_{#1}(#2)}						
\newcommand{\Lia}[2][]{\mathrm{Lip}^a_{#1}[#2]}						

\newcommand{\rep}[1]{\hat #1}									
\newcommand{\widerep}[1]{\widehat{#1}}									
\newcommand{\reptwo}[1]{\tilde{#1}}							
\newcommand{\scalar}[2]{\left\langle #1 \,\middle |\, #2\right\rangle}		
\newcommand{\hotimes}{\widehat{\otimes}}

\newcommand{\sym}[1]{{\scriptscriptstyle{(#1)}}}
\newcommand{\tym}[1]{{\scriptscriptstyle{\times #1}}}
\newcommand{\otym}[1]{{\scriptscriptstyle{\otimes #1}}}

\DeclareSymbolFont{symbolsC}{U}{pxsyc}{m}{n}
\SetSymbolFont{symbolsC}{bold}{U}{pxsyc}{bx}{n}
\DeclareFontSubstitution{U}{pxsyc}{m}{n}
\DeclareMathSymbol{\medcirc}{\mathbin}{symbolsC}{7}
\DeclareSymbolFont{symbolsZ}{OMS}{pxsy}{m}{n}
\SetSymbolFont{symbolsZ}{bold}{OMS}{pxsy}{bx}{n}
\DeclareFontSubstitution{OMS}{pxsy}{m}{n}

\DeclareMathOperator{\interior}{int}								
							
\newcommand{\seq}[1]{\paren{#1}}								
\newcommand{\tseq}[1]{{\big(#1\big)}}

\newcommand{\Cb}{\mcC_b}									
\newcommand{\Cz}{\mcC_0}									

\newcommand{\Mbp}{\mfM^+_b}
\newcommand{\Msp}{\mfM^+_\sigma}
\newcommand{\Ms}{\mfM^\pm_\sigma}
\newcommand{\Mb}{\mfM^\pm_b}
\newcommand{\MbpR}{\mfM^+_{b\mathsc{r}}}

\newcommand{\MbR}{\mfM^\pm_{b\mathsc{r}}}
\newcommand{\Ne}[1]{\msN_{#1}} 
\newcommand{\intE}{\interior_\mcE}
\newcommand{\clE}{\cl_\mcE}

\newcommand{\pfwd}{\sharp}
\DeclareMathOperator*{\esssup}{esssup}

\DeclareMathOperator{\car}{\mathds 1}

\DeclareMathOperator{\emp}{\varnothing} 
\newcommand{\N}{{\mathbb N}}
\newcommand{\R}{{\mathbb R}}

\newcommand{\restr}{\big\lvert}

\newcommand{\mrestr}[1]{_{#1}}

\allowdisplaybreaks

\usetikzlibrary{shapes.misc}
\tikzset{cross/.style={cross out, draw=black, minimum size=2*(#1-\pgflinewidth), inner sep=0pt, outer sep=0pt},
cross/.default={4pt}}

\newcommand{\comma}{\,\,\mathrm{,}\;\,}
\newcommand{\comm}{\,\,\mathrm{,}\;\,}
\newcommand{\semicolon}{\,\,\mathrm{;}\;\,}
\newcommand{\fstop}{\,\,\mathrm{.}}

\newcommand{\cdc}{\Gamma}

\newcommand{\bint}{-\!\!\!\!\!\!\int}

\usepackage{scrextend}						

\DeclareMathOperator{\inter}{int}

\newcommand{\Dz}[1]{\mbbL^{#1}}

\newcommand{\D}{\mcD} 

\newcommand{\hr}[1]{\bar\mssd_{#1}} 									

\DeclareMathOperator{\Cont}{\mcC}					

\numberwithin{equation}{section}
\theoremstyle{plain}
\newtheorem{theorem}{Theorem}[section]
\newtheorem*{theorem*}{Theorem}
\newtheorem{proposition}[theorem]{Proposition}
\newtheorem{lemma}[theorem]{Lemma}
\newtheorem{corollary}[theorem]{Corollary}

\theoremstyle{definition}
\newtheorem{definition}[theorem]{Definition}
\newtheorem*{defs*}{Definition}

\theoremstyle{remark}
\newtheorem{remark}[theorem]{Remark}
\newtheorem{example}[theorem]{Example}
\newtheorem{assumption}[theorem]{Assumption}

\renewcommand{\paragraph}[1]{\medskip\emph{#1}.\quad}

\begin{document}

\title[Persistence of Rademacher-type and Sobolev-to-Lipschitz properties
]{Persistence of Rademacher-type and\\ Sobolev-to-Lipschitz properties
}

\thanks{}

\author[L.~Dello Schiavo]{Lorenzo Dello Schiavo}
\address{Institute of Science and Technology Austria, Am Campus 1, 3400 Klosterneuburg, Austria}
\email{lorenzo.delloschiavo@ist.ac.at}
\thanks{This research was funded in whole, or in part, by the Austrian Science Fund (FWF) ESPRIT~208. For the purpose of open access, the author has applied a CC BY public copyright licence to any Author Accepted Manuscript version arising from this submission.}

\author[K.~Suzuki]{Kohei Suzuki}
\address{Department of Mathematical Science, Durham University, Science Laboratories, South Road, DH1 3LE, United Kingdom}
\email{kohei.suzuki@durham.ac.uk}
\thanks{K.S.~gratefully acknowledges funding by the Alexander von Humboldt Stiftung, Humboldt-Forschungsstipendium}



\begin{abstract}
We consider the Rademacher- and Sobolev-to-Lipschitz-type properties for arbitrary quasi-regular strongly local Dirichlet spaces. We discuss the persistence of these properties under localization, globalization, transfer to weighted spaces, tensorization, and direct integration.
As byproducts we obtain: necessary and sufficient conditions to identify a quasi-regular strongly local Dirichlet form on an extended metric topological $\sigma$-finite possibly non-Radon measure space with the Cheeger energy of the space;
the tensorization of intrinsic distances;
the tensorization of the Varadhan short-time asymptotics.
\end{abstract}

\maketitle

\setcounter{tocdepth}{3}
\makeatletter
\def\l@subsection{\@tocline{2}{0pt}{2.5pc}{5pc}{}}
\def\l@subsubsection{\@tocline{2}{0pt}{4.75pc}{5pc}{}}
\makeatother
\tableofcontents

\section{Introduction}
The interplay between analysis, geometry, and stochastic analysis on non-smooth spaces has recently gained incredibly large attention.
Extraordinarily insightful theories have been developed connecting these three aspects on:
\begin{itemize*}
\item[] Ricci-limit spaces;
\item[] sub-Riemannian manifolds;
\item[] Alexandrov and Cartan--Alexandrov--Topogonov spaces;
\item[] metric measure spaces satisfying synthetic Ricci-curvature lower bounds \`a la Lott--Sturm--Villani and Ambrosio--Gigli--Savar\'e;
\item[] Lorentzian metric measure spaces;
\item[] configuration, Wasserstein, Wiener, and other infinite-dimensional spaces;
\end{itemize*}
only to name a few.

On the one hand, every metric measure space~$(X,\mssd,\mssm)$ may be endowed with a natural convex energy functional nicely capturing the properties of the metric measure structure, namely the \emph{Cheeger energy}~$\Ch[\mssd,\mssm]$ on the space of real-valued functions on~$X$.
Starting from~$\Ch[\mssd,\mssm]$, it is possible to build a sophisticated theory of non-smooth analysis on~$(X,\mssd,\mssm)$ encompassing for instance: first- and second-order Sobolev spaces, a (non-linear) `Laplacian', a (non-linear) `heat flow', etc; see e.g.~\cite{AmbGigSav14,Gig18}.
It was a particularly fruitful intuition ---in great generality put forward by N.~Gigli--- that the validity of the parallelogram identity for~$\Ch[\mssd,\mssm]$ ---equivalently, the linearity of the heat flow---  provides a setting most suitable to the study of the interplay mentioned above.

On the other hand however, other standard ---this time naturally \emph{quadratic}--- energy forms appear in many settings and are in principle unrelated to the metric-measure structure of the underlying space.
This includes, in particular, Dirichlet energy forms on configuration spaces~\cite{AlbKonRoe98,RoeSch99,ErbHue15,LzDSSuz21,LzDSSuz22a}, spaces of measures~\cite{LzDS17+,ForSavSod22,KonLytVer15}, Wiener spaces~\cite{AidKaw01,AidZha02,HinRam03}, and others.

Linking ---and possibly reconciling--- these two aspects of the theory of non-smooth spaces provides great insight, allowing us to import tools from stochastic analysis into metric measure geometry, and vice versa.
To this end, in~\cite{LzDSSuz20} we introduced several Rademacher- and Sobolev-to-Lipschitz-type properties comparing the domain of the energy form under consideration with the space of Lipschitz functions induced by an assigned distance.
We used this comparison to give sufficient conditions for the validity of the integral-type Varadhan short-time asymptotics of the heat semigroup with respect to an assigned distance.

Here, we address the persistence of these Rademacher- and Sobolev-to-Lipschitz-type properties under various constructions/transformations on arbitrary quasi-regular strongly local Dirichlet spaces, including:
\begin{itemize}
\item \emph{localization} to form restrictions (in the sense of e.g.~\cite[\S3]{Kuw98});
\item \emph{globalization} from form restrictions on coverings of the space;
\item transfer to \emph{weighted spaces} (more general than Girsanov transforms);
\item \emph{tensorization} to product spaces;
\item \emph{direct integration} (in the sense of~\cite{LzDS20, LzDSWir21}). 
\end{itemize}

As anticipated, our first goal is to compare notions in Dirichlet-form theory (quasi-regular strongly local forms, energy measures, square field operators) with notions in metric measure geometry (Cheeger energies, minimal weak upper gradients, minimal relaxed slopes).
Whenever possible, we do so in great generality, on extended metric topological $\sigma$-finite non-Radon measure spaces; in particular, away from any assumption of geometric type (e.g., measure doubling, Poincar\'e inequalities, synthetic curvature bounds, etc.).

\paragraph{Applications to particle systems} Apart from the theoretical motivations explained above, the persistence of the aforementioned properties is inspired also by geometric and analytic constructions of infinite particle systems. 
Each of the above transformations of Dirichlet forms plays a significant role for the construction of interacting particle systems of diffusions from a single diffusion process on the base space.
Starting from a one-particle diffusion associated with the Dirichlet energy on the base space, the correspondence between each operation on Dirichlet spaces and diffusion processes is the following:
\begin{itemize} 
\item \emph{localization} $\rightsquigarrow$ \emph{killing} of a one-particle diffusion upon exiting given set;
\item \emph{globalization} $\rightsquigarrow$ \emph{patching} a one-particle diffusions; 
\item \emph{tensorization} $\rightsquigarrow$ \emph{independent copy} of a one-particle diffusion;
\item transfer to \emph{weighted spaces} $\rightsquigarrow$ \emph{interacting} many-particle system;
\item \emph{direct integration} $\rightsquigarrow$  {\it superposition} of finite-particle systems to the space of infinite particles.
\end{itemize}
The study of the persistence under these transformations leads us to lift metric measure properties of the base space to the space of infinite particles such as the configuration space or the space of atomic probability measures, see e.g.,~\cite{LzDSSuz21,LzDSSuz22a} for the application to the configuration space and~\cite{LzDS17+} for the Wasserstein space.

\subsection{The Rademacher and Sobolev-to-Lipschitz properties}
Let~$(\mcE,\dom)$ be a quasi-regular strongly local Dirichlet form on the space~$L^2(\mssm)$ of a topological Luzin space~$(X,\T)$ endowed with a Radon measure~$\mssm$.
Given a $\sigma$-finite Borel measure~$\mu$ on~$X$ (possibly different from~$\mssm$), we denote by~$\DzLocB{\mu}$ the space of bounded functions in the \emph{broad local domain} (see~\S\ref{ss:BroadLoc}) of~$(\mcE,\dom)$ with \emph{$\mu$-uniformly bounded $\mcE$-energy} (see~\S\ref{sss:LocDom}).
In particular ---in order to exemplify this concept--- when~$(\mcE,\dom)$ admits carr\'e du champ operator~$\cdc$, and~$\mu= g\mssm$ is absolutely continuous w.r.t.~$\mssm$, then~$\DzLocB{\mu}$ is the space of all functions~$f$ in the broad local domain of~$(\mcE,\dom)$ additionally satisfying~$\cdc(f)\leq g$ $\mssm$-a.e..

For~$\mu$ as above, we introduced in~\cite{LzDSSuz20} the \emph{intrinsic distance}~$\mssd_\mu$ of~$(\mcE,\dom)$ induced by~$\mu$, see~\eqref{eq:IntrinsicD}.
It is an extended pseudo-distance, that is, possibly taking the value $+\infty$ and/or vanishing outside the diagonal in~$X^\tym{2}$.
When~$\mu=\mssm$ and~$(\mcE,\dom)$ is a regular Dirichlet form, $\mssd_\mssm$ coincides with the standard intrinsic distance of~$(\mcE,\dom)$.

Now, let~$\mssd\colon X^\tym{2}\to [0,\infty]$ be any extended pseudo-distance.
For the sake of notational simplicity, throughout this introduction we denote by~$\Lipu(\mssd)$ the space of \emph{measurable} $\mssd$-Lipschitz functions on~$X$ with Lipschitz constant less than~$1$, postponing to later sections a discussion of measurability issues and the choice(s) of a $\sigma$-algebra.
Finally, for~$f\in L^0(\mssm)$, denote by~$\rep f$ and~$\reptwo f$ any of its (everywhere defined) $\mssm$-representatives.

After~\cite{LzDSSuz20} (also cf.~Dfn.~\ref{d:RadStoL} below), we say that~$(X,\mcE,\mssd,\mu)$ has:
\begin{itemize}
\item[($\Rad{\mssd}{\mu}$)] the \emph{Rademacher property} if, whenever~$\rep f\in \Lipu(\mssd)$, then~$f\in \DzLoc{\mu}$;

\item[($\dRad{\mssd}{\mu}$)] the \emph{distance-Rademacher property} if~$\mssd\leq \mssd_\mu$;

\item[($\ScL{\mssm}{\T}{\mssd}$)] the \emph{Sobolev--to--continuous-Lipschitz property} if each~$f\in\DzLoc{\mu}$ has a $\T$-continuous $\mssm$-representative $\rep f\in\Lip^1(\mssd)$;

\item[($\SL{\mu}{\mssd}$)] the \emph{Sobolev--to--Lipschitz property} if each~$f\in\DzLoc{\mu}$ has an $\mssm$-representat\-ive $\rep f\in\Lip^1(\mssd)$;

\item[($\dcSL{\mssd}{\mu}{\mssd}$)] the \emph{$\mssd$-continuous-Sobolev--to--Lipschitz property} if each~$f\in \DzLoc{\mu}$ having a $\mssd$-continuous (measurable) representative~$\rep f$ also has a representative $\reptwo f\in \Lip^1(\mssd)$ (possibly,~$\reptwo f\neq \rep f$);

\item[($\cSL{\T}{\mssm}{\mssd}$)] the \emph{continuous-Sobolev--to--Lipschitz property} if each $\T$-continuous~$f\in \DzLoc{\mu}$ satisfies $f\in\Lip^1(\mssd)$;

\item[($\dSL{\mssd}{\mu}$)] the \emph{distance Sobolev-to-Lipschitz property} if~$\mssd\geq \mssd_\mu$.
\end{itemize}

For implications between the above properties in the present setting see~\eqref{eq:EquivalenceRadStoL}.

\subsection{Main results}
Let us now summarize our main results.

\paragraph{Localization/globalization}
We prove that the Rademacher property is stable under localization (Prop.~\ref{p:LocRad}), while the Sobolev--to--Lipschitz property is stable under globalization (in the sense of sheaves, see Prop.~\ref{p:GlobCSL}) but not under localization (Ex.~\ref{ese:FailureLocSL}).
Let us stress that these stability results are trivial (or trivially false) if one replaces the local space~$\DzLocB{\mssm}$ with its global counterpart~$\DzB{\mssm}$ (see~\S\ref{sss:LocDom}), which is a first indication---see below for stronger ones---that the Rademacher and Sobolev-to-Lipschitz properties ought to be phrased in terms of \emph{local} spaces (as opposed to: subspaces of the domain).

\paragraph{Weighted spaces} Both the Rademacher and the Sobolev-to-Lipschitz property for~$(\mcE,\dom)$ on~$L^2(\mssm)$ transfer to weighted space~$L^2(\theta\mssm)$ for any density~$\theta$ bounded away from $0$ and infinity locally in the sense of quasi-open nests.
For simplicity, we state here a single result under stronger assumptions than necessary, combining the transfer of the Rademacher property with the transfer of the Sobolev-to-Lipschitz property (for minimal assumptions see Prop.s~\ref{p:IneqCdC} and~\ref{p:IneqCdC2} respectively).

\begin{theorem*}[Cor.~\ref{c:LocalityDistances}]
Let~$(\mbbX,\mcE)$ and~$(\mbbX',\mcE')$ be quasi-regular strongly local Dirichlet spaces with same underlying topological measurable space~$(X,\T,\A)=(X',\T',\A')$ and possibly different measures~$\mssm$ and~$\mssm'$.
Further let~$\mssd\colon X^\tym{2}\to [0,\infty]$ be an extended pseudo-distance.
Assume that:
\begin{enumerate}[$(a)$]
\item there exists a linear subspace~$\mcD$ both~$\dom$-dense in~$\dom$ and~$\dom'$-dense in~$\dom'$, additionally so that~$\cdc=\cdc'$ on~$\mcD$.

\item $\mssm'=\theta\mssm$ for some~$\theta\in L^0(\mssm)$, and there exists $E_\bullet, G_\bullet$ quasi-open nests for both~$\mcE$ and~$\mcE'$ with the following properties:
\begin{enumerate}[$({b}_1)$]
\item for each~$k\in \N$ there exists a constant~$a_k>0$ such that
\[
0<a_k\leq \theta \leq a_k^{-1}<\infty  \as{\mssm} \quad \text{on } G_k\semicolon
\]
(that is,~$\theta, \theta^{-1}\in \dotloc{\tparen{L^\infty(\mssm)}}(G_\bullet)$.)
\item for each~$k\in \N$ it holds~$E_k\subset G_k$ $\mcE$- and $\mcE'$-quasi-everywhere, and there exists~$\varrho_k\in \mcD$ with~$\car_{E_k}\leq \varrho_k\leq \car_{G_k}$ $\mssm$-a.e..
\end{enumerate}
\end{enumerate}
Then,
\begin{enumerate}[$(i)$]
\item $\tparen{\cdc, \DzLocB{\mssm}}= \tparen{\cdc',\DzLocBprime{\mssm'}}$ and $\mssd_\mssm= \mssd_{\mssm'}$;
\item $(\ScL{\mssm}{\T}{\mssd})$, resp.\ $(\SL{\mssm}{\mssd})$, $(\cSL{\T}{\mssm}{\mssd})$, $(\Rad{\mssd}{\mssm})$, holds if and only if $(\ScL{\mssm'}{\T}{\mssd})$, resp.\ $(\SL{\mssm'}{\mssd})$, $(\cSL{\T}{\mssm'}{\mssd})$, $(\Rad{\mssd}{\mssm'})$ holds.
\end{enumerate}
\end{theorem*}

The above theorem provides us with the following general guidelines:
\begin{itemize}
\item the broad local space of functions with uniformly bounded energy, the intrinsic distance, and various Rademacher and Sobolev-to-Lipschitz-type properties are all \emph{completely determined by a dense subspace and in a local fashion}.
\item the notion of broad local space introduced by K.~Kuwae in~\cite{Kuw98} is the right one to address the interplay between the Dirichlet-space and metric measure space structures.
\item the intrinsic distance, and therefore ---under the assumption of both the Rademacher and the Sobolev-to-Lipschitz property--- the Varadhan-type short-time asymptotics for the heat-semigroup/kernel, both transfer to weighted spaces in far greater generality than Girsanov transforms.
\end{itemize}

When considering applications to metric measure geometry it is natural to choose the space~$\mcD$ in the above Theorem to be some algebra of Lipschitz functions.
In this case, as a consequence of the above result, the Rademacher, Sobolev-to-Lipschitz-type and related properties may be shown in the simplified setting of probability spaces and subsequently transferred to general $\sigma$-finite spaces.
This applies in particular to the comparison of a Dirichlet space~$(\mcE,\dom)$ with the Cheeger energy~$\Ch[\mssd,\mssm]$ of the underlying extended metric measure space, for which we prove two comparison results (Prop.~\ref{p:IneqCdC} and~\ref{p:IneqCdC2}) separately showing the inequalities~$\mcE\leq \Ch[\mssd,\mssm]$ and~$\mcE\geq \Ch[\mssd,\mssm]$ under minimal assumptions.
Combining the two, we further have---in the general case of $\sigma$-finite extended metric spaces---the following identification of $\mcE$ with~$\Ch[\mssd,\mssm]$, previously shown by L.~Ambrosio, N.~Gigli, and G.~Savar\'e~\cite{AmbGigSav15} for energy-measure spaces, and by L.~Ambrosio, M.~Erbar, and G.~Savar\'e~\cite{AmbErbSav16} for extended metric-topological \emph{probability} spaces.
Namely we prove:
\begin{theorem*}[Cor.~\ref{c:RadStoLCheegerComparison}]
Let~$(\mbbX,\mcE)$ be a quasi-regular strongly local Dirichlet space admitting carr\'e du champ operator~$\cdc$, and~$\mssd\colon X^\tym{2}\to[0,\infty]$ be an extended distance.
Further assume that~$(\Rad{\mssd}{\mssm})$ and~$(\cSL{\T}{\mssm}{\mssd})$ hold.
Then,~$\mcE\leq \Ch[\mssd,\mssm]$.
The equality~$\mcE=\Ch[\mssd,\mssm]$ holds if and only if $(\mbbX,\mcE)$ is additionally $\T$-upper regular (see Dfn.~\ref{d:TUpperReg}).
\end{theorem*}

See Remark~\ref{r:ComparisonAGS-AES} below for a thorough comparison of our results with those in~\cite{AmbGigSav15,AmbErbSav16}.
Let us further stress that the above Theorem allows us to \emph{deduce} the parallelogram identity for the Cheeger energy, from that for the form~$\mcE$, thus providing necessary and sufficient conditions to implement the \emph{Riemannian} point of view in the sense of Gigli's.

\paragraph{Tensorization}
In the case when~$\mssd$ is a distance (\emph{not} extended), we also discuss the tensorization of the Rademacher and Sobolev-to-Lipschitz properties.

When~$\mcE=\Ch[\mssd,\mssm]$ is the Cheeger energy of a metric measure space, the tensorization of the Rademacher property is a byproduct of the tensorization of the Cheeger energy, discussed under different geometric assumptions in~\cite{AmbGigSav14b,AmbPinSpe15} and recently settled by S.~Eriksson-Bique, T.~Rajala, and E.~Soultanis for infinitesimally (quasi-)Hilbertian metric measure spaces in~\cite{EriRajSou22}; see \S\ref{sss:TensorizationConseq} below for a detailed account.
Here, we discuss the case of general quasi-regular strongly local Dirichlet spaces~$(\mbbX,\mcE)$, without any geometric assumption.

As it turns out, a product space inherits the Rademacher property from its factors.

\begin{theorem*}[Thm.~\ref{t:TensorRad}]
Let~$\mbbX\eqdef (X,\mssd,\mssm)$ be a metric measure space (Dfn.~\ref{d:MMSp}), and~$(\mcE,\dom)$ be a quasi-regular strongly local Dirichlet form on~$\mbbX$ admitting carr\'e du champ operator~$\cdc$ and satisfying~$(\Rad{\mssd}{\mssm})$.
Further let~$(\mbbX',\mcE')$ be satisfying the same assumptions as~$(\mbbX,\mcE)$.

Then, their product space~$(\mbbX^\otym{},\mcE^\otym{})$ satisfies~$(\Rad{\mssd^\otym{}}{\mssm^\otym{}})$.
\end{theorem*}

As for the Sobolev-to-Lipschitz property, we show that a product space satisfies the continuous-Sobolev-to-Lipschitz property if the factors satisfy the Sobolev-to-Lipschitz property.

\begin{theorem*}[Thm.~\ref{t:TensorSL}]
Let~$\mbbX\eqdef (X,\mssd,\mssm)$ be a metric measure space (Dfn.~\ref{d:MMSp}), and~$(\mcE,\dom)$ be a quasi-regular strongly local Dirichlet form on~$\mbbX$ satisfying~$(\SL{\mssm}{\mssd})$.
Further let~$(\mbbX',\mcE')$ be satisfying the same assumptions as~$(\mbbX,\mcE)$.
Then, their product space~$(\mbbX^\otym{},\mcE^\otym{})$ satisfies~$(\cSL{\T^\otym{}}{\mssm^\otym{}}{\mssd^\otym{}})$.
\end{theorem*}

\paragraph{Notation}
For a measure~$\mu$ on a measurable space~$(X,\A)$ we denote by~$\mu A$, resp.~$\mu f$, the $\mu$-measure of $A\in\A$, resp.\ the integral with respect to~$\mu$ of a $\A$-measurable function~$f$ (whenever the integral makes sense).

\section{Setting}\label{s:Preliminaries}
\subsection{Metric and topological spaces}
Let~$X$ be any non-empty set. A function~$\mssd\colon X^\tym{2}\rar [0,\infty]$ is an \emph{extended pseudo-distance} if it is symmetric and satisfying the triangle inequality. Any such~$\mssd$ is: a \emph{pseudo-distance} if it is everywhere finite, i.e.~$\mssd\colon X^{\tym{2}}\rar [0,\infty)$; an \emph{extended distance} if it does not vanish outside the diagonal in~$X^{\tym{2}}$, i.e.~$\mssd(x,y)=0$ iff~$x=y$; a \emph{distance} if it is both finite and non-vanishing outside the diagonal.

Let~$x_0\in X$ and~$r\in (0,\infty]$. We write~$B^\mssd_r(x_0)\eqdef \set{\mssd_{x_0}<r}$.
We call~$B^\mssd_\infty(x_0)$ the \emph{$\mssd$-accessible component} of~$x_0$ in~$X$.
Note that, if~$\mssd$ is an extended pseudo-metric, then both of the inclusions~$\set{x_0}\subset \cap_{r>0} B^\mssd_r(x_0)$ and~$B^\mssd_\infty(x_0)\subset X$ may be strict ones. 
We say that an extended metric space is \emph{complete} if~$B^\mssd_\infty(x)$ is complete for each~$x\in X$.
Finally set
\begin{align*}
\mssd(\emparg, A)\eqdef& \inf_{x\in A} \mssd(\emparg,x) \colon X\longrar [0,\infty] \comm \qquad A\subset X\fstop
\end{align*}

\begin{lemma}\label{l:SupDistanceShrinking}
Let~$\mssd\colon X^\tym{2}\to [0,\infty]$ be an extended pseudo-distance.
Furter fix~$x\in X$ and let~$\msA_x$ be any family of subsets with~$x\in\cap\msA_x$ and~$\inf_{A\in\msA_x}\diam_\mssd A=0$.
Then,
\[
\sup_{A\in\msA_x} \mssd(\emparg, A)= \mssd(\emparg,x) \fstop
\]
\begin{proof}
The inequality~`$\leq$' is always satisfied by definition of point-to-set extended pseudo-distance, thus it suffices to show the converse inequality.
Up to restricting to any up-to-countable sub-family of~$\msA_x$, it suffices to show the assertion in the case when~$\msA_x\eqdef\seq{A_n}_n$ is up-to-countable.

Now, fix~$y\in X$ and~$\eps>0$.
For each~$n\in\N_1$ let~$x_n$ be so that~$\mssd(y,A_n)\geq \mssd(y,x_n)+\eps$.
Then,
\begin{align*}
\sup_n \mssd(y,A_n) \geq&\ \eps+\sup_n \mssd(y,x_n) \geq \eps+\sup_n \tparen{\mssd(y,x)-\mssd(x,x_n)} \geq \eps+ \mssd(y,x) - \inf_n\diam_\mssd(A_n)
\\
=&\ \eps + \mssd(y,x) \fstop
\end{align*}
By arbitrariness of~$\eps>0$ we conclude that~$\sup_n\mssd(y,A_n)=\mssd(y,x)$, and the conclusion follows by arbitrariness of~$y$.
\end{proof}
\end{lemma}

For an extended pseudo-distance~$\mssd$ on~$X$, let~$\T_\mssd$ denote the (possibly \emph{not} Hausdorff) topology on~$X$ induced by the pseudo-distance~$\mssd\wedge 1$.
The topology~$\T_\mssd$ is Hausdorff if and only if~$\mssd$ is an extended distance.
The topology~$\T_\mssd$ is separable if and only if there exists a countable family of points~$\seq{x_n}_n\subset X$ so that~$X=\cup_n B^\mssd_\infty(x_n)$ and~$(B^\mssd_\infty(x_n),\mssd)$ is a separable pseudo-metric space for every~$n\in \N$.

\paragraph{Lipschitz functions} A function~$\rep f\colon X\rar \R$ is $\mssd$-Lipschitz if there exists a constant~$L>0$ so that
\begin{align}\label{eq:Lipschitz}
\tabs{\rep f(x)-\rep f(y)}\leq L\, \mssd(x,y) \comm \qquad x,y\in X \fstop
\end{align}
The smallest constant~$L$ so that~\eqref{eq:Lipschitz} holds is the (global) \emph{Lipschitz constant of $\rep f$}, denoted by~$\Li[\mssd]{\rep f}$.
Further let the \emph{slope} of~$\rep f$ at~$x$ be defined as
\[
\tslo[\mssd]{\rep f}(x)\eqdef \limsup_{\mssd(x,y)\to 0} \frac{\abs{f(x)-f(y)}}{\mssd(x,y)}\fstop
\]
Conventionally,~$\tslo[\mssd]{\rep f}(x)=0$ whenever~$x$ is isolated relative to~$\mssd$.
We omit the specification of~$\mssd$ whenever apparent from context.

\begin{remark}\label{r:McShane}
It is worth stressing that ---\emph{conventionally}--- in~\eqref{eq:Lipschitz} we set~$0\cdot\infty\eqdef\infty$.
We further note that, for $\rep f\colon X\to \R$, having~$\Li[\mssd]{\rep f}=0$ does not imply that~$\rep f$ is constant, unless $\mssd$ were in fact a distance (\emph{not}: extended).
Indeed, if~$\mssd$ has accessible components~$\seq{X_i}_{i\in I}$, then every function~$\rep f$ constant on each~$X_i$ is $\mssd$-Lipschitz with~$\Li[\mssd]{\rep f}=0$.
\end{remark}

For any non-empty~$A\subset X$ we write~$\Lip(A,\mssd)$, resp.~$\bLip(A,\mssd)$ for the family of all finite, resp.\ bounded, $\mssd$-Lipschitz functions on~$A$. For simplicity of notation, further let~$\Lip(\mssd)\eqdef \Lip(X,\mssd)$, resp.\ $\bLip(\mssd)\eqdef \bLip(X,\mssd)$.

\paragraph{Topological spaces}
A Hausdorff topological space~$(X,\T)$ is:
\begin{enumerate}[$(a)$]
\item\label{i:Top:2} a \emph{topological Luzin space} if it is a continuous injective image of a Polish space;
\item\label{i:Top:3} a \emph{metrizable Luzin space} if it is homeomorphic to a Borel subset of a compact metric space.
\end{enumerate}

\smallskip

Let~$(X,\T)$ be a Hausdorff topological space. A family of pseudo-distances~$\UP$ is a \emph{uniformity} (\emph{of pseudo-distances}) if:
\begin{enumerate*}[$(a)$]
\item it is directed, i.e., $\mssd_1\vee \mssd_2\in\UP$ for every~$\mssd_1,\mssd_2\in \UP$;
and
\item it is order-closed, i.e., $\mssd_2\in \UP$ and~$\mssd_1\leq \mssd_2$ implies~$\mssd_1\in \UP$ for every pseudo-distance~$\mssd_1$ on~$X$.
\end{enumerate*}
A uniformity is \emph{Hausdorff} if it separates points.

\paragraph{Extended metric-topological space} The next definition is a reformulation of~\cite[Dfn.~4.1]{AmbErbSav16}.

\begin{definition}[Extended metric-topological space]\label{d:AES}
Let~$(X,\T)$ be a Hausdorff topological space. An extended pseudo-distance~$\mssd\colon X^{\tym{2}}\rar [0,\infty]$ is $\T$-\emph{admissible} if there exists a uniformity~$\UP$ of \emph{$\T^\tym{2}$-continuous} pseudo-distances~$\mssd'\colon X^\tym{2}\rar [0,\infty)$, so that
\begin{align}\label{eq:d=supUP}
\mssd=\sup\set{\mssd':\mssd'\in\UP}\fstop
\end{align}
The triple $(X,\T,\mssd)$ is an \emph{extended metric-topological space} if~$\mssd$ is $\T$-admissible, and there exists a uniformity~$\UP$ witnessing the $\T$-admissibility of~$\mssd$, and additionally Hausdorff and generating~$\T$.
\end{definition}

Let~$\mssd\colon X^\tym{2}\to[0,\infty]$ be an extended pseudo-distance on~$X$.
We denote by~$\T_\mssd$ the topology induced by~$\mssd$ and we note that, even in the case when~$\mssd$ is $\T$-admissible, $\T_\mssd$ is in general strictly finer than~$\T$.
If~$\mssd$ is a distance however, then~$\T_\mssd=\T$.

\subsection{Measure spaces}\label{ss:MeasureTopSp}
Let~$(X,\T)$ be a Hausdorff topological space. We denote by~$\Bo{\T}$ the Borel $\sigma$-algebra of~$(X,\T)$. Given a Borel measure~$\mu$ on~$(X,\Bo{\T})$, we denote by~$\Bo{\T}^\mu$ the (Carath\'edory) completion of~$\Bo{\T}$ with respect to~$\mu$.
Given $\sigma$-finite measures~$\mu_0$,~$\mu_1$ on~$(X,\Bo{\T})$, we write~$\mu_0\leq \mu_1$ to indicate that~$\mu_0 A\leq \mu_1 A$ for every~$A\in\Bo{\T}$. 
Every Borel measure on a strongly Lindel\"of space has support, e.g.~\cite[p.~148]{MaRoe92}.

Let~$\A$ be a $\sigma$-algebra over~$X$.
We denote by~$\mcL^0(\A)$, resp.~$\mcL^\infty(\A)$, the vector space of all everywhere-defined real-valued, resp.~uniformly bounded, ($\A$-)measurable functions on~$X$.
When~$\mu$ is a measure on~$(X,\A)$, we denote  by~$L^0(\mu)$ the corresponding vector space of $\mu$-classes.
We denote by~$\mcL^2(\mu)$ the vector space of all $\mu$-square-integrable functions in~$\mcL^0(\A)$, by~$L^2(\mu)$ the corresponding space of~$\mu$-classes.
Let the corresponding definition of~$\mcL^p(\mu)$, resp.~$L^p(\mu)$, be given for all~$p\in [1,\infty)$.

As a general rule, we denote measurable functions by either~$\rep f$ or~$\reptwo f$, and classes of measurable functions up to a.e.\ equality simply by~$f$.
When~$\mu$ has full support on~$X$, we drop this distinction for $\T$-continuous functions, simply writing~$f$ for both the class and its unique $\T$-continuous representative.

\paragraph{Measurability and continuity of Lipschitz functions} Let~$(X,\T)$ be a Hausdorff space, $\mssd\colon X^{\tym{2}}\rar [0,\infty]$ be an extended distance on~$X$.  Let~$\rep f\colon X \rar [-\infty,\infty]$ be $\mssd$-Lipschitz with~$\rep f\not\equiv \pm\infty$. In general, $\rep f$ is \emph{neither} everywhere finite, nor $\T$-continuous, nor $\Bo{\T}^\mssm$-measurable, see \cite{LzDSSuz20}. For a given $\sigma$-algebra~$\A$ on~$X$, this motivates to set
\begin{align*}
\Lip(\mssd,\A)\eqdef \Lip(\mssd)\ \cap&\ \mcL^0(\A)\comm & \bLip(\mssd,\A)\eqdef \bLip(\mssd)\ \cap&\ \mcL^0(\A)\fstop
\\
\Lip(\mssd,\T)\eqdef \Lip(\mssd)\ \cap&\ \mcC(\T)\comm & \bLip(\mssd,\T)\eqdef \bLip(\mssd)\ \cap&\ \mcC(\T)\fstop
\end{align*}

\paragraph{Main assumptions} Everywhere in the following, $\mbbX$ is a quadruple~$(X,\T,\A,\mssm)$ so that~$\Bo{\T}\subset \A\subset \Bo{\T}^\mssm$, the reference measure~$\mssm$ is positive $\sigma$-finite on~$(X,\A)$, and one of the following holds:
\begin{enumerate}[$(\mathsc{sp}_1)$]
\item\label{ass:Hausdorff}
$(X,\T)$ is a Hausdorff space;
\item\label{ass:Luzin}
$(X,\T)$ is a topological Luzin space and~$\supp[\mssm]=X$;
\item\label{ass:Polish}
$(X,\T)$ is a second countable locally compact Hausdorff space,~$\mssm$ is Radon and $\supp[\mssm]=X$.
\end{enumerate}

\subsection{Dirichlet spaces}
Given a bilinear form~$(Q,\domain{Q})$ on a Hilbert space~$H$, we write
\begin{align*}
Q(h)\eqdef Q(h,h)\comm \qquad Q_\alpha(h_0,h_1)\eqdef Q(h_0,h_1)+\alpha\scalar{h_0}{h_1}\comm \alpha>0\fstop
\end{align*}

Let~$\mbbX$ be satisfying Assumption~\ref{ass:Hausdorff}. A \emph{Dirichlet form on~$L^2(\mssm)$} is a non-negative definite densely defined closed symmetric bilinear form~$(\mcE,\dom)$ on~$L^2(\mssm)$ satisfying the Markov property
\begin{align*}
f_0\eqdef 0\vee f \wedge 1\in \dom \qquad \text{and} \qquad \mcE(f_0)\leq \mcE(f)\comm \qquad f\in\dom\fstop
\end{align*}

If not otherwise stated,~$\dom$ is always regarded as a Hilbert space with norm~$\norm{\emparg}_\dom\eqdef\mcE_1(\emparg)^{1/2}=\sqrt{\mcE(\emparg)+\norm{\emparg}_{L^2(\mssm)}^2}$.
A \emph{Dirichlet space} is a pair~$(\mbbX,\mcE)$, where~$\mbbX$ satisfies~\ref{ass:Hausdorff} and~$(\mcE,\dom)$ is a Dirichlet form on~$L^2(\mssm)$.
A \emph{pseudo-core} is any $\dom$-dense linear subspace of~$\dom$.

\subsubsection{Quasi-notions} For any~$A\in\Bo{\T}$ set~$\dom_A\eqdef \set{u\in \dom: u= 0 \text{~$\mssm$-a.e.~on~} X\setminus A}$.
A sequence $\seq{A_n}_n\subset \Bo{\T}$ is a \emph{Borel $\mcE$-nest} if $\cup_n \dom_{A_n}$ is dense in~$\dom$.
For any~$A\in\Bo{\T}$, let~$(p)$ be a proposition defined with respect to~$A$. We say that~`$(p_A)$ holds' if~$A$ satisfies~$(p)$.
A \emph{$(p)$-$\mcE$-nest} is a Borel nest~$\seq{A_n}$ so that~$(p_{A_n})$ holds for every~$n$. In particular, a \emph{closed $\mcE$-nest}, henceforth simply referred to as an \emph{$\mcE$-nest}, is a Borel $\mcE$-nest consisting of closed sets.

A set~$N\subset X$ is \emph{$\mcE$-polar} if there exists an $\mcE$-nest~$\seq{F_n}_n$ so that~$N\subset X\setminus \cup_n F_n$.
A set~$G\subset X$ is \emph{$\mcE$-quasi-open} if there exists an $\mcE$-nest~$\seq{F_n}_n$ so that~$G\cap F_n$ is relatively open in~$F_n$ for every~$n\in \N$.
A set~$F$ is \emph{$\mcE$-quasi-closed} if~$X\setminus F$ is $\mcE$-quasi-open.
Without loss of generality, and without explicit mention, we will assume that $\mcE$-quasi-open/closed, sets are additionally Borel measurable, see~\cite[Lem.~2.6]{LzDSSuz20}.
Any countable union or finite intersection of $\mcE$-quasi-open sets is $\mcE$-quasi-open; analogously, any countable intersection or finite union of $\mcE$-quasi-closed sets is $\mcE$-quasi-closed; see~\cite[Lem.~2.3]{Fug71}.

A property~$(p_x)$ depending on~$x\in X$ holds $\mcE$-\emph{quasi-everywhere} (in short:~$\mcE$-q.e.) if there exists an $\mcE$-polar set~$N$ so that~$(p_x)$ holds for every~$x\in X\setminus N$.
Given sets~$A_0,A_1\subset X$, we write~$A_0\subset A_1$ $\mcE$-q.e.\ if~$\car_{A_0}\leq \car_{A_1}$ $\mcE$-q.e. Let the analogous definition of~$A_0=A_1$ $\mcE$-q.e.\ be given.

A function~$\rep f\in \mcL^0(\A)$ is \emph{$\mcE$-quasi-continuous} if there exists an $\mcE$-nest~$\seq{F_n}_n$ so that~$\rep f\restr_{F_n}$ is continuous for every~$n\in \N$.
Equivalently,~$\reptwo f$ is $\mcE$-quasi-continuous if and only if it is $\mcE$-q.e.\ finite and $\reptwo f^{-1}(U)$ is $\mcE$-quasi-open for every open~$U\subset \R$, see e.g.~\cite[p.~70]{FukOshTak11}.
Whenever~$f\in L^0(\mssm)$ has an $\mcE$-quasi-continuous $\mssm$-version, we denote it by~$\reptwo f\in \mcL^0(\A)$.

\paragraph{Spaces of measures} We write~$\Mbp(\A)$, resp.~$\Msp(\A)$, $\Mb(\A)$, $\Ms(\A)$, for the space of finite, resp.\ $\sigma$-finite, finite signed, extended $\sigma$-finite signed, measures on~$(X,\A)$. A further subscript~`$\mathsc{r}$' indicates (sub-)spaces of Radon measures, e.g.~$\MbpR(\A)$.
We write~$\Ms(\A,\Ne{\mcE})$ for the space of extended $\sigma$-finite signed measures not charging sets in the family~$\Ne{\mcE}$ of $\mcE$-polar Borel subsets of~$X$.

\subsubsection{General properties} 
Let~$\mbbX$ be a topological measure space as in~\S\ref{ss:MeasureTopSp}.
When~$(\mcE,\dom)$ is a Dirichlet form on~$L^2(\mssm)$, we say that~$(\mbbX,\mcE)$ is a \emph{Dirichlet space}.

A Dirichlet space~$(\mbbX,\mcE)$ is \emph{quasi-regular} if each of the following holds:
\begin{enumerate}[$({\mathsc{qr}}_1)$]
\item\label{i:QR:1} there exists an $\mcE$-nest~$\seq{F_n}_n$ consisting of $\T$-compact sets;

\item\label{i:QR:2} there exists a dense subset of~$\dom$ the elements of which all have $\mcE$-quasi-continuous $\mssm$-versions;

\item\label{i:QR:3} there exists an $\mcE$-polar set~$N$ and a countable family~$\seq{u_n}_n$ of functions~$u_n\in\dom$ having $\mcE$-quasi-continuous versions~$\reptwo u_n$ so that~$\seq{\reptwo u_n}_n$ separates points in~$X\setminus N$.
\end{enumerate}

Let~$(\mbbX,\mcE)$ be a quasi-regular Dirichlet space,~$\seq{F_n}_n$ be an $\mcE$-nest witnessing its quasi-regularity, and set~$X_0\eqdef \cup_n F_n$, endowed with the trace topology~$\T_0$,~$\sigma$-algebra~$\A_0$, and the restriction~$\mssm_0$ of~$\mssm$ to~$\A_0$.
Then,~$\mbbX_0$ satisfies~\ref{ass:Luzin}, and the space~$L^p(\mssm)$ may be canonically identified with the space~$L^p(\mssm_0)$,~$p\in[0,\infty]$, since~$X\setminus X_0$ is $\mcE$-polar, hence~$\mssm$-negligible.
By letting~$\mcE_0$ denote the image of~$\mcE$ under this identification,~$(\mbbX_0,\mcE_0)$ is a quasi-regular Dirichlet space, and~$\dom_0$ is canonically linearly isometrically isomorphic to~$\dom$. See~\cite[Rmk.~IV.3.2(iii)]{MaRoe92} for the details of this construction.

\begin{remark}
When considering a quasi-regular Dirichlet space~$(\mbbX,\mcE)$, we may and shall therefore assume, with no loss of generality, that~$\mbbX$ satisfies~\ref{ass:Luzin}.
In particular~$(X,\T)$ is separable.
\end{remark}

A Dirichlet space~$(\mbbX,\mcE)$ with~$\mbbX$ satisfying~\ref{ass:Luzin} is
\begin{itemize}
\item \emph{local} if~$\mcE(f,g)=0$ for every~$f,g\in\dom$ with~$\supp[f]$,~$\supp[g]$ compact,~$\supp[f]\cap\supp[g]=\emp$;

\item \emph{strongly local} if~$\mcE(f,g)=0$ for every~$f,g\in\dom$ with~$\supp[f]$,~$\supp[g]$  compact and~$f$ constant on a neighborhood of~$\supp[g]$;

\item \emph{regular} if~$\mbbX$ satisfies~\ref{ass:Polish}, and~$\Cz(\T)\cap \dom$ is both dense in~$\dom$ and dense in the space~$\Cz(\T)$ of all $\T$-continuous functions on~$X$ vanishing at infinity.
\end{itemize}

\paragraph{Domains}
Let~$(\mbbX,\mcE)$ be a Dirichlet space. We write~$\domb\eqdef \dom\cap L^\infty(\mssm)$. The \emph{extended domain}~$\domext$ of~$(\mcE,\dom)$ is the space of all functions~$f\in L^0(\mssm)$ so that there exists an $\mcE^{1/2}$-fundamental (i.e.\ Cauchy) sequence~$\seq{f_n}_n\subset \dom$ with $\mssm$-a.e.-$\nlim f_n=f$. We write~$\domextb\eqdef\domext\cap L^\infty(\mssm)$. The bilinear form~$\mcE$ on~$\mcF$ extends to a (non-relabeled) bilinear form on~$\domext$,~\cite[Prop.~3.1]{Kuw98}. Furthermore,
\begin{itemize}
\item $\domb$ is an algebra with respect to\ the pointwise multiplication, \cite[Prop.~I.2.3.2]{BouHir91};

\item if~$(\mbbX,\mcE)$ is quasi-regular, then $\domb$ is dense in~$\dom$, \cite[Cor.~2.1]{Kuw98}; 

\item if~$(\mbbX,\mcE)$ is quasi-regular, then $\dom$ is separable, \cite[Prop.~IV.3.3, p.~102]{MaRoe92};
\end{itemize}

\paragraph{Quasi-interior, quasi-closure} Let~$(\mbbX,\mcE)$ be a quasi-regular Dirichlet space. Every $f\in\dom$ has an \emph{$\mcE$-q.e.-unique} $\mcE$-quasi-continuous $\mssm$-representative, denoted by~$\reptwo f$, \cite[Prop.~IV.3.3.(iii)]{MaRoe92}.
For~$A\subset X$ set
\begin{align*}
\msU(A)\eqdef& \set{G : G \text{ is an~$\mcE$-quasi-open subset of } A} \comm
\\
\msF(A)\eqdef& \set{F : F \text{ is an~$\mcE$-quasi-closed superset of } A }\fstop
\end{align*}
By~\cite[Thm.~2.7]{Fug71},~$\msU(A)$ has an $\mcE$-q.e.-maximal element denoted by~$\intE A$, $\mcE$-quasi-open, and called the $\mcE$-\emph{quasi-interior} of~$A$.
Analogously,~$\msF(A)$ has an $\mcE$-q.e.-minimal element denoted by~$\clE A$, $\mcE$-quasi-closed, and called the $\mcE$-\emph{quasi-closure} of~$A$.

\paragraph{Carr\'e du champ operators}
Let~$(\mbbX,\mcE)$ be a Dirichlet space with~$\mbbX$ satisfying~\ref{ass:Hausdorff}, and set
\begin{align*}
\sqf{f,g}(h)\eqdef \mcE(fh,g)+\mcE(gh,f)-\mcE(fg,h)\comm \qquad \sqf{f}(h)\eqdef \sqf{f,f}(h) \comm \qquad f,g,h\in\domb \fstop
\end{align*}

We say that~$(\mcE,\dom)$ admits \emph{carr\'e du champ operator}~$\cdc$ if there exists~$\cdc\colon \domb^\tym{2}\to L^1(\mssm)$ such that
\[
\sqf{f,g}(h)=2 \int h\, \cdc(f,g)\diff \mssm\comma \qquad f,g,h\in\domb\fstop
\]

\paragraph{Energy measures}
Not every quasi-regular strongly local Dirichlet space admits a carr\'e du champ operator.
However we have the following:

\begin{theorem}[Cf.~{\cite[Thm.~5.2, Lem.s~5.1, 5.2]{Kuw98}}]\label{t:Kuwae}
Let~$(\mbbX,\mcE)$ be a quasi-regular strongly local Dirichlet space. Then, a bilinear form~$\sq{\emparg,\emparg}$ is defined on~$\domb^{\tym2}$ with values in~$\MbR(\Bo{\T},\Ne{\mcE})$ by
\begin{align}\label{eq:EnergyMeas}
2\int \reptwo{h}\diff \sq{f,g}= \sqf{f,g}(h) \comm \qquad f,g,h\in\domb\fstop
\end{align}
\end{theorem}

The bilinear form~$\sq{\emparg,\emparg}$ constructed above is called the \emph{energy measure} of~$(\mbbX,\mcE)$.
When~$(\mbbX,\mcE)$ is quasi-regular strongly local and additionally admits carr\'e du champ operator, then~$\sq{f,g} \ll \mssm$ for every~$f,g\in\domb$, in which case~$\cdc(f,g)\eqdef \frac{\diff \sq{f,g}}{\diff\mssm}$ is the Radon--Nikodym derivative of~$\sq{f,g}$ w.r.t.~$\mssm$.

\subsection{Broad local spaces and energy moderance}\label{ss:BroadLoc}
Let~$(\mbbX,\mcE)$ be a quasi-regular Dirichlet space.
For any $\mcE$-quasi-open~$E\subset X$ set
\begin{align}\label{eq:Xi0}
\msG(E)\eqdef&\ \set{ G_\bullet\eqdef\seq{G_n}_n :  G_n \text{~$\mcE$-quasi-open,~} G_n\subset G_{n+1} \text{~$\mcE$-q.e.,~} \cup_n G_n=E \qe{\mcE}} \comma
\end{align}
When~$E=X$ we simply write~$\msG$ in place of~$\msG(X)$.
We say that~$G_\bullet\in\msG(E)$ is $\mcE$-moderate if for each~$n\in\N_0$ there exists~$e_n\in\dom$ with~$e_n\geq 1$ $\mssm$-a.e.\ on~$G_n$, and we set
\begin{equation}\label{eq:Nests}
\begin{aligned}
\msG_0(E)\eqdef&\ \set{G_\bullet \in\msG(E): G_\bullet \text{~is $\mcE$-moderate}} \comm
\\
\msG_c(E)\eqdef&\ \set{G_\bullet \in\msG_0(E) : \cl_\T G_n \text{~is $\T$-compact for all~$n$}}\fstop
\end{aligned}
\end{equation}
For~$G_\bullet\in\msG_0(E)$, we write~$e_\bullet\eqdef\seq{e_n}_n$ for any sequence of functions witnessing the $\mcE$-moderance of~$G_\bullet$.
When the sequence~$e_\bullet$ is relevant, we write as well~$(G_\bullet,e_\bullet)\in\msG_0(E)$.
As usual, we omit the specification of~$E=X$.
Since~$\car\in\dotloc{\dom}$ by~\eqref{eq:E(1)=0}, then~$\msG_0\neq \emp$.
Clearly,~$\msG_c\subset \msG_0\subsetneq \msG$.

\begin{definition}[$\mcE$-moderance,~{\cite[Dfn.~2.22]{LzDSSuz20}}]
Let~$(\mbbX,\mcE)$ be a quasi-regular Dirichlet space. For a measure~$\mu\in\Msp(\Bo{\T},\Ne{\mcE})$ we say that $(G_\bullet,e_\bullet)\in\msG_0$ is \emph{$\mu$-moderated} if~$e_\bullet$ is additionally so that~$\mu \reptwo e_n <\infty$ for every~$n$.
We say that~$\mu$ is:
\begin{itemize}
\item \emph{$\mcE$-moderate} if there exists a $\mu$-moderated $G_\bullet\in\msG_0$;
\item \emph{absolutely $\mcE$-moderate} if for every $G_\bullet\in\msG_0$ there exists~$e_\bullet$ so that~$(G_\bullet, e_\bullet)$ is $\mu$-moderated.
\end{itemize}
\end{definition}

We refer the reader to~\cite[\S2.5.2]{LzDSSuz20} for the heuristics behind all the above definitions.

\subsubsection{Broad local spaces}
For~$G_\bullet\in\msG(E)$ and~$\msA\subset L^0(\mssm)$, we say that~$f\in L^0(\mssm\mrestr{E})$ is in the \emph{broad local space}~$\dotloc{\msA}(E,G_\bullet)$ if for every~$n$ there exists~$f_n\in \msA$ so that~$f_n=f$ $\mssm$-a.e.\ on~$G_n$.
The \emph{broad local space}~$\dotloc{\msA}(E)$ of~$(\mbbX,\mcE)$ relative to~$E$ is the space~\cite[\S4, p.~696]{Kuw98},
\begin{align}\label{eq:Xi}
\dotloc{\msA}(E)\eqdef \bigcup_{G_\bullet\in \msG(E)} \dotloc{\msA}(E,G_\bullet) \fstop
\end{align}
The set~$\dotloc{\msA}(E,G_\bullet)$ depends on~$G_\bullet$. Again, we omit the specification of~$E=X$.

\begin{proposition}[Extension of energy measure,~{\cite[Prop.~2.12]{LzDSSuz20}}]\label{p:PropertiesLoc}
Let~$(\mbbX,\mcE)$ be a quasi-regular strongly local Dirichlet space.
Then, the quadratic form $\sq{\emparg}\colon \dom\rar \MbpR(\Bo{\T},\Ne{\mcE})$ associated to the bilinear form~$\sq{\emparg,\emparg}$ in~\eqref{eq:EnergyMeas} uniquely extends to a non-relabeled form on~$\dotloc{\dom}$ with values in~$\Msp(\Bo{\T},\Ne{\mcE})$, satisfying:
\begin{enumerate}[$(i)$]
\item\label{i:p:PropertiesLoc:00} the representation property
\begin{align}\label{eq:RepresentationLoc}
\mcE(f,g)=\tfrac{1}{2}\sq{f,g} X \comm\qquad f,g\in\domext \semicolon
\end{align}

\item\label{i:p:PropertiesLoc:1} the truncation property
\begin{equation}\label{eq:TruncationLoc}
\begin{aligned}
f\wedge g \in \dotloc{\dom} \quad\text{and}\quad \sq{f\wedge g}=\car_{\ttset{\reptwo f\leq \reptwo g}}\sq{f}+\car_{\ttset{\reptwo f> \reptwo g}} \sq{g}\comm \qquad f,g\in\dotloc{\dom} \semicolon
\\
f\vee g \in \dotloc{\dom} \quad\text{and}\quad \sq{f\vee g}=\car_{\ttset{\reptwo f\leq \reptwo g}}\sq{g}+\car_{\ttset{\reptwo f> \reptwo g}} \sq{f}\comm \qquad f,g\in\dotloc{\dom} \semicolon
\end{aligned}
\end{equation}

\item\label{i:p:PropertiesLoc:2} the chain rule
\begin{equation}\label{eq:ChainRuleLoc}
\phi\circ f \in \dotloc{\dom} \quad\text{and}\quad \sq{\phi\circ f}=(\phi'\circ \reptwo f)^2 \cdot \sq{f}\comm \qquad f\in\dotloc{\dom} \comm \quad \begin{aligned}&\phi\in \mcC^1(\R)\comm\\ &\phi(0)=0\end{aligned} \semicolon
\end{equation}

\item\label{i:p:PropertiesLoc:5} the strong locality property
\begin{align}\label{eq:SLoc:2}
\car_G \sq{f}=\car_G \sq{g}\comm  \qquad G \text{~$\mcE$-quasi-open}\comm f,g\in\dotloc{\dom}\comm f\equiv g \text{~$\mssm$-a.e.\ on $G$}\fstop
\end{align}
\end{enumerate}

Furthermore,~$\dotloc{\dom}$ is an algebra for the pointwise multiplication, and
\begin{align}\label{eq:E(1)=0}
\car\in\dotloc{\dom}\comm \qquad \sq{\car}\equiv 0 \fstop
\end{align}
\end{proposition}

\begin{corollary}[Extension of carr\'e du champ operator]\label{c:BH}
Let~$(\mbbX,\mcE)$ be a quasi-regular strongly local Dirichlet space admitting carr\'e du champ operator~$(\cdc,\domb)$.
Then, the bilinear form~$\cdc\colon \domb^\tym{2}\to L^1(\mssm)$ extends to a non-relabeled bilinear form~$\cdc\colon\dotloc{\dom}^\tym{2}\to L^0(\mssm)$ representing the quadratic form~$\sq{\emparg}\colon \dotloc{\dom}\to \Msp(\Bo{\T},\Ne{\mcE})$ in Proposition~\ref{p:PropertiesLoc} in the sense that~$\tfrac{\diff\sq{f}}{\diff \mssm}=\cdc(f,f)$ for every~$f\in\dotloc{\dom}$.

Furthermore, the form~$\cdc\colon\dotloc{\dom}\to L^0(\mssm)$ satisfies the strong locality property
\begin{equation}\label{eq:c:BH:0}
\cdc(f,h)=\cdc(g,h) \as{\mssm} \quad \text{on} \quad \set{f\equiv g} \comma \qquad f,g,h\in\dotloc{\dom}\fstop
\end{equation}

\begin{proof}
Since~$\dotloc{\dom}=\dotloc{(\domb)}$, e.g.~\cite[Eqn.~(2.15)]{LzDSSuz20}, the extension is an immediate consequence of the definition of~$\cdc\colon\domb\to L^1(\mssm)$ together with the strong locality property of~$\sq{\emparg}$ in Proposition~\ref{p:PropertiesLoc}\ref{i:p:PropertiesLoc:5}.
Noting that our definition of broad local space~$\dotloc{\dom}$ is more restrictive than the definition of local space in~\cite[Dfn.~I.7.1.3]{BouHir91}, the equality~\eqref{eq:c:BH:0} is a standard consequence of e.g.~\cite[Prop.~I.7.1.4]{BouHir91}.
\end{proof}
\end{corollary}

\subsubsection{Local domains of bounded-energy and intrinsic distances}\label{sss:LocDom}
Let~$(\mbbX,\mcE)$ be a quasi-regular strongly local Dirichlet space, and~$\mu\in\Msp(\Bo{\T},\Ne{\mcE})$. For $\mcE$-quasi-open~$E\subset X$ and any $G_\bullet\in\msG_0(E)$, set
\begin{align*}
\DzLoc{\mu}(E,G_\bullet)\eqdef& \set{f\in \dotloc{\dom}(E,G_\bullet): \sq{f}\leq \mu}\comm \qquad \DzLoc{\mu}(E)\eqdef \tset{f\in \dotloc{\dom}(E): \sq{f}\leq \mu}\comm
\\
\Dz{\mu}(E)\eqdef&\ \DzLoc{\mu}(E)\cap \dom\comm \quad \ \DzLoc{\mu,\T}(E)\eqdef \DzLoc{\mu}(E)\cap \Cont(E,\T) \comm \quad \ \Dz{\mu,\T}(E)\eqdef \DzLoc{\mu,\T}(E)\cap \dom \fstop
\end{align*}
As usual, we omit~$E$ from the notation whenever~$E=X$. For~$G_\bullet\in\msG_0(E)$ we additionally denote by
\begin{align*}
\DzLocB{\mu}(E,G_\bullet)\eqdef \DzLoc{\mu}(E,G_\bullet)\cap L^\infty(\mssm)
\end{align*}
the space of $\mssm$-essentially uniformly bounded functions in~$\DzLoc{\mu}(E,G_\bullet)$.
Let the analogous definitions for~$\DzLocB{\mu}(E)$,~$\DzB{\mu}(E)$,~$\DzLocB{\mu,\T}(E)$, and~$\DzB{\mu,\T}(E)$ be given.

\begin{lemma}[{\cite[Prop.~2.26]{LzDSSuz20}}]\label{l:2.26}
Let~$(\mbbX,\mcE)$ be a quasi-regular strongly local Dirichlet space, and~$\mu\in\Msp(\Bo{\T},\Ne{\mcE})$ be absolutely $\mcE$-moderate.
Then, for every~$G_\bullet, G_\bullet'\in\msG_0$,
\[
\DzLocB{\mu}(G_\bullet) = \DzLocB{\mu}(G_\bullet')\fstop
\]
In particular, if~$f\in\DzLocB{\mu}$, then there exists~$G_\bullet\in \msG_c$ additionally so that~$\cl_\T G_k\subset G_{k+1}$ for all~$k\in\N$, and~$f_\bullet\subset \dom$, such that~$(G_\bullet,f_\bullet)$ witnesses that~$f\in\DzLocB{\mu}$.
\begin{proof}
The first assertion is shown in~\cite[Prop.~2.26]{LzDSSuz20}. The second follows from~\cite[Lem.~3.5(iii)]{Kuw98}.
\end{proof}
\end{lemma}

\begin{lemma}\label{l:ModerateDomain}
Let~$(\mbbX,\mcE)$ be a quasi-regular strongly local Dirichlet space, and~$\mu\in\Msp(\Bo{\T},\Ne{\mcE})$ be absolutely $\mcE$-moderate.
If~$\car\in\dom$, then
\[
\DzLocB{\mu}=\DzB{\mu}\fstop
\]
\begin{proof}
Since~$\car\in\dom$, the constant sequence~$X_\bullet\eqdef \seq{X}_k$ satisfies~$X_\bullet \in\msG_0$.
Since~$\mu$ is absolutely $\mcE$-moderate,~$X_\bullet$ is $\mu$-moderated, and we have that~$\DzLocB{\mu}=\DzLocB{\mu}(X_\bullet)$ by Lemma~\ref{l:2.26}.
On the other hand, it follows from the definition that~$\DzLocB{\mu}(X_\bullet)=\DzB{\mu}$.
\end{proof}
\end{lemma}

\subsubsection{Intrinsic distances}
We recall a definition of generalized intrinsic distance introduced in~\cite{LzDSSuz20}.
For more information on intrinsic distances, we refer the reader to~\cite[\S{2.6}]{LzDSSuz20}.

\begin{definition}[Intrinsic distance]\label{d:IntrinsicD}
Let~$\mu\in\Msp(\Bo{\T},\Ne{\mcE})$. The \emph{intrinsic distance generated by~$\mu$} is the extended pseudo-distance~$\mssd_\mu\colon X^{\times 2}\rar [0,\infty]$ defined as
\begin{align}\label{eq:IntrinsicD}
\mssd_\mu(x,y)\eqdef \sup\tset{f(x)-f(y) : f\in \DzLocB{\mu,\T}} \fstop
\end{align}
\end{definition}

Note that~$\mssd_\mu$ is always $\T^\tym{2}$-l.s.c., hence $\Bo{\T^\tym{2}}$- and $\A^{\otym 2}$-measurable, for it is the supremum of a family of $\T^\tym{2}$-continuous functions. Furthermore,~$\mssd_\mu$ is $\T$-admissible by definition, as witnessed by the bounded uniformity
\begin{align*}
\UP_\mu\eqdef \set{\mssd_f(x,y)\eqdef \abs{f(x)-f(y)}: f\in\DzLocB{\mu,\T}}\fstop
\end{align*}

In the case when~$(\mbbX,\mcE)$ is a regular strongly local Dirichlet space and~$\mu\eqdef\mssm$,~\eqref{eq:IntrinsicD} coincides with the standard definition of intrinsic distance, see~\cite[Prop.~2.31]{LzDSSuz20}.

\subsubsection{Localizability}
Let us recall (a version of) the definition of \emph{$\mu$-uniform $\T$-localizability} given in~\cite[\S3.4]{LzDSSuz20}, also cf.~\cite[\S2.5.2]{LzDSSuz20}, and amounting to the existence of good Sobolev cut-off functions.

\begin{definition}\label{d:Localizability}
Let~$(\mbbX,\mcE)$ be a quasi-regular strongly local Dirichlet space, and $\mu\in \Msp(\Bo{\T},\Ne{\mcE})$.
We say that $(\mbbX,\mcE,\mu)$ is \emph{$\mu$-uniformly latticially}, resp.\ \emph{algebraically}, \emph{$\T$-localizable}, in short: $(\mbbX,\mcE,\mu)$ satisfies~$(\Loc{\mu,\T})$, resp.~$(\aLoc{\mu,\T})$, if $\mu$ is $\mcE$-moderate, and there exists a latticial approximation to the identity $\seq{\theta_n}_n$, resp.\ an algebraic approximation to the identity~$\seq{\psi_n}_n$, uniformly bounded by~$\mu$ in energy measure, viz.
\begin{align}
\tag*{$(\Loc{\mu,\T})$}\label{ass:Loc}
&& 0 \leq \theta_n\leq \theta_{n+1} \nearrow_n \infty\comma \qquad \theta_n \in \DzB{\mu,\T} \quad \tparen{\subset \Cb(\T)}\fstop
\\
\tag*{$(\aLoc{\mu,\T})$}\label{ass:ALoc}
\text{resp.}&& 0 \leq \psi_n\leq \psi_{n+1} \nearrow_n 1\comma \qquad \psi_n \in \DzB{\mu,\T} \quad \tparen{\subset \Cb(\T)}\fstop
\end{align}
\end{definition}

\begin{remark}[{\cite[Rmk.~3.28, 3.29]{LzDSSuz20}}]\label{r:AlgLocTrivial}
Let us note the following.
\begin{enumerate*}[$(a)$]
\item We may additionally assume with no loss of generality that $\inter_\T \set{\theta_n=n}\neq \emp$ and that~$\theta_n\leq n$ for every~$n\in\N$.

\item If~$\car\in\dom$, then~\ref{ass:ALoc} is trivially satisfied with~$\psi_n\eqdef \car$ for all~$n$, and the same holds for~\ref{ass:Loc} for the functions~$\theta_n$ in~\cite[Rmk.~3.29]{LzDSSuz20}.

\item \ref{ass:Loc} implies~\ref{ass:ALoc} by setting~$\psi_n\eqdef \theta_n\vee 1$.
\end{enumerate*}
\end{remark}

\subsection{Rademacher and Sobolev-to-Lipschitz properties}\label{sss:RadStoL}
Let~$(\mbbX,\mcE)$ be a quasi-regular strongly local Dirichlet space,~$\mu\in \Msp(\Bo{\T})$, and~$\mssd\colon X^\tym{2}\to [0,\infty]$ be an extended pseudo-distance.

\begin{definition}[Rademacher and Sobolev-to-Lipschitz properties~{\cite{LzDSSuz20}}]\label{d:RadStoL}
We say that~$(\mbbX,\mcE,\mssd,\mu)$ has the:
\begin{itemize}
\item[($\Rad{\mssd}{\mu}$)] \emph{Rademacher property} (\emph{for}~$\A$) if, whenever~$\rep f\in \Lipu(\mssd,\A)$, then~$f\in \DzLoc{\mu}$;
\item[($\dRad{\mssd}{\mu}$)] \emph{distance-Rademacher property} if~$\mssd\leq \mssd_\mu$;

\item[($\ScL{\mssm}{\T}{\mssd}$)] \emph{Sobolev--to--continuous-Lipschitz property} if each~$f\in\DzLoc{\mu}$ has an $\mssm$-representative $\rep f\in\Lip^1(\mssd,\T)$;

\item[($\SL{\mu}{\mssd}$)] \emph{Sobolev--to--Lipschitz property} (\emph{for}~$\A$) if each~$f\in\DzLoc{\mu}$ has an $\mssm$-representat\-ive $\rep f\in\Lip^1(\mssd,\A)$;

\item[($\dcSL{\mssd}{\mu}{\mssd}$)] \emph{$\mssd$-continuous-Sobolev--to--Lipschitz property} (\emph{for}~$\A$) if each~$f\in \DzLoc{\mu}$ having a $\mssd$-continuous $\A$-measurable representative~$\rep f$ also has a representative $\reptwo f\in\Lip^1(\mssd,\A^\mssm)$ (possibly,~$\reptwo f\neq \rep f$);

\item[($\cSL{\T}{\mssm}{\mssd}$)] \emph{continuous-Sobolev--to--Lipschitz property} if each~$f\in \DzLoc{\mu,\T}$ satisfies $f\in\Lip^1(\mssd,\T)$;

\item[($\dSL{\mssd}{\mu}$)] \emph{distance Sobolev-to-Lipschitz property} if~$\mssd\geq \mssd_\mu$.
\end{itemize}
\end{definition}

As noted in~\cite{LzDSSuz20}, in all the above definitions we may equivalently additionally require that~$\rep f$ be uniformly bounded.
In the following, we shall make use of this fact without further mention.

We refer the reader to~\cite[Rmk.s~3.2, 4.3]{LzDSSuz20} for comments on the terminology and to~\cite[Lem.~3.6, Prop.~4.2]{LzDSSuz20} for the interplay of all such properties.
For a quasi-regular strongly local Dirichlet space~$(\mbbX,\mcE)$, they reduce to the following scheme:
\begin{equation}\label{eq:EquivalenceRadStoL}
\begin{tikzcd}
& (\Rad{\mssd}{\mu}) \arrow[r, Rightarrow] & (\dRad{\mssd}{\mu}) & & \text{\cite[Lem.~3.6]{LzDSSuz20}}\comma
\\
(\ScL{\mu}{\T}{\mssd}) \arrow[r, Rightarrow] & (\SL{\mu}{\mssd}) \arrow[r, Rightarrow] & (\cSL{\mu}{\T}{\mssd}) \arrow[r, Leftrightarrow] & (\dSL{\mssd}{\mu}) &  \text{\cite[Prop.~4.2]{LzDSSuz20}}\fstop
\end{tikzcd}
\end{equation}

\begin{remark}[About~$(\dcSL{\mssd}{\mu}{\mssd})$]\label{r:dcSL}
Let us note that, whereas both~$(\dcSL{\mssd}{\mu}{\mssd})$ and $(\cSL{\T}{\mu}{\mssd})$ are implied by~$(\SL{\mu}{\mssd})$ and coincide on metric spaces, they do \emph{not} ---~at least in principle~--- imply each other on extended metric spaces.
In particular, while the $\mssd$-Lipschitz representative in~$(\cSL{\T}{\mu}{\mssd})$ is taken to coincide with the given $\T$-continuous one, it is important in the definition of~$(\dcSL{\mssd}{\mu}{\mssd})$ to allow for the $\mssd$-Lipschitz representative to be different from the $\mssd$-continuous one, and for the former to be only~$\A^\mssm$-measurable, rather than $\A$-measurable.
\end{remark}

\begin{lemma}[{\cite[Prop.~3.7]{LzDSSuz20}}]\label{l:RadCompleteness}
Let~$(\mbbX,\mcE)$ be a quasi-regular strongly local Dirichlet space with~$\mbbX$ satisfying~\ref{ass:Luzin},~$\mssd\colon X^\tym{2}\to [0,\infty]$ be an extended pseudo-distance on~$X$, and~$\mu\in \Msp(\Bo{\T},\Ne{\mcE})$ be $\mcE$-moderate.
Further assume that~$(\mbbX,\mcE)$ satisfies~$(\dRad{\mssd}{\mu})$. If~$(X,\T,\mssd)$ is a complete extended metric-topological space (Dfn.~\ref{d:AES}), then so is~$(X,\T,\mssd_\mu)$.
\end{lemma}

\subsubsection{Metric completions}
Let~$(\mbbX,\mcE)$ be a Dirichlet space, and~$\mssd\colon X^\tym{2}\rar [0,\infty]$ be an extended distance. Further let~$(X^\iota,\mssd^\iota)$ be the abstract completion of~$(X,\mssd)$ and denote by~$\iota$ the completion embedding~$\iota\colon X\rar X^\iota$.
If~$\iota(X)$ is a Borel subset of~$X^\iota$, then~$\iota$ is $\Bo{\T_{\mssd}}/\Bo{\T_{\mssd^\iota}}$-measurable, and the image form~$(\mcE^\iota,\dom^\iota)$ of~$(\mcE,\dom)$ via~$\iota$ is well-defined on the image space~$\mbbX^\iota$.
When~$(\mbbX^\iota,\mcE^\iota)$ is a quasi-regular strongly local Dirichlet space, we denote its intrinsic distance by~$\mssd_{\mssm^\iota}$.

\begin{remark}\label{r:Completeness}
Let~$(\mbbX,\mcE)$ be a quasi-regular strongly local Dirichlet space, and~$\mssd\colon X^\tym{2}\to [0,\infty]$ be an extended distance on~$X$.
In order to discuss the interplay between~$(\mbbX,\mcE)$ and~$\mssd$ ---~precisely: the relation between~$\mssd$ and~$\mssd_\mssm$~---, it is \emph{natural} to assume that~$(X,\mssd)$ be complete.
In particular, it is usually not enough to resort to the completion of~$(X,\mssd)$, as we now show.

Indeed, let~$A\eqdef X^\iota\setminus\iota(X)$ be the complement of the $\iota$-image of~$X$ inside its completion.
The set~$A$ is in general not~$\mcE^\iota$-polar, that is, $(\mbbX,\mcE)$ and~$(\mbbX^\iota,\mcE^\iota)$ are in general not quasi-homeomorphic.
Furthermore, even in the case when~$(\mbbX,\mcE)$ and~$(\mbbX^\iota,\mcE^\iota)$ are quasi-homeomorphic, the corresponding intrinsic distances may be not related in any way.
That is, the metric completion~$(X^\iota,\mssd_\mssm^\iota)$ of~$(X,\mssd_\mssm)$ may differ (even as a set) from the (extended) metric space~$(X^\iota,\mssd_{\mssm^\iota})$, where this time we denoted by~$X^\iota$ the $\mssd$-completion of~$(X,\mssd)$ and by~$\mssd_{\mssm^\iota}$ the intrinsic distance of the Dirichlet form~$(\mcE^\iota,\dom^\iota)$.
For examples and counterexamples to the above statements, and for their relationship with the Sobolev-to-Lipschitz property, we refer the reader to~\cite[Ex.~4.7, \S4.1]{LzDSSuz20}.
\end{remark}

\subsection{Cheeger energies}\label{sss:CheegerE}
The content of this section is mostly taken from the detailed discussion of extended metric measure spaces put forward by L.~Ambrosio, N.~Gigli, and G.~Savar\'e in~\cite{AmbGigSav14}, Ambrosio, M.~Erbar, and Savar\'e in~\cite{AmbErbSav16}, and the one ---~more general still~--- in Savar\'e's monograph~\cite{Sav19}.

\begin{definition}[Extended metric-topological measure space,~{\cite[Dfn.~4.7]{AmbErbSav16}}]
By an \emph{extended metric-topological measure space}~$(X,\T,\mssd,\mssm)$ we mean an extended metric-topological space (Dfn.~\ref{d:AES}) together with a Radon measure~$\mssm$ restricted to the Borel $\sigma$-algebra~$\Bo{\T}$.
We use the expression \emph{extended metric-topological probability space} to indicate that~$\mssm$ is additionally a probability measure.
\end{definition}

Let~$(X,\T,\mssd,\mssm)$ be an extended metric-topological measure space.

\paragraph{Minimal relaxed slopes}
A $\Bo{\T}$-measurable function~$G\colon X\to [0,\infty]$ is a ($\mssm$-)\emph{relaxed slope} of~$f\in L^2(\mssm)$ if there exist~$\seq{f_n}_n\subset \Lip(\mssd,\Bo{\T})$ so that
\begin{itemize}
\item $L^2(\mssm)$-$\nlim f_n=f$ and~$\slo{f_n}$ $L^2(\mssm)$-weakly converges to~$\tilde G\in L^2(\mssm)$;
\item $\tilde G\leq G$ $\mssm$-a.e..
\end{itemize}

We say that~$G$ is the ($\mssm$-)\emph{minimal} ($\mssm$-)\emph{relaxed slope} of~$f\in L^2(\mssm)$, denoted by~$\slo[*]{f}$, if its $L^2(\mssm)$-norm is minimal among those of all relaxed slopes.
The notion is well-posed, and $\slo[*]{f}$ is in fact $\mssm$-a.e.\ minimal as well, see~\cite[\S{4.1}]{AmbGigSav14}.

\paragraph{Minimal weak upper gradients}
A Borel probability measure~$\boldpi$ on~$\AC^2(I;X)$ is a \emph{test plan of bounded compression} if there exists a constant~$C=C_\boldpi>0$ such that
\begin{align*}
(\ev_t)_\pfwd \boldpi \leq C\mssm \comma \qquad \ev_t\colon \AC^2(I;X)\to X\comma \ev_t\colon x_\emparg\mapsto x_t\fstop
\end{align*}
A Borel subset~$A\subset \AC^2(I;X)$ is called \emph{negligible} if~$\boldpi(A)=0$ for every test plan of bounded compression.
A property of $\AC^2(I;X)$-curves is said to hold \emph{for a.e.-curve} if it holds for every curve in a co-negligible set.

A $\Bo{\T}$-measurable function~$G\colon X\to [0,\infty]$ ($\mssm$-)\emph{weak upper gradient} of~$\rep f\in \mcL^0(\mssm)$ if
\begin{align*}
\tabs{\rep f(x_1)-\rep f(x_0)}\leq \int_0^1 G(x_r)\, \abs{\dot x}_r\diff r<\infty \quad \text{for a.e.\ curve } \seq{x_t}_t\in \AC^2(I;X)\fstop
\end{align*}

We say that~$G$ is the ($\mssm$-)\emph{minimal} ($\mssm$-)\emph{weak upper gradient} of~$f\in L^2(\mssm)$, denoted by~$\slo[w]{f}$, if it is $\mssm$-.a.e.-minimal among the weak upper gradients of~$\rep f$ for every representative~$\rep f$ of~$f$.
See e.g.~\cite[Dfn.~2.12]{AmbGigSav14} for the well-posedness of this notion, independently of the representatives of~$f$.

\paragraph{Asymptotic Lipschitz constants and asymptotic slopes}
For~$f\in \bLip(\mssd,\T)$ set
\begin{align*}
\Lia[\mssd]{f}(x,r)\eqdef \sup_{\substack{y,z\in B^\mssd_r(x)\\ \mssd(y,z)>0}} \frac{\abs{f(z)-f(y)}}{\mssd(z,y)}\comma \qquad x\in X\comma r>0 \fstop
\end{align*}
The \emph{asymptotic Lipschitz constant}~$\Lia[\mssd]{f}\colon X\rar [0,\infty]$ is defined as
\begin{align*}
\Lia[\mssd]{f}(x)\eqdef \lim_{r\downarrow 0} \Lia[\mssd]{f}(x,r) \comma
\end{align*}
with the usual convention that~$\Lia[\mssd]{f}(x)=0$ whenever~$x$ is a $\mssd$-isolated point in~$X$.
We drop the subscript~$\mssd$ whenever apparent from context.
Note that~$\Lia{f}$ is \mbox{$\mssd$-u.s.c.}\ by construction, thus if~$\mssd$ additionally metrizes~$\T$, then~$\Lia{f}$ is $\T$-u.s.c.\ as well, and therefore it is $\Bo{\T}$-measurable.

\paragraph{Cheeger energies}
Each extended metric-topological measure space~$(X,\T,\mssd,\mssm)$ is naturally endowed with a convex local energy functional, the l.s.c.\ $L^2(\mssm)$-relaxation of the natural energy on Lipschitz functions, called the \emph{Cheeger energy} of the space; e.g.~\cite{AmbGigSav14, Sav19}.
Several ---~\emph{a priori} inequivalent~--- definitions of Cheeger energy are possible. We collect here three of them, referring to~\cite{Sav19} for additional ones.

\begin{definition}[{\cite[Thm.~4.5]{AmbGigSav14}}]
The $\Ch[*]$-\emph{Cheeger energy} of~$f\in L^2(\mssm)$ is defined as
\begin{gather*}
\Ch[*,\mssd,\mssm](f)\eqdef \int \slo[*]{f}^2 \diff\mssm\comma \qquad
\domain{\Ch[*,\mssd,\mssm]}\eqdef \set{f\in L^2(\mssm) : \Ch[*,\mssd,\mssm](f)<\infty} \fstop
\end{gather*}
\end{definition}

\begin{definition}[{\cite[Rmk.~5.12]{AmbGigSav14}}]\label{d:CheegerW}
The $\Ch[w]$-\emph{Cheeger energy} of~$f\in L^2(\mssm)$ is defined as
\begin{gather*}
\Ch[w,\mssd,\mssm](f)\eqdef \int \slo[w]{f}^2 \diff\mssm\comma \qquad
\domain{\Ch[w,\mssd,\mssm]}\eqdef \set{f\in L^2(\mssm) : \Ch[w,\mssd,\mssm](f)<\infty} \fstop
\end{gather*}
\end{definition}

\begin{definition}[{\cite[Dfn.~6.1]{AmbErbSav16}}]\label{d:Cheeger}
The $\Ch[a]$-\emph{Cheeger energy} of~$f\in L^2(\mssm)$ is defined as
\begin{gather*}
\Ch[a,\mssd,\mssm](f)\eqdef \inf\nliminf \int g_n^2 \diff\mssm\comma \quad
\domain{\Ch[a,\mssd,\mssm]}\eqdef \set{f\in L^2(\mssm) : \Ch[a,\mssd,\mssm](f)<\infty} \comma
\end{gather*}
where the infimum is taken over all sequences~$\seq{f_n}_n \subset \bLip(\mssd,\T)$, $L^2(\mssm)$-strongly converging to~$f$, and all sequences~$\seq{g_n}_n$ of $\Bo{\T}^\mssm$-measurable functions satisfying~$g_n\geq \Lia{f_n}$ $\mssm$-a.e..
\end{definition}

In all cases, we shall omit the specification of either~$\mssd$,~$\mssm$ or both, whenever not relevant or apparent from context.
We refer the reader to~\cite[\S4]{AmbGigSav14} for a thorough treatment of~$\Ch[*]$, and to~\cite[\S6]{AmbErbSav16} for a thorough treatment of~$\Ch[a]$ in the setting of extended metric-topological probability spaces.

In the following, in order to refer to results in the literature, we shall need to make use of all the above definitions on some extended metric-topological \emph{probability} space.
To this end, we first show that they coincide on every such space.

\begin{proposition}\label{p:ConsistencyCheeger}
Let~$(X,\T,\mssd,\mssm)$ be an extended metric-topological \emph{probability} space. Then,
\begin{equation*}
\Ch[*,\mssd,\mssm]=\Ch[a,\mssd,\mssm]=\Ch[w,\mssd,\mssm]\comma
\end{equation*}
and each of these functionals is densely defined on~$L^2(\mssm)$.

\begin{proof}
Since~$\mssm X<\infty$, the domain of~$\Ch[*,\mssd,\mssm]$ contains~$\bLip(\mssd,\Bo{\T})$, and the latter is dense in~$L^2(\mssm)$ by e.g.~\cite[Prop.~4.1]{AmbGigSav14}.
Let us show the identification.
Firstly, note that our Definition~\ref{d:Cheeger} (i.e.~\cite[Dfn.~6.1]{AmbErbSav16}) differs from~\cite[Dfn.~5.1]{Sav19} ---~as do our definition of the asymptotic Lipschitz constant (again after~\cite{AmbErbSav16}) and the one in~\cite{Sav19}. Thus, we need to show that our definition of~$\Ch[a]$ coincides with the one of~$\CE[2]$ in~\cite[Dfn.~5.1]{Sav19}.
Once this identity is established, the assertion will be a consequence of the identification of both~$\Ch[*]$ and~$\CE[2]$ with~$\Ch[w]$.

The identification of~$\CE[2]$ with~$\Ch[w]$ is shown in~\cite[Thm.~11.7]{Sav19} for the choice $\msA\eqdef \bLip(\mssd,\T)$.
The identification of~$\Ch[*]$ with~$\Ch[w]$ is shown in~\cite[Thm.~6.2]{AmbGigSav14}. The assumption in Equation~(4.2) there is trivially satisfied since~$\mssm X=1$.
Thus, the proof is concluded if we show that
\begin{equation}\label{eq:p:Cheeger:1}
\Ch[w]\leq \Ch[a]\leq \CE[2]\fstop
\end{equation}

Since~$\T_\mssd$ is finer than~$\T$, for each~$x\in X$ and for each neighborhood~$U$ of~$x$, there exists~$r>0$ so that~$B^\mssd_r(x)\subset U$. Thus, for any $\Bo{\T}$-measurable~$f\colon X\rar \overline{\R}$,
\begin{align*}
\sup_{x,y\in B^\mssd_r(x)} \frac{f(z)-f(y)}{\mssd(x,y)} \leq \sup_{x,y\in U} \frac{f(z)-f(y)}{\mssd(x,y)} \fstop
\end{align*}
As a consequence, $\Lia[\mssd]{f}$~is dominated by the asymptotic Lipschitz constant of~$f$ as defined in~\cite[Eqn.~(2.48)]{Sav19}, and the second inequality in~\eqref{eq:p:Cheeger:1} follows.

The first inequality in~\eqref{eq:p:Cheeger:1} is a consequence of~\cite[Prop.~6.3$(b)$ and $(g)$]{AmbErbSav16}. Importantly, we note that~\cite{AmbErbSav16} denotes by~$\slo[w]{f}$ the minimal relaxed slope~$\slo[*]{f}$.
\end{proof}
\end{proposition}

\section{Localization and globalization}\label{s:LocGlob}
Let~$(\mbbX,\mcE)$ be a quasi-regular Dirichlet space, and~$E\subset X$ be $\mcE$-quasi-open, with~$\mssm E>0$.
Further set
\begin{equation}\label{eq:LocalizationKuwae}
\dom_E\eqdef\set{u\in\dom: \reptwo u \equiv 0\ \qe{\mcE} \text{ on } X\setminus E} \comma \qquad \mcE_E(u,v)\eqdef \mcE(u,v)\comma \quad u,v\in\dom_E\fstop
\end{equation}
Note that~$\dom_E\subset L^2(\mssm\mrestr{E})$.
Let us gather here various results proved by K.~Kuwae in~\cite{Kuw98}.

\begin{proposition}[Kuwae]\label{p:Kuwae}
Let~$(\mbbX,\mcE)$ be a quasi-regular Dirichlet space, and~$E\subset X$ be $\mcE$-quasi-open, with~$\mssm E>0$.
Then,
\begin{enumerate}[$(i)$]
\item\label{i:p:Kuwae:1} \cite[Lem.~3.4(i)]{Kuw98} $\dom_E$ is dense in~$L^2(m\mrestr{E})$ and~$(\mcE_E,\dom_E)$ is a Dirichlet form on~$L^2(m\mrestr{E})$;

\item\label{i:p:Kuwae:2} \cite[Lem.~3.4(ii), Lem.~3.5(ii)\&(iv)]{Kuw98} $A\subset E$ is $\mcE$-polar, resp.\ $\mcE$-quasi-open, if and only if it is $\mcE_E$-polar, resp. $\mcE_E$-quasi-open;

\item\label{i:p:Kuwae:3} \cite[Lem.~3.4(ii)]{Kuw98} the restriction to~$E$ of any $\mcE$-quasi-continuous function is $\mcE_E$-quasi-continuous;

\item\label{i:p:Kuwae:4} \cite[Lem.~3.4(ii)]{Kuw98} $(\mcE_E,\dom_E)$ is quasi-regular;

\item\label{i:p:Kuwae:5} \cite[Thm.~4.2]{Kuw98} $\dotloc{\dom}(E)=\dotloc{(\dom_E)}$. In particular,~$\dom\restr_{E}\eqdef\set{f\restr_E: f\in\dom}\subset \dotloc{(\dom_E)}$.
\end{enumerate}
\end{proposition}

We denote by~$\mbbX_E\eqdef (E,\T_E,\A_E,\mssm\mrestr{E})$ the restricted space of~$\mbbX$ to~$E$, defined in the obvious way, and by~$(\mbbX_E,\mcE_E)$ the corresponding quasi-regular Dirichlet space constructed in Proposition~\ref{p:Kuwae}\ref{i:p:Kuwae:1},~\ref{i:p:Kuwae:4}.
When~$(\mbbX_E,\mcE_E)$ is a quasi-regular strongly local Dirichlet space, we denote by~$\sqE{\emparg,\emparg}\colon (\dom_E)_b^\tym{2}\to \MbpR(\Bo{\T},\Ne{\mcE_E})$ its energy measure, and by~$\sqE{\emparg}$ the corresponding extension to~$\dotloc{(\dom_E)}$ with values in~$\Msp(\Bo{\T},\Ne{\mcE_E})$ constructed in Proposition~\ref{p:PropertiesLoc}.

\begin{corollary}[Kuwae]\label{c:Kuwae}
Let~$(\mbbX,\mcE)$ be a quasi-regular strongly local Dirichlet space, and~$E\subset X$ be $\mcE$-quasi-open, with~$\mssm E>0$.
Then, $(\mbbX_E,\mcE_E)$ is a quasi-regular strongly local Dirichlet space. Furthermore,
\[
\sqE{f,g}=\car_E\sq{f,g}\comma \quad f,g\in\dom_E \qquad \text{and} \qquad \sqE{f}=\car_E\sq{f} \comma \quad f\in\dotloc{(\dom_E)} \fstop
\]
\begin{proof}
The quasi-regularity of~$(\mbbX_E,\mcE_E)$ was noted in Proposition~\ref{p:Kuwae}\ref{i:p:Kuwae:4}.
By (the proof\footnote{Mistakenly marked as `Proof of Corollary~6.1', see~\cite[p.~702]{Kuw98}.} of)~\cite[Cor.~5.1]{Kuw98}, the form~$(\mcE_E,\dom_E)$ is as well strongly local, and~$\sqE{f,g}=\car_E \sq{f,g}$ for every~$f,g\in\dom_E$.
The last assertion follows by extension to the broad local domain as in Proposition~\ref{p:PropertiesLoc}\ref{i:p:PropertiesLoc:5}.
\end{proof}
\end{corollary}

\begin{remark}[\emph{Caveat}]\label{r:Caveat}
Let~$(\mbbX,\mcE)$ be a quasi-regular strongly local Dirichlet space, and~$E\subset X$ be $\mcE$-quasi-open with~$0<\mssm E<\infty$.
Then,~$(\mcE_E,\dom_E)$ is conservative (and~$\car_E\in\dom_E$) if and only if $E$ is additionally $\mcE$-quasi-closed.
Indeed, if $E$ is both $\mcE$-quasi-open and $\mcE$-quasi-closed, then it is $\mcE$-invariant, see e.g.~\cite[Cor.~4.6.3, p.~194]{FukOshTak11}.
Thus,~$\car_E\in\dom$ by $\mcE$-invariance of~$E$, e.g.~\cite[p. 53]{FukOshTak11}, and thus~$\car_E\in\dom_E$ by definition of~$\dom_E$.
If otherwise,~$E$ is not $\mcE$-quasi-closed, then it is not $\mcE$-invariant, hence~$\car_E\not\in\dom\supset \dom_E$.
\end{remark}

\begin{corollary}\label{c:RestrictionDzLoc}
Let~$(\mbbX,\mcE)$ be a quasi-regular strongly local Dirichlet space,~$E\subset X$ be $\mcE$-quasi-open, with~$\mssm E>0$, and~$\mu\in\Msp(\Bo{\T},\Ne{\mcE})$.
Then, with obvious meaning of the notation,
\begin{gather*}
\DzLoc{\mu}(E)= \DzLocE{\mu\mrestr{E}}\eqdef \set{f\in\dotloc{(\dom_E)}: \sqE{f}\leq \mu\mrestr{E}}\comma
\\
\DzLocB{\mu}(E)= \DzLocBE{\mu\mrestr{E}}\comma \qquad \DzLocB{\mu,\T}(E)=\ \DzLocBE{\mu\mrestr{E},\T} \fstop
\end{gather*}

Further assume that~$\mu$ is additionally absolutely $\mcE$-moderate, and that there exists~$e_E\in\domb$ with~$0\leq \reptwo e_E\leq 1$ $\mcE$-q.e.\ on~$E$.
Then,
\begin{equation}\label{eq:c:RestrictionDzLoc:1}
\DzLocB{\mu}\restr_E\eqdef \set{f\restr_E : f\in \DzLocB{\mu}}\subset \DzLocE{\mu\mrestr{E}} \fstop
\end{equation}

\begin{proof}
We only show the first equality, all others being trivial consequences of the first one.
Firstly, note that~$\mu\mrestr{E}\in\Msp(\Bo{\T_E},\Ne{\mcE_E})$ as a consequence of Proposition~\ref{p:Kuwae}\ref{i:p:Kuwae:2}.
Thus, the statement is well-posed.
Let~$f\in\DzLoc{\mu}(E)$.
Then,~$\sqE{f}=\car_E\sq{f}\leq \car_E \mu\defeq\mu\mrestr{E}$, where the first equality is shown in Corollary~\ref{c:Kuwae} and the inequality holds by definition of~$\DzLoc{\mu}(E)$.
This concludes the first assertion in light of Proposition~\ref{p:Kuwae}\ref{i:p:Kuwae:5}.

We show the second assertion. Let~$(G_\bullet,e_\bullet)\in\msG_0$ be $\mu$-moderated, and set~$G_k'\eqdef G_k\cup E$ for every~$k\in \N$.
Since~$e_E\equiv 1$ $\mssm$-a.e.\ on~$E$, the function~$e_k\vee e_E$ satisfies~$e_k\vee e_E \in\domb$ and~$e_k\vee e_E\equiv 1$ $\mssm$-a.e.\ on~$G_k'$. 
Thus,~$G_\bullet'\in\msG_0$.
Now, let~$(G^f_\bullet, f_\bullet)$ with~$G^f_\bullet\in\msG_0$ be witnessing that~$f\in \DzLocB{\mu}$.
Since~$\mu$ is absolutely $\mcE$-moderate, both~$G^f_\bullet$ and~$G'_\bullet$ are $\mu$-moderated.
Thus,~$f\in\DzLocB{\mu}(G^f_\bullet)=\DzLocB{\mu}(G'_\bullet)$ by Lemma~\ref{l:2.26}.
By definition of the latter space, there exists~$f^E\in \dom$ with~$f^E\equiv f$ $\mssm$-a.e.\ on~$E$ and~$\sq{f}\leq \mu$.
By Proposition~\ref{p:Kuwae}\ref{i:p:Kuwae:5} we conclude that~$f^E\restr_E \in \dotloc{(\dom_E)}$.
In fact,~$f^E\restr_{E}\in \DzLocBE{\mu\mrestr{E}}$ by locality of~$\sq{f^E}$ as in the first assertion, which concludes the second assertion and the proof.
\end{proof}
\end{corollary}

\begin{remark}
\begin{enumerate*}[$(a)$]
\item In general,~$\DzLocE{\mu\mrestr{E},\T}\not\subset\DzLocB{\mu,\T}\restr_E$.
Indeed, let~$\mcE$ be the $0$-form on $L^2(\mssm)$. Then, $\DzLocE{\mu\mrestr{E},\T}=\Cb(\T_E)\not\subset \Cb(\T)\restr_{E}=\DzLoc{\mu,\T}\restr_{E}$, unless $E$ is additionally $\T$-closed.

\item In general,~$\DzB{\mu}\restr_E \not\subset \DzE{\mu\mrestr{E}}$, since~$\car\in \Dz{\mu}\restr_E$ yet~$\car_E\notin \dom \supset \dom_E$, unless~$E$ is additionally $\mcE$-quasi-closed.
\end{enumerate*}
\end{remark}

Let us state here the following localization property for~$(\Rad{\mssd}{\mu})$.

\begin{proposition}[Localization of the Rademacher property]\label{p:LocRad}
Let~$(\mbbX,\mcE)$ be a quasi-regular strongly local Dirichlet space with~$\mbbX$ satisfying~\ref{ass:Luzin}, $\mu\in \Msp(\Bo{\T},\Ne{\mcE})$ be absolutely $\mcE$-moderate, and~$\mssd\colon X^\tym{2}\to [0,\infty]$ be an extended pseudo-distance on~$X$.
Further assume that
\begin{enumerate}[$(a)$]
\item\label{i:p:LocRad:1} $\T_\mssd$ is separable;
\item\label{i:p:LocRad:2} $\mssd(\emparg, x_0)\colon X\to [0,\infty]$ is $\A$-measurable for every~$x_0\in X$.
\end{enumerate}

If~$(\mbbX,\mcE,\mssd,\mu)$ satisfies~$(\Rad{\mssd}{\mu})$, then~$(\mbbX_E,\mcE_E,\mssd,\mu\mrestr{E})$ satisfies~$(\Rad{\mssd}{\mu\mrestr{E}})$ for every $\mcE$-quasi-open $E\subset X$ with~$\mssm E>0$.
\end{proposition}

We need two preparatory results.
The first one is a version of McShane's Extension Theorem for Lipschitz functions.
We note that our proof of the following statement in~{\cite[Lem.~2.1]{LzDSSuz20}} contains a gap, which we fill in the proof below.

\begin{lemma}[Constrained McShane extensions,~{\cite[Lem.~2.1]{LzDSSuz20}}]\label{l:McShane}
Let~$(X,\mssd)$ be an extended metric space. Fix~$A\subset X$,~$A\neq \emp$, and let~$\rep f\colon A\rar \R$ be a function in~$\bLip(A,\mssd)$. Further set
\begin{equation}\label{eq:McShane}
\begin{aligned}
\overline{f}\colon x&\longmapsto \sup_A \rep f\wedge \inf_{a\in A} \tparen{\rep f(a)+\Li[\mssd]{\rep f}\,\mssd(x,a)} \comm & 0\cdot\infty&\eqdef \infty\comm
\\
\underline{f}\colon x&\longmapsto \inf_A \rep f \vee \sup_{a\in A} \tparen{\rep f(a)-\Li[\mssd]{\rep f}\,\mssd(x,a)}\comm & 0\cdot\infty&\eqdef \infty \fstop
\end{aligned}
\end{equation}

Then,
\begin{enumerate}[$(i)$]
\item\label{i:l:McShane:1} $\underline{f}=\rep f=\overline{f}$ on~$A$ and~$\inf_A \rep f\leq \underline{f}\leq \overline{f}\leq \sup_A \rep f$ on~$X$;

\item\label{i:l:McShane:2} $\underline{f}$, $\overline{f}\in \bLip(\mssd)$ with~$\Li[\mssd]{\underline{f}}=\Li[\mssd]{\overline{f}}=\Li[\mssd]{\rep f}$;

\item\label{i:l:McShane:3} $\underline{f}$, resp.~$\overline{f}$, is the minimal, resp.\ maximal, function satisfying~\ref{i:l:McShane:1}-\ref{i:l:McShane:2}, that is, for every~$\rep g\in \bLip(\mssd)$ with~$\inf_A \rep f\leq \rep g\leq \sup_A \rep f$, $\rep g\restr_A=\rep f$ on~$A$ and~$\Li[\mssd]{\rep g}\leq \Li[\mssd]{\rep f}$, it holds that~$\underline{f}\leq \rep g \leq \overline{f}$.
\end{enumerate}

\begin{proof}
It is erroneously claimed in~\cite{LzDSSuz20} that, if~$\rep f$ is non-constant, one can assume that~$\Li[\mssd]{\rep f}>0$ with no loss of generality, cf.\ Remark~\ref{r:McShane}.
Since this does not hold, we treat separately the case when~$\Li[\mssd]{\rep f}=0$.
In this case, partition~$A$ into its (pairwise disjoint) $\mssd$-accessible components~$\seq{A_i}_{i\in I}$.
Since~$\Li[\mssd]{\rep f}=0$, then~$\rep f\restr_{A_i}\equiv a_i$ is constant for every~$i$.
Let~$X_i$ be the unique~$\mssd$-accessible component of~$X$ containing~$A_i$. 
Set~$\underline{f}(x)\eqdef a_i$ if~$x\in X_i$ and~$\underline{f}(x)\eqdef \inf_A \rep f$ otherwise, and similarly~$\overline{f}(x)\eqdef a_i$ if~$x\in X_i$ and~$\overline{f}(x)\eqdef \sup_A \rep f$ otherwise.
Then,~$\underline{f}$ and~$\overline{f}$ satisfy the required properties.

The rest of the proof is as in~\cite{LzDSSuz20}.
\end{proof}
\end{lemma}

\begin{corollary}\label{c:McShane}
Let~$(X,\mssd)$ be a separable extended pseudo-metric space. Further let~$\A$ be a $\sigma$-algebra on~$X$ and assume that~$\mssd(\emparg,x_0)\colon X\to [0,\infty]$ is $\A$-measurable for every~$x_0\in X$.
Then,~$\Lip(\mssd)=\Lip(\mssd,\A)$.
\begin{proof}
By a standard truncation argument, it suffices to show that every~$\rep f\in\bLip(\mssd)$ is $\A$-measurable.
Choosing~$A=X$ in Lemma~\ref{l:McShane} we have that~$\rep f=\overline{f}$, thus it suffices to show that~$\overline{f}$ is measurable.
Since~$(X,\mssd)$ is separable, there exists a countable $\mssd$-dense set~$D\subset X$. Since~$(x,a)\mapsto\tparen{\rep f(a)+\Li[\mssd]{\rep f}\, \mssd(x,a)}$ is jointly $\mssd$-continuous,
\[
\inf_{a\in X} \tparen{\rep f(a)+\Li[\mssd]{\rep f}\, \mssd(x,a)}=\inf_{a\in D} \tparen{\rep f(a)+\Li[\mssd]{\rep f}\, \mssd(x,a)} \fstop
\]
Since the infimum of a countable family of $\A$-measurable functions is $\A$-measurable, the conclusion follows since~$x\mapsto \tparen{\rep f(a)+\Li[\mssd]{\rep f}\, \mssd(x,a)}$ is $\A$-measurable for every~$a\in X$ by assumption.
\end{proof}
\end{corollary}

\begin{proof}[Proof of Proposition~\ref{p:LocRad}]
Let~$\rep f\in \bLipu(E,\mssd,\A_E)$.
By Lemma~\ref{l:McShane}\ref{i:l:McShane:2} and Corollary~\ref{c:McShane}, the upper constrained McShane extension~$\overline{f}\colon X\to \R$ of~$\rep f\colon E\to \R$ defined in~\eqref{eq:McShane} satisfies~$\overline f\in \bLipu(\mssd,\A)$.
By~$(\Rad{\mssd}{\mu})$ we conclude that~$\tclass[\mssm]{\overline f}\in \DzLocB{\mu}$.
Thus,~$\tclass[\mssm\mrestr{E}]{\overline f\restr_E}\in \DzLocBE{\mu_E}$ by~\eqref{eq:c:RestrictionDzLoc:1} in Corollary~\ref{c:RestrictionDzLoc}.
This concludes the assertion since~$\overline f=\rep f$ everywhere on~$E$ by Lemma~\ref{l:McShane}\ref{i:l:McShane:1}.
\end{proof}

\begin{remark}\label{r:LocRad}
\begin{enumerate*}[$(a)$]
\item In Proposition~\ref{p:LocRad}, assumption~\ref{i:p:LocRad:1} is usually stronger than~\ref{i:p:LocRad:2}. For instance, if~$\mssd_\mu$ satisfies the assumption in Proposition~\ref{p:LocRad}\ref{i:p:LocRad:1}, then it satisfies as well the assumption in Proposition~\ref{p:LocRad}\ref{i:p:LocRad:2} by~\cite[Thm.~3.13]{LzDSSuz20}.

\item\label{i:r:LocRad:2} It is essential for the validity of the above proposition that the Rademacher property be phrased with the broad \emph{local} space~$\DzLocE{\mu\mrestr{E}}$ (as opposed to~$\DzE{\mu\mrestr{E}}$).
Indeed,~$\car_E\in \bLip(E,\mssd,\A_E)$, yet in general~$\car_E\notin \dom_E$ as in Remark~\ref{r:Caveat}.
\end{enumerate*}
\end{remark}

The statement analogous to Proposition~\ref{p:LocRad} with, e.g.,~$(\SL{\mu}{\mssd})$ in place of~$(\Rad{\mssd}{\mu})$ does \emph{not} hold.
This is a consequence of the failure of the converse inclusion in~\eqref{eq:c:RestrictionDzLoc:1}.
In the next example we discuss a quasi-regular strongly local Dirichlet space~$(\mbbX,\mcE)$ satisfying~$(\SL{\mu}{\mssd})$ for which there exists a $\T$-open set~$E$ with~$\mssm E=\mssm X$ and such that~$(\mbbX_E,\mcE_E,\mssd,\mu\mrestr{E})$ does not satisfy~$(\SL{\mu\mrestr{E}}{\mssd})$.

\begin{example}\label{ese:FailureLocSL}
Let~$\mbbX$ be standard interval~$[-1,1]$, and~$\mssd$ be the standard Euclidean distance on~$X$.
Set~$E\eqdef [-1,1]\setminus \set{0}$, and let~$\mu\eqdef \mssm=\Leb^1$ be the standard Lebesgue measure on~$X$.
Further let~$(\mcE,\dom)$ be the regular Dirichlet form properly associated with the Brownian motion on~$X$ with reflecting boundary conditions.
Now, let~$\rep f\colon E\to \R$ be defined by~$\rep f\eqdef \car_{(0,1]}-\car_{[-1,0)}$.
Setting~$G_k^E\eqdef [-1,1]\setminus[-\tfrac{1}{k},\tfrac{1}{k}]$ and~$f^E_k(x)\eqdef -1 \vee k x \wedge 1$, it is straightforwardly verified that~$(G^E_\bullet, f^E_\bullet)$ witnesses that~$f\in \DzLocBE{\mu_E}$.
On the other hand,~$f\notin \DzLocB{\mu}$.
Indeed, argue by contradiction that there exists~$(G_\bullet, f_\bullet)$, with~$G_\bullet\in\msG_0$, witnessing that~$f\in \DzLocB{\mu}$.
Since~$\set{0}$ is \emph{not} $\mcE$-polar, there exists~$G_k\in G_\bullet$ and~$f_k\in\dom$ such that~$f\equiv f_k$ $\mssm$-a.e.\ on~$G_k$ and~$G_k\ni \set{0}$.
This is a contradiction since every~$h\in \dom$ has a $\T$-continuous $\mssm$-representative, yet~$f$ has no $\mssm$-representative continuous at~$0$.
\end{example}

\begin{remark}
Again specularly to the case of the Rademacher property discussed in Remark~\ref{r:LocRad}, if we had phrased, e.g.,~$(\SL{\mu}{\mssd})$ with~$\DzB{\mu}$ in place of~$\DzLocB{\mu}$, then the localization of~$(\SL{\mu}{\mssd})$ to~$E$ would be immediate, since~$\DzBE{\mu_E}\subset \DzB{\mu}$ because~$\dom_E\subset \dom$.
\end{remark}

Example~\ref{ese:FailureLocSL} shows that the lack of a \emph{Zygmund Lemma} ---~`$W^{1,\infty}\hookrightarrow\Lip$'~--- for $E$ is an obstruction to the localization to~$E$ of the Sobolev-to-Lipschitz property on~$X$.
In fact, the situation for~$(\SL{\mu}{\mssd})$ is opposite to that for~$(\Rad{\mssd}{\mu})$, in the sense that one should rather expect \emph{globalization} (as opposed to: \emph{localization}), which takes the following form.

\begin{proposition}[Globalization of Sobolev--to--Lipschitz-type properties]\label{p:GlobCSL}
Let~$(\mbbX,\mcE)$ be a quasi-regular strongly local Dirichlet space with~$\mbbX$ satisfying~\ref{ass:Luzin},~$\mu\in\Msp(\Bo{\T},\Ne{\mcE})$ be absolutely $\mcE$-moderate, and~$\mssd\colon X^\tym{2}\to [0,\infty]$ be an extended distance on~$X$.
Further assume that
\begin{enumerate}[$(a)$]
\item\label{i:p:GlobCSL:1} $\T_{\mssd}=\T$;
\item\label{i:p:GlobCSL:2} $(X,\mssd)$ is a (possibly: extended) length space;
\item\label{i:p:GlobCSL:3} there exists a $\T$-open covering~$\msE$ of~$X$, with the following properties:
\begin{enumerate}[$({c}_1)$]
\item~$\mssm E>0$ for each~$E\in\msE$;
\item $\msE$ is \emph{$\mcE$-moderate}, i.e.\ for each~$E\in\msE$ there exists~$e_E\in\dom$ with~$0\leq e_E\leq 1$ $\mssm$-a.e.\ on~$E$, cf.~\cite[Dfn.~2.19]{LzDSSuz20};
\item~$(\SL{\mu\mrestr{E}}{\mssd})$, resp.~$(\cSL{\T_E}{\mu\mrestr{E}}{\mssd})$, holds for $(\mbbX_E,\mcE_E,\mssd,\mu\mrestr{E})$ for every $E\in\msE$.
\end{enumerate}
\end{enumerate}
Then,~$(\SL{\mu}{\mssd})$, resp.~$(\cSL{\T}{\mu}{\mssd})$, holds for~$(\mbbX,\mcE,\mssd,\mu)$.
\begin{proof}
Arguing as in the proof of~\cite[Lem.~3.23]{LzDSSuz20} with~$\mssd$ in place of~$\mssd_\mu$, we may restrict ourselves to the case when~$\mssd$ is everywhere finite.

We start with the case of~$(\SL{\mu}{\mssd})$.
Let~$f\in \DzLocB{\mu}$. Since $\msE$ is $\mcE$-moderate and $\mu$ is absolutely $\mcE$-moderate,~\eqref{eq:c:RestrictionDzLoc:1} in Corollary~\ref{c:RestrictionDzLoc} holds, and it implies that~$f\restr_E\in \DzLocBE{\mu_E}$ for every~$E\in\msE$.
We may now apply~$(\SL{\mu_E}{\mssd})$ to conclude that for every~$E\in\msE$ there exists~$\rep f^E\in \bLipu(E,\mssd,\A_E)$ with~$\tclass[\mssm_E]{\rep f^E}\equiv f\restr_{E}$ $\mssm_E$-a.e..
As a consequence of~\ref{i:p:GlobCSL:1}, $\bLipu(E,\mssd,\A_E)=\bLipu(E,\mssd,\T_E)$, thus we may omit the notation for representatives, and simply write~$f^E$ in place of both~$\rep f^E$ and~$\tclass[\mssm]{\rep f^E}$.
Now, since~$f^{E_1} \equiv f \equiv f^{E_2}$ $\mssm$-a.e.\ on~$E_1\cap E_2$ for every~$E_1,E_2\in\msE$, and since~$\mssm$ has full $\T$-support by~\ref{ass:Luzin}, it follows from the $\T_{E_i}$-continuity of~$f^{E_i}$ that~$f^{E_1}\equiv f^{E_2}$ everywhere on~$E_1\cap E_2$.
Therefore, since $\mcC(\T)$ is a sheaf on~$(X,\T)$, there exists a unique~$\rep f\in\mcC(\T)$ with~$\rep f\restr_{E}\equiv f^E$ for every~$E\in\msE$.
Since~$\tclass[\mssm]{\rep f}\equiv f^E\equiv f$~$\mssm$-a.e.\ on~$E$ for each~$E$ in the covering~$\msE$ of~$X$, we conclude that~$\rep f$ is a (in fact: \emph{the} unique) $\T$-continuous $\mssm$-representative of~$f$.
As customary, we may therefore drop the distinction between~$\rep f$ and~$f$.
It remains to show that~$f\in \bLipu(\mssd)$.
Since~$(X,\mssd)$ is a length space, it is $1$-quasi-convex (see e.g.~\cite[Dfn.~2.5.4]{CobMicNic19}), and the fact that~$f\restr_E\in \bLipu(E,\mssd)$ for every~$E$ in the $\T_\mssd$-open covering~$\msE$ implies that~$f\in \bLipu(\mssd)$, see e.g.~\cite[Thm.~2.5.6]{CobMicNic19}, which concludes the proof.

For the case of~$(\cSL{\T}{\mu}{\mssd})$ it suffices to note that~\eqref{eq:c:RestrictionDzLoc:1} still holds when the broad local spaces of functions with bounded energy are replaced by their additionally continuous counterparts, since the restriction to~$E$ of any $\T$-continuous function is $\T_E$-continuous.
\end{proof}
\end{proposition}

\begin{remark}
One relevant case for the application of Proposition~\ref{p:GlobCSL} is when~$\mssd=\mssd_\mssm$, in which case the assumption in Proposition~\ref{p:GlobCSL}\ref{i:p:GlobCSL:1} amounts to the \emph{strict locality} of the Dirichlet space~$(\mbbX,\mcE)$ in the sense of~\cite[Dfn.~3.19]{LzDSSuz20}.
In this case, the assumption in Proposition~\ref{p:GlobCSL}\ref{i:p:GlobCSL:2} holds as soon as~$(X,\mssd_\mssm)$ is locally complete, see~\cite[Thm.~3.24]{LzDSSuz20}.
\end{remark}

\begin{remark}
When~$\mssd=\mssd_\mu$, it would not be difficult to show that Proposition~\ref{p:GlobCSL} (even without the assumption in~\ref{i:p:GlobCSL:2}) holds as well when, e.g.,~$(\cSL{\T}{\mu}{\mssd_\mu})$ is replaced by~$(\Rad{\mssd_\mu}{\mu})$ and~$(\cSL{\T_E}{\mu_E}{\mssd_\mu})$ is replaced by~$(\Rad{\mssd_\mu}{\mu_E})$.
However, this new statement should not be interpreted as a globalization result for the Rademacher property.
Rather, since~$\mbbX$ satisfies~\ref{ass:Luzin}, the topology~$\T=\T_{\mssd_\mu}$ is separable.
Thus,~$(\Rad{\mssd_\mu}{\mu})$ is simply \emph{verified} for~$(\mbbX,\mcE,\mssd_\mu,\mu)$ ---that is: \emph{independently} of the Rademacher property for~$(\mbbX_E,\mcE_E,\mssd_\mu,\mu_E)$--- by virtue of~\cite[Cor.~3.15]{LzDSSuz20}.
\end{remark}

\section{Weighted spaces}\label{s:Transfer}
In this section, let~$\mbbX$ and~$\mbbX'$ be structures with same underlying topological measurable space~$(X,\T,\A)=(X',\T',\A')$ and different measures~$\mssm$ and~$\mssm'$ with~$\mssm\sim\mssm'$.
Further let~$\mssd\colon X^\tym{2}\to [0,\infty]$ be an extended distance.

Now, let~$(\mcE,\dom)$, resp.~$(\mcE',\dom')$, be a quasi-regular strongly local Dirichlet space with underlying space~$\mbbX$, resp.~$\mbbX'$, defined on~$L^2(\mssm)$, resp.~$L^2(\mssm')$, and admitting carr\'e du champ operator~$\cdc$, resp.~$\cdc'$.
We write
\begin{equation*}
\tparen{\cdc, \msA}\leq \tparen{\cdc',\msA'}
\end{equation*}
to indicate that~$\msA\supset \msA'$ and~$\cdc'\geq \cdc$ on~$ \msA'$, and analogously for the opposite inequality.

\subsection{Rademacher and Sobolev-to-Lipschitz properties for weighted spaces}
Let~$(\mssP)$ denote either~$(\Rad{\mssd}{\mssm})$, $(\ScL{\mssm}{\T}{\mssd})$, $(\SL{\mssm}{\mssd})$, or~$(\cSL{\T}{\mssm}{\mssd})$.
Note that $\Lipu(\mssd,\A)$ and~$\Cb(\T)$ do not depend on~$(\mcE,\dom)$ nor on~$\mssm$.
Furthermore, since~$\mssm\sim\mssm'$, we have that~$L^\infty(\mssm)= L^\infty(\mssm')$ as Banach spaces.
As a consequence of the facts above,~$(\mssP)$ is a \emph{local} property in the sense of the following Proposition, a proof of which is straightforward.

\begin{proposition}[Weighted spaces]\label{p:Locality}
Retain the notation above. Then,
\begin{enumerate}[$(i)$]
\item\label{i:p:Locality:1} if $\tparen{\cdc, \DzLocB{\mssm}}\leq \tparen{\cdc',\DzLocBprime{\mssm'}}$ and~$(\ScL{\mssm}{\T}{\mssd})$, resp.~$(\SL{\mssm}{\mssd})$,~$(\cSL{\T}{\mssm}{\mssd})$ holds,
then $(\ScL{\mssm'}{\T}{\mssd})$, resp.~$(\SL{\mssm'}{\mssd})$,~$(\cSL{\T}{\mssm'}{\mssd})$ holds as well;

\item\label{i:p:Locality:2} if $\tparen{\cdc, \DzLocB{\mssm}}\geq \tparen{\cdc',\DzLocBprime{\mssm'}}$ and~$(\Rad{\mssd}{\mssm})$ holds, then~$(\Rad{\mssd}{\mssm'})$ holds as well.
\end{enumerate}
\end{proposition}

Let us now show that it suffices to verify the assumptions in Proposition~\ref{p:Locality} only on a common pseudo-core (i.e.\ a linear subspace dense in both domains).
In the statement of the next result, let~$\msG_0\eqdef \msG_0^\mcE$ and~$\msG_0'\eqdef \msG_0^{\mcE'}$ be defined as in~\eqref{eq:Nests}.
 
\begin{proposition}[Comparison of square fields]\label{p:LocalityProbab}
Let~$(\mbbX,\mcE)$ and~$(\mbbX',\mcE')$ be quasi-regular strongly local Dirichlet spaces with same underlying topological measurable space~$(X,\T,\A)=(X',\T',\A')$ and {possibly} different measures~$\mssm$ and~$\mssm'$, both satisfying~\ref{ass:Luzin}.
Further assume that:
\begin{enumerate}[$(a)$]
\item\label{i:p:LocalityProbab:3} $\mssm'=\theta\mssm$ for some~$\theta\in L^0(\mssm)$, and there exists $E_\bullet, G_\bullet\in\msG_0\cap \msG_0'$ with the following properties:
\begin{enumerate}[$({a}_1)$]
\item\label{i:p:LocalityProbab:3.1} for each~$k\in \N$ there exists a constant~$a_k$ such that $\theta\geq a_k>0$ $\mssm$-a.e.\ on~$G_k$;
(that is,~$\theta^{-1}\in \dotloc{\tparen{L^\infty(\mssm)}}(G_\bullet)$.)
\item\label{i:p:LocalityProbab:2} for each~$k\in \N$ it holds~$E_k\subset G_k$ $\mcE$- and $\mcE'$-quasi-everywhere, and there exists~$\varrho_k\in \dom$ with~$\car_{E_k}\leq \varrho_k\leq \car_{G_k}$ $\mssm$-a.e., and~$\cdc(\varrho_k)\in L^1(\mssm'_{G_k})$;
\end{enumerate}

\item\label{i:p:LocalityProbab:4} there exists~$\mcD$ a pseudo-core for both~$(\mcE,\dom)$ on~$L^2(\mssm)$ and~$(\mcE',\dom')$ on~$L^2(\mssm')$, additionally so that~$\cdc\leq \cdc'$ on~$\mcD$.
\end{enumerate}
Then,~$\tparen{\cdc, \DzLocB{\mssm}}\leq \tparen{\cdc',\DzLocBprime{\mssm'}}$ and~$\mssd_\mssm\geq \mssd_{\mssm'}$.
\end{proposition}

\begin{remark}\label{r:CutOff}
\begin{enumerate}[$(a)$, wide]
\item Assumption~\ref{i:p:LocalityProbab:2} in Proposition~\ref{p:LocalityProbab} is implied by the following: 
\begin{itemize}
\item[$(a_2')$] for each~$k\in\N$ it holds~$E_k\subset G_k$ $\mcE$- and $\mcE'$-quasi-everywhere and there exists~$\varrho_k\in\mcD$ with~$\car_{E_k}\leq \varrho_k\leq \car_{G_k}$ $\mssm$-a.e.; 
\end{itemize}
indeed, in this case~$\cdc(\varrho_k)\leq \cdc'(\varrho_k)\in L^1(\mssm')$ by assumption~\ref{i:p:LocalityProbab:4} in Proposition~\ref{p:LocalityProbab} and since~$\mcD\subset \dom'$, and we may thus choose~$C_k\eqdef \cdc'(\varrho_k)$.
Depending on the choice of~$\mcD$, assumption~$(a_2')$ is possibly more easily verified; in particular, it is immediately satisfied if~$\mcD=\dom'$.

\item The existence of the cut-off functions in~Proposition~\ref{p:LocalityProbab}\ref{i:p:LocalityProbab:2} is a quite mild assumption.
For instance, let~$\mssd\colon X^\tym{2}\to [0,\infty)$ be a distance metrizing~$\T$ and assume that~$(\mbbX,\mcE)$ satisfies~$(\Rad{\mssd}{\mssm})$.
We say that~$A_1$, $A_2\subset X$ are $\T$-\emph{well-separated} if~$\cl_\T A_1\cap \cl_\T A_2=\emp$.
Then it is readily verified that, for \emph{every} pair of sets~$E,G$ with
\begin{itemize}
\item~$E\subset G$, and~$G$ (hence~$E$) $\mssd$-bounded and of finite $\mssm'$-measure,
\item~$E$ and~$G^\complement$ $\T$-well-separated,
\end{itemize}
the function~$\varrho_{E,G}\eqdef 0\vee \tparen{1-\mssd(E,G^\complement)^{-1}\mssd(\emparg,E)}$ satisfies
\[
\car_{\cl_\T E} \leq \varrho_k \leq \car_{\cl_\T G}\comma \qquad \cdc(\varrho_{E,G})\leq \mssd(E,G^\complement)^{-1}\car_{\cl_\T G}\in L^1(\mssm')\fstop
\]
\end{enumerate}
\end{remark}

\begin{proof}[Proof of Proposition~\ref{p:LocalityProbab}]
Let~$\varrho_k$ be given by~\ref{i:p:LocalityProbab:2}.
Since~$G_\bullet\in\msG_0$ (resp.~$G_\bullet\in\msG_0'$), we have in particular that~$\mssm G_k<\infty$ (resp.~$\mssm'G_k<\infty$) for every~$k\in \N$.
Thus,~$\varrho_k\in L^p(\mssm)\cap L^p(\mssm')$ for every every~$p\in [1,\infty]$ by interpolation, for every~$k\in\N$.

Now, let~$(G'_\bullet, f_\bullet)$, with~$G'_\bullet \in\msG_0'$, be witnessing that~$f\in \DzLocBprime{\mssm'}$.
With no loss of generality, by truncation, we may assume that~$\sup_n\norm{f_n}_{L^\infty(\mssm)}\leq M\eqdef \norm{f}_{L^\infty(\mssm)}<\infty$.
In light of Lemma~\ref{l:2.26} and since~$G_\bullet\in\msG_0\cap\msG_0'$, we may assume with no loss of generality that~$G'_\bullet= G_\bullet$.
For each~$k\in\N$, since~$f_k\in\dom'$, there exists~$\seq{f^\sym{n}_k}_n\subset \D\subset\dom'$ a sequence of functions $\dom'$-converging to~$f_k$.
With no loss of generality, up to replacing~$\seq{f^\sym{n}_k}_n$ with a non-relabeled subsequence, we may further assume that~$\nlim f^\sym{n}_k=f_k$ $\mssm'$- (hence~$\mssm$-)a.e..
By a standard truncation argument, by the Markov property and by the closability of~$(\mcE',\dom')$, we may finally also assume that~$\abs{f^\sym{n}_k}\leq \abs{f_k}\leq M$ $\mssm'$-a.e.\ for every~$k,n\in \N$.
As a consequence, by Dominated Convergence with dominating function~$M\car_{G_k}\in L^2(\mssm')$ we have in particular that
\begin{equation}\label{eq:p:LocalityProbab:1}
L^2(\mssm')\text{-}\nlim f^\sym{n}_k\varrho_k= f\varrho_k\fstop
\end{equation}
Furthermore, by the Leibniz rule for~$\cdc$ and by the assumption, we have that~$\seq{f^\sym{n}_k\varrho_k}_n\subset \domb$ satisfies
\begin{align*}
\cdc(f^\sym{n}_k\varrho_k-f^\sym{m}_k\varrho_k) =&\ \abs{f^\sym{n}_k-f^\sym{m}_k}^2 \cdc(\varrho_k) + \abs{\varrho_k}^2 \cdc(f^\sym{n}_k-f^\sym{m}_k) 
\\
\leq&\ \abs{f^\sym{n}_k-f^\sym{m}_k}^2\cdc(\varrho_k)+\car_{G_k}\cdc'(f^\sym{n}_k-f^\sym{m}_k)
\fstop
\end{align*}
Since~$\varrho_k\equiv 0$ $\mcE$-q.e.\ on~$G_k^\complement$, and since~$G_k$ is $\mcE$-quasi-open, by~\eqref{eq:SLoc:2} we conclude that
\begin{align}
\label{eq:p:LocalityProbab:2}
\cdc(f^\sym{n}_k\varrho_k-f^\sym{m}_k\varrho_k)
\leq&\ \car_{G_k}\tparen{\abs{f^\sym{n}_k-f^\sym{m}_k}^2\cdc(\varrho_k)+ \cdc'(f^\sym{n}_k-f^\sym{m}_k)} \fstop
\end{align}

Now, since~$\theta\geq a_k>0$ on~$G_k$, the $L^2(\mssm'\mrestr{G_k})$-convergence implies the $L^2(\mssm\mrestr{G_k})$-convergence.
Thus, there exists
\begin{equation}\label{eq:p:LocalityProbab:2.5}
L^2(\mssm)\text{-}\nlim f^\sym{n}_k\varrho_k = L^2(\mssm')\text{-}\nlim f^\sym{n}_k\varrho_k = f_k\varrho_k = f\varrho_k\comma
\end{equation}
where the last equality holds by definition of~$(G_\bullet,f_\bullet)$.
Furthermore, integrating~\eqref{eq:p:LocalityProbab:2} we see that
\begin{equation}\label{eq:p:LocalityProbab:2.75}
\begin{aligned}
\lim_{n,m} &\int \cdc(f^\sym{n}_k\varrho_k-f^\sym{m}_k\varrho_k) \diff\mssm
\\
&\leq a_k^{-1}\lim_{n,m}\int_{G_k} \abs{f^\sym{n}_k-f^\sym{m}_k}^2\cdc(\varrho_k) \diff\mssm' +\lim_{n,m}\mcE'(f^\sym{n}_k-f^\sym{m}_k)=0\comma
\end{aligned}
\end{equation}
where the first limit in the right-hand side vanishes by Dominated Convergence with dominating function~$4M^2 \cdc(\varrho_k)$, and the second vanishes by definition of~$\seq{f^\sym{n}_k}_n$.

Combining~\eqref{eq:p:LocalityProbab:2.5} and~\eqref{eq:p:LocalityProbab:2.75} implies the existence of
\begin{align*}
\dom\text{-}\nlim f^\sym{n}_k\varrho_k=f\varrho_k\in\domb \fstop
\end{align*}
Thus, since~$G_\bullet$ is increasing to~$X$ $\mcE$-quasi-everywhere, and~$\seq{f\varrho_k}_k\subset \dom$ is a sequence of functions satisfying~$f\varrho_k\equiv f$ on~$G_k$, then~$f\in\dotloc{\dom}$ by definition of broad local space.

Since~$\varrho_k\equiv 1$ on~$E_k$, we have~$\car_{E_k}\cdc(\varrho_k)\equiv 0$ by~\eqref{eq:SLoc:2}. Thus, again by the Leibniz rule and by the assumption,
\begin{align*}
\car_{E_k}\cdc(f^\sym{n}_k\varrho_k)= \car_{E_k}\abs{f^\sym{n}_k}^2\cdc(\varrho_k)+\car_{E_k}\abs{\varrho_k}^2\cdc(f^\sym{n}_k)=\car_{E_k} \cdc(f^\sym{n}_k)\leq \car_{E_k}  \cdc'(f^\sym{n}_k)\comma
\end{align*}
and, taking the limit as~$n\to\infty$ (possibly, up to choosing a suitable non-relabeled subsequence),
\begin{equation}\label{eq:p:LocalityProbab:3}
\begin{aligned}
\car_{E_k} \cdc(f\varrho_k)=& \car_{E_k} \cdc(f_k\varrho_k)= \car_{E_k} \nlim\cdc(f^\sym{n}_k\varrho_k)\leq \car_{E_k} \nlim \cdc'(f^\sym{n}_k)
\\
=& \car_{E_k}  \cdc'(f_k)=\car_{E_k}\cdc'(f) \leq 1
\end{aligned}
\quad\as{\mssm}\comma
\end{equation}
where the last equality holds by locality~\eqref{eq:SLoc:2} of~$\cdc$ and definition of~$(G_\bullet,f_\bullet)$.
Again by~\eqref{eq:SLoc:2}, and by~\eqref{eq:p:LocalityProbab:3},
\begin{align*}
\car_{E_k} \cdc(f) =\car_{E_k} \cdc(f\varrho_k)\leq \car_{E_k}  \cdc'(f) \as{\mssm}\comma
\end{align*}
hence, letting~$k\to\infty$ and since~$\mssm\tparen{ \cap_k E_k^\complement}=0$,
\begin{align*}
\cdc(f) \leq \cdc'(f) \leq 1 \as{\mssm}\comma
\end{align*}
which also shows that~$f\in\DzLocB{\mssm}$, and thus concludes the proof of the first assertion.

The assertion on the intrinsic distances is an immediate consequence of the first assertion together with the definition~\eqref{eq:IntrinsicD} of intrinsic distance.
\end{proof}

Simmetrizing the assumptions in Proposition~\ref{p:LocalityProbab} and combining it with Proposition~\ref{p:Locality}, we obtain the following.

\begin{corollary}[Mutual implications for weighted spaces]\label{c:LocalityDistances}
Let~$(\mbbX,\mcE)$ and~$(\mbbX',\mcE')$ be quasi-regular strongly local Dirichlet spaces with same underlying topological measurable space~$(X,\T,\A)=(X',\T',\A')$ satisfying~\ref{ass:Luzin} and different measures~$\mssm$ and~$\mssm'$.
Further let~$\mssd\colon X^\tym{2}\to [0,\infty]$ be an extended pseudo-distance.
Assume that:
\begin{enumerate}[$(a)$]
\item\label{i:c:LocalityDistances:4} there exists~$\mcD$ a pseudo-core for both~$(\mcE,\dom)$ on~$L^2(\mssm)$ and~$(\mcE',\dom')$ on~$L^2(\mssm')$, additionally so that~$\cdc=\cdc'$ on~$\mcD$.

\item\label{i:c:LocalityDistances:3} $\mssm'=\theta\mssm$ for some~$\theta\in L^0(\mssm)$, and there exists $E_\bullet, G_\bullet\in\msG_0\cap \msG_0'$ with the following properties:
\begin{enumerate}[$({b}_1)$]
\item\label{i:c:LocalityDistances:1} for each~$k\in \N$ there exists a constant~$a_k>0$ such that
\[
0<a_k\leq \theta \leq a_k^{-1}<\infty  \as{\mssm} \quad \text{on } G_k\semicolon
\]
that is,~$\theta, \theta^{-1}\in \dotloc{\tparen{L^\infty(\mssm)}}(G_\bullet)$.
\item\label{i:c:LocalityDistances:2} for each~$k\in \N$ it holds~$E_k\subset G_k$ $\mcE$- and $\mcE'$-quasi-everywhere, and there exists~$\varrho_k\in \mcD$ with~$\car_{E_k}\leq \varrho_k\leq \car_{G_k}$ $\mssm$-a.e..
\end{enumerate}
\end{enumerate}
Then,
\begin{enumerate}[$(i)$]
\item $\tparen{\cdc, \DzLocB{\mssm}}= \tparen{\cdc',\DzLocBprime{\mssm'}}$ and $\mssd_\mssm= \mssd_{\mssm'}$;
\item $(\ScL{\mssm}{\T}{\mssd})$, resp.\ $(\SL{\mssm}{\mssd})$, $(\cSL{\T}{\mssm}{\mssd})$, $(\Rad{\mssd}{\mssm})$, holds if and only if $(\ScL{\mssm'}{\T}{\mssd})$, resp.\ $(\SL{\mssm'}{\mssd})$, $(\cSL{\T}{\mssm'}{\mssd})$, $(\Rad{\mssd}{\mssm'})$ holds.
\end{enumerate}
\end{corollary}

\subsection{Form comparison}
In this section we provide a full comparison of a Dirichlet form~$(\mcE,\dom)$ with a Cheeger energy~$\Ch[*,\mssd,\mssm]$ defined on the same space.
Roughly speaking, we show that~$\mcE$ is dominated by $\Ch[*,\mssd,\mssm]$ under the Rademacher property for~$\mssd$, and that the reverse domination holds under the continuous-Sobolev--to--Lipschitz property for~$\mssd$ and the \emph{$\T$-upper regularity} property for~$(\mcE,\dom)$ (Dfn.~\ref{d:TUpperReg}).

\subsubsection{Form comparison under the Rademacher property}
The next Proposition~\ref{p:IneqCdC} is an extension of the same result obtained for \emph{energy-measure spaces} by L.~Ambrosio, N.~Gigli, and G.~Savar\'e in~\cite{AmbGigSav15}, and for extended metric-topological \emph{probability} spaces by L.~Ambrosio, M.~Erbar, and G.~Savar\'e in~\cite{AmbErbSav16}.
In the proof we make use of assertions proven in~\cite{AmbErbSav16} for extended metric-topological probability spaces, and of assertions proven in~\cite{AmbGigSav14}.
The compatibility between different definitions in these references is granted by Proposition~\ref{p:ConsistencyCheeger}.

\begin{lemma}\label{l:AGS}
Let~$(X,\T,\mssd)$ be an extended metric-topological space (Dfn.~\ref{d:AES}) and~$K\subset X$ be $\T$-compact.
Then,~$\mssd(\emparg, K)$ is $\T$-lower semicontinuous (in particular, $\Bo{\T}$-measurable).
\end{lemma}

\begin{remark}
When~$(X,\T)$ is additionally Polish, Lemma~\ref{l:AGS} is claimed without proof in (the proof of)~\cite[Lem.~4.11]{AmbGigSav14}.
We provide a proof for completeness and for reference.
\end{remark}

\begin{proof}[Proof of Lemma~\ref{l:AGS}]
By~\cite[Eqn.~(4.3)]{AmbErbSav16} the $\T$-admissible extended distance~$\mssd$ is $\T^\tym{2}$-lower semicontinuous.
Let~$\seq{x_\alpha}_\alpha\subset X$ be any net $\T$-converging to~$x\in X$.
For each~$\eps>0$ there exists~$y_{\alpha(\eps)}\in K$ so that~$\mssd(x_\alpha,K)\geq \mssd(x_\alpha,y_{\alpha(\eps)})-\eps$.
By $\T$-compactness of~$K$, for each~$\eps>0$ there exists~$y_\eps\in K$ a $\T$-accumulation point of~$\seq{y_{\alpha(\eps)}}_{\alpha(\eps)}$.
Thus, by the above inequality, by $\T^\tym{2}$-lower-semicontinuity of~$\mssd$, and since~$y_\eps\in K$,
\begin{align*}
\liminf_\alpha \mssd(x_\alpha,K)\geq \liminf_{\alpha(\eps)}\liminf_\alpha \mssd(x_\alpha,y_{\alpha(\eps)})-\eps \geq \mssd(x,y_\eps) -\eps \geq \mssd(x,K) -\eps \comma
\end{align*}
which concludes the assertion by arbitrariness of~$\eps>0$.
\end{proof}

The following is a rewriting of~\cite[Lem.~4.11]{AmbGigSav14}. Whereas our assumptions are milder, the proof in~\cite{AmbGigSav14} applies with minor modifications.

\begin{lemma}[{\cite[Lem.~4.11]{AmbGigSav14}}]\label{l:AGS2}
Let~$(X,\T,\mssd)$ be a complete extended metric-topological measure space, and~$\mssm$ be a $\sigma$-finite $\T$-Radon measure on~$\Bo{\T}$.
Further let~$\mssm'\eqdef \theta \mssm$ be another $\sigma$-finite measure on~$\Bo{\T}$ with density~$\theta$ satisfying the following condition: there exists a sequence of $\T$-compact sets~$K_i$ with~$K_i\subset K_{i+1}$, and constants~$r_i,c_i,C_i>0$ such that
\begin{equation}\label{eq:l:AGS2:0}
\mssm \tparen{\cap_i K_i^\complement}=0 \qquad \text{and} \qquad 0<c_i\leq \theta \leq C_i <\infty \quad \as{\mssm} \text{ on } \set{\mssd(\emparg, K_i) < r_i} \fstop
\end{equation}
Then, the relaxed gradient~$\slo[*,\mssd,\mssm']{f}$ induced by~$\mssm'$ coincides $\mssm$-a.e.\ with~$\slo[*,\mssd,\mssm]{f}$ for every~$f\in \domain{\Ch[*,\mssd,\mssm']}\cap \domain{\Ch[*,\mssd,\mssm]}$.
If moreover there exists a single~$r>0$ such that~\eqref{eq:l:AGS2:0} holds with~$r$ in place of~$r_i$, then
\[
f\in \domain{\Ch[*,\mssd,\mssm]}\comma \quad f\comma \slo[*,\mssd,\mssm]{f} \in L^2(\mssm') \implies f\in \domain{\Ch[*,\mssd,\mssm']} \fstop
\]
\end{lemma}
\begin{proof}
Since~$\mssm$ is $\sigma$-finite $\T$-Radon, there exists a sequence of $\T$-compact sets~$K_i$ satisfying~$\mssm\tparen{\cap_i K_i^\complement}$.
The rest of the proof holds exactly as in~\cite{AmbGigSav14} having care to substitute the $\T$-compact set~$K$ there with~$K_i$ such that~$\mssm (K\cap K_i)>0$ which exists since~$\seq{K_i}_i$ exhausts~$X$ up to an $\mssm$-negligible set.
\end{proof}

\begin{proposition}\label{p:IneqCdC}
Let~$(\mbbX,\mcE)$ be a quasi-regular strongly local Dirichlet space with~$\mbbX$ satisfying~\ref{ass:Luzin} and admitting carr\'e du champ operator, and $\mssd\colon X^\tym{2}\to [0,\infty)$ be an extended distance.
Further assume that
\begin{enumerate}[$(a)$]
\item\label{i:p:IneqCdC:A} $(X,\T,\mssd)$ is a complete extended metric-topological space (Dfn.~\ref{d:AES}); 
\item\label{i:p:IneqCdC:B} $\mssm$ is \emph{$\mssd$-moderate} in the following sense: there exists a $\T$-compact $\mcE$-nest~$\seq{K_i}_i$ and~$\eps>0$ such that
\begin{equation}\label{eq:p:IneqCdC:00}
\kappa_i\eqdef \mssm\set{\mssd(\emparg, K_i)< \eps} <\infty\fstop
\end{equation}
\end{enumerate}
If $(\mbbX,\mcE)$ satisfies~$(\Rad{\mssd}{\mssm})$ for~$\Bo{\T}$, then
\begin{equation}\label{eq:p:IneqCdC:0}
\cdc(f)\leq \slo[w,\mssd_{\mssm}]{f}^2 \leq \slo[w,\mssd]{f}^2 \quad \as{\mssm}\comma \qquad f\in \bLip(\mssd,\Bo{\T})\fstop
\end{equation}
In particular,
\begin{equation}\label{eq:p:IneqCdC:001}
\mcE\leq \Ch[w,\mssd_{\mssm},\mssm]\leq \Ch[w,\mssd,\mssm]\fstop
\end{equation}
\end{proposition}

\begin{remark}\label{r:IneqCdC}
The existence of a $\T$-compact $\mcE$-nest in~\ref{i:p:IneqCdC:B} is always satisfied as a consequence of the quasi-regularity of~$(\mbbX,\mcE)$.
The second condition is a form of \emph{moderance} of~$\mssm$ ---in the sense of~\cite[\S2.5.2]{LzDSSuz20}--- compatible with~$\mssd$.
Together with the assumption of the Rademacher property~$(\Rad{\mssd}{\mssm})$, it is tantamount to the $\mssm$-uniform algebraic $\T$-localizability (Dfn.~\ref{d:Localizability}) of~$(\mbbX,\mcE)$.
Consistently with Remark~\ref{r:AlgLocTrivial}, the assumption in~\eqref{eq:p:IneqCdC:00} is trivially satisfied whenever~$\mssm X<\infty$, which is why this assumption does not appear in~\cite{AmbErbSav16}, addressing only probability spaces.
\end{remark}

\begin{proof}[Proof of Proposition~\ref{p:IneqCdC}]
By~\eqref{eq:EquivalenceRadStoL} the space~$(\mbbX,\mcE)$ satisfies~$(\dRad{\mssd}{\mssm})$, which implies the second inequality by definition of the objects involved.
Thus, it suffices to show the first inequality.
Firstly, let us note that~$(X,\T,\mssd_{\mssm})$ is a complete extended metric-topological space by Lemma~\ref{l:RadCompleteness}.
Suppose for the moment that~$\mssm$ be a probability measure.
Since~$(X,\T)$ is a topological Luzin space, it is a Radon space.
In particular,~$\mssm$ is Radon.
Then, we can apply~\cite[Thm.~12.5]{AmbErbSav16} to the complete metric-topological Radon probability space~$(X,\T,\mssd,\mssm)$ and conclude the assertion.

\paragraph{Heuristics}
Now, let us show how to extend the statement to the case of any $\sigma$-finite measure~$\mssm$.
We need to treat simultaneously the square field operator~$\cdc$ and the minimal weak upper gradient~$\slo[w,\mssd_\mssm,\mssm]{\emparg}$.
To this end, we find a probability density~$\theta\in L^1(\mssm)$, set~$\mssm'\eqdef \theta\mssm$, and show that~$\cdc=\cdc'$ and~$\slo[w,\mssd_\mssm,\mssm]{\emparg}=\slo[w,\mssd_\mssm,\mssm']{\emparg}$.
The conclusion follows from these equalities together with the inequality established in the probability case applied to~$\mssm'$.

For the square field, we make use of the result on the Girsanov transform of~$(\mcE,\dom)$ by a factor~$\sqrt{\theta}\in\dom$, thoroughly discussed in the generality of quasi-regular Dirichlet spaces by C.-Z.~Chen and W.~Sun in~\cite{CheSun06}.
For the minimal weak upper gradient, we make use of the locality result for~$\slo[w,\mssd_\mssm,\mssm]{\emparg}$ under transformation of the reference measure by a factor~$\theta$ locally bounded away from~$0$ and infinity on neighborhoods of compact sets, Lemma~\ref{l:AGS2}.

\paragraph{Reduction}
Since~$\mssm X=\infty$ by assumption, there exists~$i_*$ such that~$\mssm K_i\geq 1$ for every~$i\geq i_*$.
Thus, up to discarding the first~$i_*$ elements of~$\seq{K_i}_i$, we may and will assume with no loss of generality that~$\kappa_i\geq1$ for every~$i$.
As a consequence,~$\kappa_i^{-1/2}\leq 1$ for every~$i\in\N$.
We shall further assume that~$\eps\leq 1$, again with no loss of generality up to replacing~$\eps$ by~$\eps\vee 1$.
Throughout the proof we let~$(K_i)_\delta\eqdef \set{\mssd(\emparg,K_i)<\delta}$ for all~$\delta>0$.

\paragraph{Construction of a density}
We start by showing that there exists a $\Bo{\T}$-measurable~$\theta\colon X\to \R$ satisfying:
\begin{equation}\label{eq:IneqCdC:0.1}
\theta\in L^1(\mssm)\cap L^\infty(\mssm)\comma\qquad \norm{\theta}_{L^1(\mssm)}=1\comma\qquad \theta>0 \as{\mssm}\comma\qquad \sqrt\theta\in \dom\fstop
\end{equation}
Let~$\seq{K_i}_i$ and~$\eps$ be as in~\eqref{eq:p:IneqCdC:00}, and~$S\colon [0,\infty]\to [0,1]$ be defined by
\begin{equation}\label{eq:MultiTrunc}
S(t)\eqdef
\begin{cases}
-\frac{1}{2}t+1 & \text{if } 0\leq t\leq 1
\\
\frac{1}{2} & \text{if } 1\leq t \leq 2
\\
-\frac{1}{2}t+\frac{3}{2} & \text{if } 2\leq t\leq 3
\\
0 & \text{if } 3\leq t\leq +\infty
\end{cases} \fstop
\end{equation}

Further set (cf.\ Fig.s~\ref{fig:FunctionPsi} and~\ref{fig:Phi3} below)
\[
\psi_i\eqdef S\tparen{\tfrac{3}{\eps}\,\mssd(\emparg, K_i)}\comma \qquad f_i\eqdef \frac{\eps}{3\sqrt{\kappa_i}}\,\psi_i \comma \qquad \varphi_n\eqdef \sum_{i=0}^n 2^{-i-1} f_i \fstop
\]
All the above functions are $\Bo{\T}$-measurable (in fact: $\T$-upper semicontinuous) by Lemma~\ref{l:AGS}.

\begin{figure}[htb!]
\includegraphics{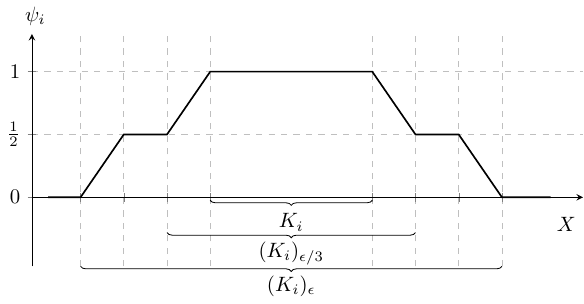}
\caption{The functions~$\psi_i$.}
\label{fig:FunctionPsi}
\end{figure}

\begin{figure}
\includegraphics{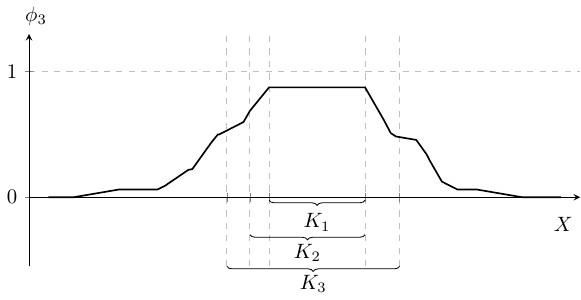}
\caption{The function~$\varphi_3$, taking value~$\tfrac{1}{2}+\tfrac{1}{4}+\tfrac{1}{8}$ on~$K_1$.}
\label{fig:Phi3}
\end{figure}

Since $\kappa_i\eqdef \mssm \set{\mssd(x,K_i)<\eps}<\infty$ we have that~$f_i\in L^2(\mssm)$ and~$\norm{f_i}_{L^2(\mssm)}\leq \tfrac{\eps}{3}\leq 1$ for every~$i$.
Thus, the sequence~$\seq{\varphi_n}_n$ satisfies
\begin{align*}
\norm{\varphi_n}_{L^2(\mssm)}\leq 1\comma \qquad  L^2(\mssm)\text{-}\nlim \varphi_n= \sum_i^\infty 2^{-i-1} f_i \defeq \varphi\in L^2(\mssm)
\end{align*}
by Monotone Convergence.
Furthermore,~$f_i\in \bLip(\mssd,\Bo{\T})$ and~$\Li[\mssd]{f_i}\leq (\kappa_i)^{-1/2}\leq 1$ for every~$i$ by definition, hence
\begin{align}\label{eq:IneqCdC:0.15}
\varphi_n\in \bLip(\mssd,\Bo{\T}) \qquad \text{and} \qquad \Li[\mssd]{\varphi_n}\leq 1
\end{align}
for every~$n$ by triangle inequality for the Lipschitz semi-norm~$\Li[\mssd]{\emparg}$.

As a consequence,~$\seq{\varphi_n}_n$ is uniformly bounded in~$\dom$ by~$(\Rad{\mssd}{\mssm})$ and thus~$\varphi\in \dom$ by e.g.~\cite[Lem.~I.2.12]{MaRoe92}.
Since
\begin{equation}\label{eq:l:IneqCdC:0.25}
\varphi\restr_{(K_i)_{\eps/3}}\geq 2^{-i-1}f_i\restr_{(K_i)_{\eps/3}}=\frac{2^{-i-2}\eps}{3\sqrt{\kappa_i}}>0
\end{equation}
and since~$\seq{K_i}_i$ exhausts~$X$ up to an $\mssm$-negligible set, then~$\varphi>0$ $\mssm$-a.e.\ on~$X$.
Letting $\theta\eqdef \varphi^2/\norm{\varphi}_{L^2(\mssm)}^2$ shows the assertion in~\eqref{eq:IneqCdC:0.1}.

\paragraph{Dirichlet forms}
Now, let us set~$\mssm'\eqdef \theta\mssm\sim \mssm$.
Since~$\sqrt\theta\in\domb$, by~\cite[Thm.~2.2]{CheSun06} the form
\begin{align*}
\mcE'(f)\eqdef \int_X \cdc(f)\diff\mssm'\comma \qquad f\in\dom \comma 
\end{align*}
is closable, and its closure~$\tparen{\mcE',\dom'}$ is a Dirichlet form on~$L^2(\mssm')$ with square field operator~$\cdc'$ satisfying
\begin{align}\label{eq:l:IneqCdC:0.5}
\cdc'(f)=\cdc(f) \as{\mssm'}\comma \qquad f\in \dom\cap L^\infty(\mssm')=\domb\comma
\end{align}
and the equality extends to all~$f\in\domb'$ by $\dom'$-density of~$\dom$ in~$\dom'$.
In particular, since~$\mssm'$ is a finite measure,~$\domb'\supset \bLip(\mssd,\Bo{\T})$ by~$(\Rad{\mssd}{\mssm})$, and~\eqref{eq:l:IneqCdC:0.5} holds for every~$f\in \bLip(\mssd,\Bo{\T})$.

\paragraph{Minimal relaxed slopes}
On the complete extended metric-topological measure space $(X,\T,\mssd_\mssm,\mssm)$ we have that
\begin{equation}\label{eq:l:IneqCdC:1}
\slo[*,\mssd_\mssm,\mssm]{f}=\slo[w,\mssd_\mssm,\mssm]{f}\comma \qquad f\in\bLip(\mssd_\mssm,\Bo{\T})\fstop
\end{equation}
This readily follows from the locality of both~$\slo[*]{\emparg}$ and~$\slo[w]{\emparg}$ in the sense of e.g.~\cite[Prop.~4.8(b)]{AmbGigSav14} and~\cite[Eqn.~(2.18)]{AmbGigSav14b}.

Secondly, let us verify that the space~$(X,\T,\mssd_\mssm,\mssm)$ satisfies the assumptions of Lemma~\ref{l:AGS2}.
The fact that~$(X,\T,\mssd_\mssm)$ is an extended metric-topological space was already noted in the beginning of the proof.
Thus, it suffices to show that~$\theta$ as above satisfies~\eqref{eq:l:AGS2:0} with~$\mssd_\mssm$ in place of~$\mssd$ there.
In fact, since~$\theta\in L^\infty(\mssm)$, it suffices to show the existence of~$c_i$ and~$r_i$.
We show that there exists~$c_i$ satisfying~\eqref{eq:l:AGS2:0} with~$r_i\eqdef \eps/3$.
Since~$\mssd\leq \mssd_\mssm$ by~$(\Rad{\mssd}{\mssm})$ and~\eqref{eq:EquivalenceRadStoL}, we have that~$\set{\mssd_\mssm(\emparg,K_i)<\eps/3}\subset (K_i)_{\eps/3}$ for every~$i\in\N$.
By~\eqref{eq:l:IneqCdC:0.25}, it suffices to set~$c_i\eqdef \paren{\frac{2^{-i-2} \eps}{3\sqrt{\kappa_i}\norm{\varphi}_{L^2(\mssm)}}}^2$.

Now, applying Lemma~\ref{l:AGS2} to the probability measure~$\mssm'$, we have that
\begin{align}\label{eq:l:IneqCdC:2}
\slo[*,\mssd_\mssm,\mssm]{f}=\slo[*,\mssd_\mssm,\mssm']{f}\comma \qquad f\in \bLip(\mssd_\mssm,\Bo{\T})\supset\bLip(\mssd,\Bo{\T})\fstop
\end{align}

\paragraph{Intrinsic distances}
We claim that~$\mssd_\mssm=\mssd_{\mssm'}$.
Let us verify the assumptions of Corollary~\ref{c:LocalityDistances} with~$\mcD\eqdef\dom$.
Indeed,~$\dom$ is a (pseudo-)core for both~$(\mcE,\dom)$ and~$(\mcE',\dom')$ and~$\cdc=\cdc'$ on~$\dom$ by~\eqref{eq:l:IneqCdC:0.5}.
This verifies assumption~\ref{i:c:LocalityDistances:4} in Corollary~\ref{c:LocalityDistances}.
In light of~\eqref{eq:l:AGS2:0} we see that~$\theta$ satisfies the assumptions in Corollary~\ref{c:LocalityDistances}\ref{i:c:LocalityDistances:1} with~$E_k\eqdef \intE K_k$,~$G_k\eqdef \intE (K_k)_{\eps/3}$, and~$a_k\eqdef c_k$.
Indeed~$E_\bullet,G_\bullet\in\msG_0$ since~$K_\bullet$ is an $\mcE$-nest, and~$E_\bullet,G_\bullet\in\msG_0'$ since~$\dom$ is $\dom'$-dense in~$\dom$ by construction of~$(\mcE',\dom')$.
Now, let~$\varrho_k\eqdef \braket{1-\tfrac{3}{\eps}\mssd(\emparg, K_k)}_+$ and note that~$\varrho_k\in \bLip(\mssd,\Bo{\T})$ by Lemma~\ref{l:AGS}.
Thus,~$\cdc(\varrho_k)\leq \Li[\mssd]{\varrho_k}^2= 9\eps^{-2}$ by~$(\Rad{\mssd}{\mssm})$ and in fact~$\cdc(\varrho_k)\leq 9\eps^{-2}\, \car_{G_k}$ by~\eqref{eq:SLoc:2}.
Since~$\theta\geq c_k>0$ on~$(K_k)_{\eps/3}\supset G_k$, we conclude that~$\cdc(\varrho_k)\in L^1(\mssm)$ too.
Thus, the assumptions in Corollary~\ref{c:LocalityDistances}\ref{i:c:LocalityDistances:2} hold as well. 
This concludes the verification of the assumptions in Corollary~\ref{c:LocalityDistances}, and thus~$\mssd_\mssm=\mssd_{\mssm'}$.

\paragraph{Conclusion}
Respectively by~\eqref{eq:l:IneqCdC:0.5},~\cite[Thm.~12.5]{AmbErbSav16},~\eqref{eq:l:IneqCdC:1} and~$\mssd_\mssm=\mssd_{\mssm'}$,~\eqref{eq:l:IneqCdC:2}, and again~\eqref{eq:l:IneqCdC:1},
\begin{align*}
\cdc(f)=\cdc'(f)\leq \slo[w,\mssd_{\mssm'},\mssm']{f}^2 = \slo[*,\mssd_\mssm,\mssm']{f}^2 = \slo[*,\mssd_\mssm,\mssm]{f}^2 = \slo[w,\mssd_\mssm,\mssm]{f}^2\comma \qquad f\in\bLip(\mssd,\Bo{\T})\comma
\end{align*}
which concludes the proof.
\end{proof}

\subsubsection{Form comparison under the Sobolev-to-Lipschitz property}
The chain of inequalities~\eqref{eq:p:IneqCdC:0} in Proposition~\ref{p:IneqCdC} is completed by including the slope of a Lipschitz function, viz.
\begin{equation*}
\cdc(f)\leq \wslo[\mssd_{\mssm}]{f}^2 \leq \wslo[\mssd]{f}^2\leq \slo[\mssd]{f}^2 \quad \as{\mssm}\comma \qquad f\in \bLip(\mssd,\Bo{\T})\comma
\end{equation*}
where the last inequality holds by the very definition of the minimal weak upper gradient.
In the same spirit of duality as in \S\ref{s:LocGlob}, this suggests that the opposite chain of inequalities, viz.
\begin{equation}\label{eq:SLChain1}
\cdc(f)\geq \wslo[\mssd_{\mssm}]{f}^2 \geq \wslo[\mssd]{f}^2\geq \slo[\mssd]{f}^2 \quad \as{\mssm}\comma \qquad f\in \DzLocB{\mssm,\T}\comma
\end{equation}
may be satisfied if some Sobolev-to-Lipschitz-type property holds.

Indeed, under the assumption of~$(\cSL{\T}{\mssm}{\mssd})$---hence, in light of~\eqref{eq:EquivalenceRadStoL}, under~$(\SL{\mssm}{\mssd})$ as well---the second inequality is a consequence of~$(\dSL{\mssd}{\mssm})$, given by~\eqref{eq:EquivalenceRadStoL}, together with the definition of the objects involved.
The last inequality is in fact an equality, since the opposite inequality always holds, as detailed above.
Thus,~\eqref{eq:SLChain1} in fact reads
\begin{equation}\label{eq:SLChain2}
\cdc(f)\geq \wslo[\mssd_{\mssm}]{f}^2 \geq \wslo[\mssd]{f}^2= \slo[\mssd]{f}^2 \quad \as{\mssm}\comma \qquad f\in \DzLocB{\mssm,\T}\fstop
\end{equation}

In order to discuss the validity of~\eqref{eq:SLChain2} we shall further need the following definition, introduced by L.~Ambrosio, N.~Gigli, and G.~Savar\'e in~\cite[Dfn.~3.13]{AmbGigSav15} (also cf.~\cite[Dfn.~12.4]{AmbErbSav16} for the generality of extended metric-topological measure spaces).

\begin{definition}[$\T$-upper regularity]\label{d:TUpperReg}
Let~$(\mbbX,\mcE)$ be a (quasi-regular strongly local) Dirichlet space with~$\mbbX$ satisfying~\ref{ass:Luzin} and admitting carr\'e du champ operator~$\cdc$.
We say that~$(\mcE,\dom)$ is \emph{$\T$-upper regular} if for every $f$ in a pseudo-core of~$(\mcE,\dom)$ there exists~$\seq{f_n}_n \subset \domb\cap \Cb(\T)$ and a sequence of bounded $\T$-upper semi-continuous functions~$g_n\colon X\to \R$ such that
\begin{equation}\label{eq:d:TUpperReg:0}
L^2(\mssm)\text{-}\nlim f_n =f\comma \qquad \sqrt{\cdc(f_n)}\leq g_n \as{\mssm}\comma \qquad \nlimsup \int_X g_n^2\diff\mssm \leq \mcE(f) \fstop
\end{equation}
\end{definition}

\begin{remark}
It is noted in the proof of~\cite[Thm.~3.14]{AmbGigSav15} (after the approximation results for the asymptotic Lipschitz constant in~\cite[\S8.3]{AmbGigSav12B}), resp.\ in~\cite[Thm.~9.2]{AmbErbSav16}, that the Cheeger energy~$\Ch[w,\mssd,\mssm]$ of a metric measure space, resp.\ of an extended metric-topological probability space, is always $\T$-upper regular.
As a consequence, in comparing a general quasi-regular Dirichlet form~$(\mcE,\dom)$ with a Cheeger energy on the same (extended metric-topological) space, it is in fact necessary to discuss the $\T$-upper regularity of~$(\mcE,\dom)$.
\end{remark}

Under the assumption of $\T$-upper regularity, we now show the dual statement to Proposition~\ref{p:IneqCdC}.
Again, we adapt to the $\sigma$-finite case the corresponding result in~\cite{AmbErbSav16} for extended metric-topological probability spaces.

\begin{proposition}\label{p:IneqCdC2}
Let~$(\mbbX,\mcE)$ be a quasi-regular strongly local Dirichlet space with~$\mbbX$ satisfying~\ref{ass:Luzin} and admitting carr\'e du champ operator, and $\mssd\colon X^\tym{2}\to [0,\infty]$ be an extended distance.
Further assume that
\begin{enumerate}[$(a)$]
\item\label{i:p:IneqCdC2:B} $\DzLocB{\mssm,\T}$ generates~$\T$;
\item\label{i:p:IneqCdC2:C} $(\mbbX,\mcE,\mssm)$ satisfies~$(\Loc{\mssm,\T})$;
\item\label{i:p:IneqCdC2:D} $(\mcE,\dom)$ is $\T$-upper regular.
\end{enumerate}
If $(\mbbX,\mcE)$ satisfies~$(\cSL{\T}{\mssd}{\mssm})$, then
\begin{equation}\label{eq:p:IneqCdC2:0}
\cdc(f)\geq \slo[w,\mssd_{\mssm}]{f}^2 \geq \slo[w,\mssd]{f}^2 \quad \as{\mssm}\comma \qquad f\in \DzLocB{\mssm,\T}\fstop
\end{equation}
In particular,~$\mcE\geq \Ch[w,\mssd_{\mssm},\mssm]\geq \Ch[w,\mssd,\mssm]$.
\end{proposition}

\begin{remark}[Comparison with Proposition~\ref{p:IneqCdC}]\label{r:ComparisonSL}
As anticipated above, Proposition~\ref{p:IneqCdC2} ought to be understood as `dual' to Proposition~\ref{p:IneqCdC}.
In this respect:
\begin{enumerate}[$(a)$, wide]
\item\label{i:r:ComparisonSL:1} 
The assumption in Proposition~\ref{p:IneqCdC2}\ref{i:p:IneqCdC2:B} is one about the largeness of~$\DzLocB{\mssm,\T}$.
It is required to guarantee the validity of~\cite[Eqn.~(12.1b)]{AmbErbSav16} for the ---in our case: \emph{given}--- topology~$\T$.
In Proposition~\ref{p:IneqCdC} such largeness is granted directly by the Rademacher property.
Furthermore, let us note that this assumption is to be compared with the one in Proposition~\ref{p:IneqCdC}\ref{i:p:IneqCdC:A}.
Indeed, recall that~$\mssd_\mssm$ is always $\T^\tym{2}$-lower semicontinuous and $\T$-admissible, as noted after Definition~\ref{d:IntrinsicD}.
Now, since~$\DzLocB{\mssm,\T}$ generates~$\T$, it in particular separates points (since~$(X,\T)$ is Hausdorff), hence $\mssd_\mssm$ separates points as well.
Thus,~$(X,\T,\mssd_\mssm)$ is an extended metric-topological space (Dfn.~\ref{d:AES}), which translates into the assumption in Proposition~\ref{p:IneqCdC}\ref{i:p:IneqCdC:A} when~$\mssd$ is replace by~$\mssd_\mssm$.

\item\label{i:r:ComparisonSL:b} When~$\mbbX$ is additionally locally compact, the assumption in Proposition~\ref{p:IneqCdC2}\ref{i:p:IneqCdC2:B} may be directly replaced with
\begin{itemize}[$(a')$, leftmargin=5em]
\item $\mssd_\mssm$ separates points, i.e.\ it is an extended distance (as opposed to: extended \emph{pseudo}-distance).
\end{itemize}

\item As already commented in Remark~\ref{r:IneqCdC}, the assumption in Proposition~\ref{p:IneqCdC2}\ref{i:p:IneqCdC2:C} is dual to the one in Proposition~\ref{p:IneqCdC}\ref{i:p:IneqCdC:B}.

\item As for $\T$-upper regularity, we refer the reader to~\cite{AmbGigSav15,AmbErbSav16} for a detailed explanation of this condition.
As pointed out in~\cite{AmbGigSav15,AmbErbSav16}, the assumption of $\T$-upper regularity is \emph{necessary} to the validity of (the first inequality in~\eqref{eq:p:IneqCdC2:0} in) Proposition~\ref{p:IneqCdC2}.
\end{enumerate}
\end{remark}

\begin{proof}[Proof of Remark~\ref{r:ComparisonSL}\ref{i:r:ComparisonSL:b}] 
Since generating~$\T$ is a local property, it suffices to show the statement when~$(X,\T)$ is compact.
Since~$\mssd_\mssm$ separates points, so does~$\DzLocB{\mssm,\T}$.
Fix a $\T$-closed set $K\subset X$, and note that it is $\T$-compact.
For fixed~$x\in K^\complement$ and each~$y\in K$ let~$f_y\in \DzLocB{\mssm,\T}$ be separating~$x$ from~$y$.
Without loss of generality ---up to possibly changing the sign of~$f$ and adding to it a constant function in~$\DzLocB{\mssm,\T}$---, we may assume that~$f(x)=0$ and~$f(y)=a_y>0$.
(We may however \emph{not} directly assume that~$f(y)=1$, since~$\DzLocB{\mssm,\T}$ is not a linear space.)
Since each~$f_y$ is $\T$-continuous, the family~$\set{f_y> a_y/2}_{y\in K}$ is a $\T$-open cover of~$K$.
Let~$\seq{y_i}_{i\leq n}$ be defining a finite subcover and set~$f_{x,K}\eqdef \wedge_{i\leq n} f_{y_i}$.
Then,~$f_{x,K}\in \DzLocB{\mssm,\T}$,~$f_{x,K}(x)=0$, and~$f_{x,K}>0$ everywhere on~$K$; that is,~$f_{x,K}$ separates~$x$ from~$K$.
Since~$K$ was arbitrary,~$\DzLocB{\mssm,\T}$ separates points from $\T$-closed sets.
By standard arguments, any family of $\T$-continuous functions separating points from $\T$-closed sets generates~$\T$.
\end{proof}

\begin{proof}[Proof of Proposition~\ref{p:IneqCdC2}]
Since~$(\cSL{\T}{\mssm}{\mssd})$ implies~$(\dSL{\mssd}{\mssm})$ by~\eqref{eq:EquivalenceRadStoL}, the second inequality in~\eqref{eq:p:IneqCdC2:0} holds by definition of the objects involved.
In order to show the first inequality, we argue similarly to the proof of Proposition~\ref{p:IneqCdC}.

\paragraph{Construction of a density}
Let~$\seq{\theta_i}_i$ be a latticial approximation of the identity witnessing~$(\Loc{\mssm,\T})$. 
Without loss of generality, we may and will assume that~$\set{\theta_1\geq 3}\neq \emp$, hence that~$\set{\theta_i\geq 3}\neq \emp$ for every~$i\in\N$.
Further let~$S\colon [0,\infty]\to [0,1]$ be defined as in~\eqref{eq:MultiTrunc}, and set
\begin{equation*}
\psi_i\eqdef S\tparen{3-(\theta_i(\emparg)\wedge 3)}\comma \qquad i\in \N\fstop
\end{equation*}
Note that~$\psi_i$ has the same shape as in Figure~\ref{fig:FunctionPsi}, and, for each~$i\in\N$,
\begin{enumerate}[$(a)$]
\item\label{i:p:IneqCdC2:proof1} since~$\car \in \DzLocB{\mssm,\T}$ and~$\cdc(\car)\equiv 0$ by locality, and since~$\theta_i\in\DzB{\mssm,\T}$, we have~$\psi_i\in \DzB{\mssm,\T}\subset \dom$ by~\eqref{eq:ChainRuleLoc} and~\eqref{eq:SLoc:2}.
\item\label{i:p:IneqCdC2:proof2} since~$\theta_i\nearrow_i \infty$, the sets~$E_i\eqdef \inter_\T\set{\psi_i\equiv 1}$ form a $\T$-open exhaustion of~$X$;
\item\label{i:p:IneqCdC2:proof3} since~$\theta_i\nearrow_i \infty$ and~$\theta_i\in\Cb(\T)$ (hence~$\psi_i\in\Cb(\T)$), the sets~$G_i\eqdef \set{\psi_i>\tfrac{1}{2}}$ form a $\T$-open exhaustion of~$X$;
\item\label{i:p:IneqCdC2:proof4} $\cl_\T E_i\subset G_i$;
\item\label{i:p:IneqCdC2:proof5} since~$\psi_i\in L^2(\mssm)$, we have~$\mssm G_i<\infty$.
\end{enumerate}

Since~$\seq{G_i}_i$ is an exhaustion of~$X$, we may and will assume with no loss of generality ---up to dropping some of the first elements of all sequences above--- that~$\mssm G_1\geq 1$, in such a way~$(\mssm G_i)^{-1/2}\leq 1$ for every~$i\in\N$.
Now, set
\[
f_i\eqdef \frac{1}{\sqrt{\mssm G_i}}\psi_i \comma \qquad \varphi_n\eqdef \sum_i^n 2^{-i-1} f_i \comma
\]
and note that~$f_i\in \DzLocB{\mssm,\T}$, hence~$f_i\in \Lip(\mssd,\T)$ and~$\Li[\mssd]{f_i}\leq \Li[\mssd]{\psi_i}\leq 1$ by~$(\cSL{\T}{\mssm}{\mssd})$ and by~\ref{i:p:IneqCdC2:proof1} above, hence~\eqref{eq:IneqCdC:0.15} holds.
Furthermore, similarly to the proof of Proposition~\ref{p:IneqCdC}, we have~$\norm{\varphi_n}_{L^2(\mssm)}\leq 1$ for every~$n$.
Thus, similarly to the proof of Proposition~\ref{p:IneqCdC}, there exists the monotone limit~$\varphi\in \dom$.
Since~$\psi_i\nearrow_i 1$, we additionally have~$\varphi>0$ $\mssm$-a.e..
Letting~$\theta\eqdef \varphi^2/\norm{\varphi}_{L^2(\mssm)}^2$ shows the assertion in~\eqref{eq:IneqCdC:0.1}.

\paragraph{Dirichlet forms}
Arguing exactly as in the proof of Proposition~\ref{p:IneqCdC} we conclude~\eqref{eq:l:IneqCdC:0.5}.
Since~$\mssm'$ is a finite measure,~$\domb'\supset \DzLocB{\mssm,\T}$, and thus~\eqref{eq:l:IneqCdC:0.5} holds for every~$f\in \DzLocB{\mssm,\T}$.

\paragraph{Intrinsic distances}
We claim that~$\mssd_\mssm=\mssd_{\mssm'}$.
Let us verify the assumptions in Corollary~\ref{c:LocalityDistances} with~$\mcD\eqdef\dom$. 
Indeed,~$\dom$ is a (pseudo-)core for both~$(\mcE,\dom)$ and~$(\mcE',\dom')$ and~$\cdc=\cdc'$ on~$\dom$ by~\eqref{eq:l:IneqCdC:0.5}.
This verifies assumption~\ref{i:c:LocalityDistances:4} in Corollary~\ref{c:LocalityDistances}.
The sequences~$E_\bullet$ and~$G_\bullet$ are $\T$-open exhaustions of~$X$, and thus satisfy~$E_\bullet, G_\bullet\in\msG_0\cap\msG_0'$ (see~\ref{i:p:IneqCdC2:proof2},~\ref{i:p:IneqCdC2:proof3} above).
Furthermore~$a_i\eqdef\tparen{\frac{2^{-i-2}}{\sqrt{\mssm G_i} \norm{\varphi}_{L^2(\mssm)}}}^2\leq \theta$ everywhere on~$E_i$.
This verifies assumption~\ref{i:c:LocalityDistances:1} in Corollary~\ref{c:LocalityDistances}.

Finally, let~$\varrho_k\eqdef \tbraket{\theta_i\wedge 3 -2}_+\subset \dom$ and note that~$\varrho_k\in \DzB{\mssm,\T}$ by~\eqref{eq:TruncationLoc} and~\eqref{eq:SLoc:2}, and that~$\car_{E_i}\leq \varrho_i \leq \car_{G_i}$ since~$G_i\eqdef\set{\psi_i\geq \tfrac{1}{2}}=\set{\theta_i\geq 1}$.
That is, assumption~\ref{i:c:LocalityDistances:2} in Corollary~\ref{c:LocalityDistances} is satisfied.
This concludes the verification of the assumptions in Corollary~\ref{c:LocalityDistances}, and thus~$\mssd_\mssm=\mssd_{\mssm'}$.

\paragraph{Minimal relaxed slopes}
In light of Remark~\ref{r:ComparisonSL}\ref{i:r:ComparisonSL:1}, the space~$(X,\T,\mssd_\mssm)$ is an extended metric-topological space.
As in the proof of Proposition~\ref{p:IneqCdC}, let us verify that the space~$(X,\T,\mssd_\mssm,\mssm)$ satisfies the assumptions of Lemma~\ref{l:AGS2}.
By quasi-regularity of~$(\mcE,\dom)$ there exists a $\T$-compact $\mcE$-nest~$\seq{K'_i}_i$.
For each~$i\in\N$, the set~$K_i\eqdef \cl_\T(K_i'\cap E_i)$ is $\T$-compact, being the closure of a relatively $\T$-compact set, satisfies~$K_i\subset G_i$ by~\ref{i:p:IneqCdC2:proof4} above, and is thus a $\T$-compact $\mcE$-nest since~$\seq{E_i}_i$ and~$\seq{K_i}_i$ are $\mcE$-nests.
Thus, since~$\theta\in L^\infty(\mssm)$, it suffices to verify that there exists~$\eps>0$ and, for every~$i\in\N$, there exists~$c_i>0$, so that~$\theta>c_i$ on~$(K_i)_\eps\eqdef \set{\mssd_\mssm(\emparg, K_i)<\eps}$.

Since~$\psi_i\in \DzB{\T,\mssm}$ by~\ref{i:p:IneqCdC2:proof1} above,
\[
\mssd_\mssm(\emparg, K_i)\geq \inf_{x\in K_i} \tparen{\psi_i(x)-\psi_i(\emparg)}=1-\psi_i(\emparg)\comma
\]
hence
\[
(K_i)_\eps\subset\set{\psi_i\geq 1-\eps}=\set{f_i\geq \frac{1-\eps}{\sqrt{\mssm G_i}}}=\set{\frac{2^{-i-1}f_i}{\norm{\varphi}_{L^2(\mssm)}^2}\geq \frac{2^{-i-1}(1-\eps)}{\sqrt{\mssm G_i}\norm{\varphi}^2_{L^2(\mssm)}}}\comma
\]
and choosing~$\eps\eqdef 1/2$ and~$c_i\eqdef a_i$ as above we conclude that~$(K_i)_{1/2}\subset \set{\theta\geq c_i}$.
Thus,~\eqref{eq:l:AGS2:0} holds and the assumptions of Lemma~\ref{l:AGS2} are satisfied.
Now, applying (the second assertion in) Lemma~\ref{l:AGS2} to the probability measure~$\mssm'$, we have that
\begin{align}\label{eq:l:IneqCdC2:2}
\slo[*,\mssd_\mssm,\mssm]{f}=\slo[*,\mssd_\mssm,\mssm']{f}\comma \qquad f\in\domain{\Ch[*,\mssd_\mssm,\mssm]}
\fstop
\end{align}

\paragraph{Conclusion}
When~$\mssm$ is a probability measure, the sought inequality~$\cdc(f)\geq\wslo[\mssd_\mssm,\mssm]{f}^2$ is shown in~\cite{AmbErbSav16} under the assumption of $\T$-upper regularity.
Thus, it suffices to verify that $\T$-upper regularity too is transferred from~$(\mcE,\dom)$ to~$(\mcE',\dom')$.
Let~$\mcD$ be a pseudo-core for~$(\mcE,\dom)$ witnessing its $\T$-upper regularity, and note that~$\mcD$ is also a pseudo-core of~$(\mcE',\dom')$ by definition of the latter.
Now let~$f$, $\seq{f_n}_n$, and~$\seq{g_n}_n$ be as in~\eqref{eq:d:TUpperReg:0}.
Since~$\cdc'\equiv\cdc$ by~\eqref{eq:l:IneqCdC:0.5}, it is not difficult to show that~\eqref{eq:d:TUpperReg:0} holds as well with~$\cdc'$ in place of~$\cdc$ and~$\mssm'$ in place of~$\mssm$.

Respectively by definition of~$\DzB{\mssm,\T}$, construction of~$\dom'$, and by~\cite[Thm.~12.5]{AmbErbSav16}, we have $\DzB{\mssm,\T}\subset \dom\subset \dom'\subset\domain{\Ch[*,\mssd_\mssm,\mssm']}$.
Respectively by~\eqref{eq:l:IneqCdC:0.5},~\cite[Thm.~12.5]{AmbErbSav16},~\eqref{eq:l:IneqCdC:1} and~$\mssd_\mssm=\mssd_{\mssm'}$, we see that
\[
\cdc(f)=\cdc'(f)\geq\wslo[\mssd_{\mssm'},\mssm']{f}^2=\slo[*,\mssd_\mssm,\mssm']{f}^2\comma \qquad f\in\DzB{\mssm,\T}\fstop
\]
Applying Lemma~\ref{l:AGS2} while exchanging the roles of~$\mssm$ and~$\mssm'$ with~$\theta^{-1}$ in place of~$\theta$ we conclude from the above inequality that, in fact,~$\DzB{\mssm,\T}\subset \domain{\Ch[*,\mssd_\mssm,\mssm]}$.
Thus, continuing the above chain of inequalities, we conclude by~\eqref{eq:l:IneqCdC2:2}, and again by~\eqref{eq:l:IneqCdC:1}, that
\[
\cdc(f)\geq \slo[*,\mssd_\mssm,\mssm']{f}^2=\slo[*,\mssd_\mssm,\mssm]{f}^2=\wslo[\mssd_\mssm,\mssm]{f}^2\comma \qquad f\in\DzB{\mssm,\T}\fstop
\]
This last chain of inequalities extends to~$\DzLocB{\mssm,\T}$ by locality of all the objects involved, and the conclusion follows.
\end{proof}

Combining Propositions~\ref{p:IneqCdC} and~\ref{p:IneqCdC2} we further obtain the following identification of~$\mcE$ with~$\Ch[w,\mssd,\mssm]$ (or, equivalently, with~$\Ch[*,\mssd,\mssm]$).

\begin{corollary}\label{c:RadStoLCheegerComparison}
Let~$(\mbbX,\mcE)$ be a quasi-regular strongly local Dirichlet space with~$\mbbX$ satisfying~\ref{ass:Luzin} and admitting carr\'e du champ operator~$\cdc$, and~$\mssd\colon X^\tym{2}\to[0,\infty]$ be an extended distance.
Further assume that~$(\Rad{\mssd}{\mssm})$ and~$(\cSL{\T}{\mssm}{\mssd})$ hold.
Then,~$\mcE\leq \Ch[w,\mssd,\mssm]$, and the equality holds if and only if~$(\mbbX,\mcE)$ is additionally $\T$-upper regular.
\end{corollary}

\begin{remark}[Comparison with~\cite{AmbGigSav15,AmbErbSav16}]\label{r:ComparisonAGS-AES}
Corollary~\ref{c:RadStoLCheegerComparison} ought to be compared with~\cite[Thm.~3.14]{AmbGigSav15} and~\cite[Thm.~12.5]{AmbErbSav16}.
Let us first note that the two approaches have different premises: in~\cite{AmbGigSav15, AmbErbSav16} the authors compare the form~$(\mcE,\dom)$ with the Cheeger energy~$\Ch[w,\mssd_\mssm,\mssm]$ constructed from (the intrinsic distance~$\mssd_\mssm$ of)~$(\mcE,\dom)$.
We rather compare~$(\mcE,\dom)$ and~$\Ch[w,\mssd,\mssm]$ for some \emph{assigned}~$\mssd$, \emph{a priori} unrelated to~$\mssd_\mssm$.

Apart from the point of view, a main difference among the results in~\cite{AmbGigSav15}, those in~\cite{AmbErbSav16}, and ours lies in the generality of the respective assumptions, as we now detail.

\paragraph{Comparison with~\cite{AmbGigSav15}}
Under the standing assumption in~\cite{AmbGigSav15}:
\begin{enumerate}[$({a}_1)$]
\item\label{i:r:ComparisonAGS-AES:1} $\mssd_\mssm$ is a finite distance metrizing~$\T$ and $(X,\mssd_\mssm)$ is separable and complete;
\item\label{i:r:ComparisonAGS-AES:2} $(\mbbX,\mcE,\mssm)$ satisfies (a slightly stronger version of)~$(\Loc{\mssm,\T})$ (see~\cite[Eqn.~(3.28)]{AmbGigSav15});
\end{enumerate}
it is in fact possible to \emph{prove}~$(\Rad{\mssd_\mssm}{\mssm})$ ---see~\cite[Thm.~3.9]{AmbGigSav15} or (in a more general setting than in~\cite{AmbGigSav15})~\cite[Cor.~3.15]{LzDSSuz20}---, while~$(\cSL{\T}{\mssm}{\mssd_\mssm})$ holds by definition of~$\mssd_\mssm$ (cf.\ the proof of~\cite[Thm.~3.9]{AmbGigSav15}).
Furthermore, thanks to the Rademacher property~$(\Rad{\mssd_\mssm}{\mssm})$ and to~\ref{i:r:ComparisonAGS-AES:1} above, the form~$(\mcE,\dom)$ is quasi-regular by~\cite[Prop.~3.21]{LzDSSuz20}.
Thus, our assumptions are weaker (and our result stronger) than those in~\cite{AmbGigSav15}.

\paragraph{Comparison with~\cite{AmbErbSav16}}
It is (part of) the assumptions in~\cite{AmbErbSav16} that
\begin{enumerate}[$({b}_1)$]
\item $(X,\T,\mssd)$ is a completely regular Hausdorff \emph{extended} metric-topological space;
\item $\mssm$ is a probability measure and ---roughly--- $\DzLocB{\mssm,\T}$ generates~$\T$.
\end{enumerate}

Again, these assumptions trivially imply~$(\cSL{\T}{\mssm}{\mssd_\mssm})$. 
They are however skew to ours in the level of generality of the spaces involved.
Indeed, whereas we always assume that~$(X,\T)$ satisfies~\ref{ass:Luzin} (which is less restrictive than~\cite{AmbGigSav15} but more restrictive than~\cite{AmbErbSav16}), we are able to address the case of \emph{extended} distances (as in~\cite{AmbErbSav16}) on $\sigma$-finite spaces (as in~\cite{AmbGigSav15}, and as opposed to~\cite{AmbErbSav16}, addressing only probability spaces).
\end{remark}

\begin{proof}[Proof of Corollary~\ref{c:RadStoLCheegerComparison}]
It suffices to verify the assumptions in Propositions~\ref{p:IneqCdC} and~\ref{p:IneqCdC2}.
\end{proof}

Finally, let us now comment on the last equality in~\eqref{eq:SLChain2}, viz.
\begin{equation}\label{eq:SLChain3}
\wslo[\mssd]{f}^2= \slo[\mssd]{f}^2 \quad \as{\mssm}\comma \qquad f\in \DzLocB{\mssm,\T}\fstop
\end{equation}

\begin{remark}[On the validity of~\eqref{eq:SLChain3}]\label{r:Validity1}
When~$\mssd$ is a distance (as opposed to: an extended distance), the equality~\eqref{eq:SLChain3} is known to hold for all $\mssd$-Lipschitz functions under some assumptions on the metric measure space~$(X,\mssd,\mssm)$.
In particular, it is satisfied whenever $(X,\mssd,\mssm)$ is a (measure) doubling metric measure space (see Dfn.~\ref{d:MMSp} below) supporting a weak $(1,2)$-Poincar\'e inequality (see Dfn.~\ref{d:DP} below).
\end{remark}

\section{Tensorization}
In this section, we will mostly confine ourselves to \emph{metric measure spaces} in the sense of the following definition.
\begin{definition}[Metric measure space]\label{d:MMSp}
A triple~$(X,\mssd,\mssm)$ is called a \emph{metric measure space} if~$(X,\mssd)$ is a complete and separable metric space, and~$\mssm$ is a Borel measure on~$X$ finite on $\mssd$-bounded sets.
\end{definition}

For the purpose of comparison with known results in the literature, we recall the definition of infinitesimal Hilbertianity, introduced by N.~Gigli in~\cite{Gig13}.
\begin{definition}[Infinitesimal Hilbertianity]\label{d:IH}
A metric measure space~$(X,\mssd,\mssm)$ is \emph{infinitesimally Hilbertian} if the Cheeger energy~$\Ch[*,\mssd,\mssm]$ satisfies the parallelogram identity on~$L^2(\mssm)$, viz.
\begin{align}\tag{$\IH{\mssd}{\mssm}$}
2\tparen{\Ch[*,\mssd,\mssm](f)+\Ch[*,\mssd,\mssm](g)}= \Ch[*,\mssd,\mssm](f+g) + \Ch[*,\mssd,\mssm](f-g) \fstop
\end{align}
\end{definition}

In the case when~$(X,\mssd,\mssm)$ is infinitesimal Hilbertian, the Cheeger energy~$\Ch[*,\mssd,\mssm]$ is a quadratic functional.
The non-relabelled bilinear form induced on~$L^2(\mssm)$ by polarization is in fact a local Dirichlet form.

\subsection{Product structures}\label{sss:Products}
Let~$(\mbbX,\mcE)$ and~$(\mbbX',\mcE')$ be Dirichlet spaces.
Further let
\[
\mbbX^\otym{}=(X^\otym{},\tau^\otym{},\A^\otym{},\mssm^\otym{})\eqdef (X\times X', \T\times\T',\A\hotimes \A',\mssm\hotimes \mssm')
\]
be the corresponding product space.
For a function~$\rep f^\otym{}\colon X^\otym{}\to\R$ and any~$\mbfx\eqdef(x,x')\in X^\otym{}$, define the sections of~$\rep f^\otym{}$ at~$\mbfx$ by~$\rep f^\otym{}_{\mbfx,1}\colon y\mapsto \rep f^\otym{}(y, x')$ and~$\rep f^\otym{}_{\mbfx,2}\colon y'\mapsto \rep f^\otym{}(x,y')$.
Since no confusion may arise, we also write~$f^\otym{}\eqdef \tclass[\mssm^\otym{}]{\rep f}$,~$f^\otym{}_{\mbfx,1}\eqdef\tclass[\mssm]{\rep f^\otym{}_{\mbfx,1}}$ and~$f^\otym{}_{\mbfx,2}\eqdef\tclass[\mssm']{\rep f^\otym{}_{\mbfx,2}}$.
We stress that the subscript number indicates the \emph{free} coordinate.

\subsubsection{Products of Dirichlet spaces}
Set
\begin{align*}
\mcD^\otym{}\eqdef \set{f\in L^2(\mssm^\otym{}): \begin{gathered} f^\otym{}_{\mbfx,1} \in \dom\ \forallae{\mssm'} x'\in X'
\\
f^\otym{}_{\mbfx,2} \in \dom'\ \forallae{\mssm} x\in X
\\
\mcE^\otym{}(f^\otym{})\eqdef \int\mcE(f^\otym{}_{\mbfx,1}) \diff\mssm'\ +\int \mcE'(f^\otym{}_{\mbfx,2}) \diff\mssm<\infty
 \end{gathered}}\fstop
\end{align*}

\begin{proposition}[Product structures]\label{p:Products}
The following assertions hold:
\begin{enumerate}[$(i)$]
\item\label{i:p:Products:1} If both~$\mbbX$ and~$\mbbX'$ satisfy either~\ref{ass:Hausdorff},~\ref{ass:Luzin}, or~\ref{ass:Polish}, then so does~$\mbbX^\otym{}$;
\item\label{i:p:Products:2} the quadratic form~$(\mcE^\otym{},\mcD^\otym{})$ is closable, and its closure is a Dirichlet form on~$\mbbX^\otym{}$ with domain
\begin{equation}\label{eq:ProductDomain}
\dom^\otym{}\eqdef \set{f\in L^2(\mssm^\otym{}) : \begin{gathered} x'\mapsto f^\otym{}_{\mbfx,1}\in L^2(\mssm'; \dom)
\\
x\mapsto f^\otym{}_{\mbfx,2}\in L^2(\mssm; \dom')
\end{gathered}}\semicolon
\end{equation}

\item\label{i:p:Products:3} if both~$(\mbbX,\mcE)$ and~$(\mbbX',\mcE')$ are strongly local, then so is~$(\mbbX^\otym{},\mcE^\otym{})$;

\item\label{i:p:Products:4} the algebraic tensor product $\dom\otimes_{\R}\dom'$ is $\dom^\otym{}$-dense in~$\dom^\otym{}$.

\item\label{i:p:Products:5} if both~$(\mbbX,\mcE)$ and~$(\mbbX',\mcE')$ are either quasi-regular or regular, then so is~$(\mbbX^\otym{},\mcE^\otym{})$;

\item\label{i:p:Products:6} if~$(\mbbX,\mcE)$, resp.~$(\mbbX',\mcE')$, admits carr\'e du champ operator~$\cdc$, resp.~$\cdc'$, then $(\mbbX^\otym{},\mcE^\otym{})$ admits carr\'e du champ operator
\begin{equation}\label{eq:CdCProduct}
\cdc^\otym{}(f^\otym{})(x,x')\eqdef \cdc(f^\otym{}_{\mbfx,1})(x)+\cdc'(f^\otym{}_{\mbfx,2})(x')\comma \qquad \mbfx\eqdef (x,x') \fstop
\end{equation}
\end{enumerate}

\begin{proof}
A proof of~\ref{i:p:Products:1} is standard.
Proofs of~\ref{i:p:Products:2},~\ref{i:p:Products:3}, and~\ref{i:p:Products:6} are found in~\cite[Prop.~V.2.1.2, p.~201]{BouHir91}.
A proof of~\ref{i:p:Products:4} is found in~\cite[Prop.~V.2.1.3(b), p.~201]{BouHir91}.
A proof of~\ref{i:p:Products:5} in the regular case follows from~\ref{i:p:Products:4}.
The quasi-regular case follows from the regular case via the transfer method, cf.~e.g.~\cite{Kuw98,CheMaRoe94}.
\end{proof}
\end{proposition}

Combining~\eqref{eq:CdCProduct} with Corollary~\ref{c:BH}, we see that~\eqref{eq:CdCProduct} extends to~$f\in\dotloc{\dom^\otym{}}$.

\begin{corollary}
Let~$(\mbbX,\mcE)$ and~$(\mbbX',\mcE')$ be quasi-regular strongly local Dirichlet spaces with~$\mbbX,\mbbX'$ satisfying~\ref{ass:Luzin}.
Then,
\begin{equation}\label{eq:t:Tensor:5}
\cdc^\otym{}(f^\otym{})(x,x')\eqdef \cdc(f^\otym{}_{\mbfx,1})(x)+\cdc'(f^\otym{}_{\mbfx,2})(x')\comma \qquad \mbfx\eqdef (x,x')\comma \qquad f\in\dotloc{\dom^\otym{}} \fstop
\end{equation}
\end{corollary}

As a consequence, we further have
\begin{align}
\nonumber
f\otimes\car\in&\ \DzLocB{\mssm^\otym{}} \comma \car\otimes f'\in \DzLocB{\mssm^\otym{}} \comma & f\in&\ \DzLocB{\mssm}\comma f'\in \DzLocB{\mssm'} \comma
\\
\label{eq:TensorDzLocBT}
f\otimes\car\in&\ \DzLocB{\mssm^\otym{},\T^\otym{}} \comma \car\otimes f'\in \DzLocB{\mssm^\otym{},\T^\otym{}} \comma & f\in&\ \DzLocB{\mssm,\T}\comma f'\in \DzLocB{\mssm',\T'} \fstop
\end{align}

\paragraph{Sectioning}
Denote the sections of a set~$A^\otym{}\subset X^\otym{}$ by
\[
A^\otym{}_{\mbfx,1}\eqdef \set{x\in X: (x,x')\in A^\otym{}} \quad \text{and}\quad A^\otym{}_{\mbfx,2}\eqdef \set{x'\in X': (x,x')\in A^\otym{}}\comma \qquad \mbfx\eqdef (x,x,')\in X^\otym{} \fstop
\]

\begin{lemma}[Sectioning of quasi-notions]\label{l:Quasi-Sections}
Let~$(\mbbX,\mcE)$ and~$(\mbbX',\mcE')$ be quasi-regular strongly local Dirichlet spaces with~$\mbbX,\mbbX'$ satisfying~\ref{ass:Luzin}.
Then, the following assertions hold:
\begin{enumerate}[$(i)$]
\item\label{i:l:Quasi-Sections:1} if~$\seq{F^\otym{}_k}_k$ is an $\mcE^\otym{}$-nest, then~$\tseq{(F^\otym{}_k)_{\mbfx,1}}_k$ is an $\mcE$-nest for $\mssm'$-a.e.~$x'\in X'$, resp.~$\tseq{(F^\otym{}_k)_{\mbfx,2}}_k$ is an $\mcE'$-nest for $\mssm$-a.e.~$x\in X$;

\item if~$P^\otym{}\subset X^\otym{}$ is $\mcE^\otym{}$-polar, then~$P^\otym{}_{\mbfx,1}$ is $\mcE$-polar for $\mssm'$-a.e.~$x'\in X'$, resp.~$P^\otym{}_{\mbfx,2}$ is $\mcE'$-polar for $\mssm$-a.e.~$x\in X$;

\item\label{i:l:Quasi-Sections:3} if~$G^\otym{}\subset X^\otym{}$ is $\mcE^\otym{}$-quasi-open, then~$G^\otym{}_{\mbfx,1}$ is $\mcE$-quasi-open for $\mssm'$-a.e.~$x'\in X'$, resp.~$G^\otym{}_{\mbfx,2}$ is $\mcE'$-quasi-open for $\mssm$-a.e.~$x\in X$;

\item if~$f^\otym{}$ is $\mcE^\otym{}$-quasi-continuous, then~$f^\otym{}_{\mbfx,1}$ is $\mcE$-quasi-open for $\mssm'$-a.e.~$x'\in X'$, resp.~$f^\otym{}_{\mbfx,2}$ is $\mcE'$-quasi-open for $\mssm$-a.e.~$x\in X$.
\end{enumerate}

\begin{proof}
We only show~\ref{i:l:Quasi-Sections:1}, the other assertions being a straightforward consequence.
Let~$\seq{F^\otym{}_k}_k$ be an $\mcE^\otym{}$-nest. For each~$x'\in X$ and for each~$k\in\N$ let~$F_k\eqdef \set{x\in X: (x,x')\in F^\otym{}_k}$ be the $x'$-section of~$F^\otym{}_k$.
We show that~$\seq{F_k}_k$ is an $\mcE$-nest for~$\mssm'$-a.e.~$x'\in X'$.
Since~$F^\otym{}_k$ is $\T^\otym{}$-closed and sectioning preserves closedness, $F_k$~is $\T$-closed for every~$k$.
It suffices to show that~$\cap_k F_k^\complement$ is $\mcE$-polar.
When~$\mbbX$ and~$\mbbX'$ are probability spaces, this is claimed in, e.g.,~\cite[Exercise V.2.1(2), p.~208]{BouHir91}.

In order to address the case of~$\sigma$-finite $\mssm^\otym{}$, let~$\phi\in\dom$, resp.~$\phi'\in\dom'$, with~$\phi>0$ $\mssm$-a.e., resp.~$\phi'>0$ $\mssm'$-a.e., and set~$\phi^\otym{}\eqdef \phi\otimes \phi'\in\dom^\otym{}$.
Consider the Girsanov transforms~$(\mcE^\phi,\dom^\phi)$ of~$(\mcE,\dom)$,~$(\mcE^{\phi'},\dom^{\phi'})$ of~$(\mcE',\dom')$, and~$(\mcE^{\phi^\otym{}},\dom^{\phi^\otym{}})$ of~$(\mcE^\otym{},\dom^\otym{})$.
The assertion follows since Girsanov transforms by quasi-everywhere strictly positive functions preserve nests, hence polarity; see e.g.~\cite[p.~449]{CheSun06}.
\end{proof}
\end{lemma}

\begin{corollary}[Sectioning of broad local domains]\label{c:Sectioning-BLD}
Let~$(\mbbX,\mcE)$ and~$(\mbbX',\mcE')$ be quasi-regular strongly local Dirichlet spaces with~$\mbbX,\mbbX'$ satisfying~\ref{ass:Luzin}.
If~$f^\otym{}\in \dotloc{\dom^\otym{}}$, then~$f^\otym{}_{\mbfx,1}\in \dotloc{\dom}$ for $\mssm'$-a.e.~$x'\in X$ and~$f^\otym{}_{\mbfx,2}\in \dotloc{\dom'}$ for~$\mssm$-a.e.~$x\in X$.

\begin{proof}
Let~$(G^\otym{}_\bullet, f^\otym{}_\bullet)$ be witnessing that~$f^\otym{}\in\dotloc{\dom^\otym{}}$. 
We show the statement for $X$-sections, the one for $X'$-sections being analogous.
It suffices to note that, for $\mssm'$-a.e.~$x'\in X'$,
\begin{enumerate*}[$(a)$]
\item $(G^\otym{}_n)_{\mbfx,1}$ is $\mcE$-quasi-open by Lemma~\ref{l:Quasi-Sections}\ref{i:l:Quasi-Sections:3}

\item $(f^\otym{}_n)_{\mbfx,1}\in \dom$ by definition~\eqref{eq:ProductDomain} of~$\dom^\otym{}$;

\item $(f^\otym{}_n)_{\mbfx,1}\equiv f^\otym{}_{\mbfx,1}$ $\mssm$-a.e.\ on~$(G^\otym{}_n)_{\mbfx,1}$ by definition~$f^\otym{}_\bullet$.
\end{enumerate*}
\end{proof}
\end{corollary}

\begin{corollary}[Sectioning of broad local spaces of uniformly bounded energy]\label{c:Sectioning-BLS}
Let $(\mbbX,\mcE)$, resp.\ $(\mbbX',\mcE')$, be a quasi-regular strongly local Dirichlet spaces with~$\mbbX$, resp.~$\mbbX'$, satisfying~\ref{ass:Luzin}.
If~$f^\otym{}\in \DzLocB{\mssm^\otym{}}$, resp.~$\DzLocB{\mssm^\otym{},\T^\otym{}}$, then~$f^\otym{}_{\mbfx,1}\in \DzLocB{\mssm}$, resp.~$\DzLocB{\mssm,\T}$, for $\mssm'$-a.e.~$x'\in X$ and~$f^\otym{}_{\mbfx,2}\in \DzLocB{\mssm'}$, resp.~$\DzLocB{\mssm',\T'}$, for~$\mssm$-a.e.~$x\in X$.

\begin{proof}
Straightforward consequence of Corollary~\ref{c:Sectioning-BLD} and~\eqref{eq:CdCProduct}.
\end{proof}
\end{corollary}

\subsubsection{Products of metric objects}\label{sss:MetricProduct}
Let~$X,X'$ be any (non-empty) sets.
For extended pseudo-distances~$\mssd\colon X^\tym{2}\to [0,\infty]$ and~$\mssd'\colon X^{\prime\tym{2}}\to [0,\infty]$, we denote by
\[
\mssd^\otym{}\tparen{(x,x'),(y,y')}=(\mssd\otimes \mssd')\tparen{(x,x'),(y,y')}\eqdef \sqrt{\mssd(x,y)^2+ \mssd'(x',y')^2}
\]
the ($\ell^2$-)product extended pseudo-distance on~$X\times X'$, and by
\[
(\mssd\oplus\mssd')\tparen{(x,x'),(y,y')}\eqdef \mssd(x,y)+ \mssd'(x',y')
\]
the ($\ell^1$-)product extended pseudo-distance on~$X\times X'$.
Note that~$\mssd^\otym{}$ and~$\mssd\oplus\mssd'$ induce the same topology on~$X\times X'$.

\begin{lemma}\label{l:TensorIntrinsicD}
Let $(\mbbX,\mcE)$, resp.\ $(\mbbX',\mcE')$, be a quasi-regular strongly local Dirichlet spaces with~$\mbbX$, resp.~$\mbbX'$, satisfying~\ref{ass:Luzin}.
Then,
\begin{equation}\label{eq:l:TensorIntrinsicD:0}
\mssd_{\mssm^\otym{}} \leq \mssd_\mssm\oplus\mssd_{\mssm'} \as{\mssm^\otym{}} \fstop
\end{equation}
\begin{proof}
Set~$\mbfx\eqdef (x,x')$ and~$\mbfy\eqdef(y,y')$. Then, 
\begin{align}
\nonumber
\mssd_{\mssm^\otym{}}\tparen{(x,x'),(y,y')}\leq&\ \mssd_{\mssm^\otym{}}\tparen{(x,x'),(y,x')}+\mssd_{\mssm^\otym{}}\tparen{(y,x'),(y,y')}
\\
\label{eq:l:TensorIntrinsicD:1}
=&\ \inf_{f^\otym{}\in\DzLocB{\mssm^\otym{},\T^\otym{}}} f^\otym{}(x,x')-f^\otym{}(y,x')+ \inf_{g^\otym{}\in\DzLocB{\mssm^\otym{},\T^\otym{}}} g^\otym{}(y,x')-g^\otym{}(y,y')
\fstop
\end{align}
Choosing~$f^\otym{}$ of the form~$f\otimes\car$, resp.~$\car\otimes f'$ in the first, resp.\ second, summand of~\eqref{eq:l:TensorIntrinsicD:1}, we may continue~\eqref{eq:l:TensorIntrinsicD:1} with
\begin{align*}
\mssd_{\mssm^\otym{}}\tparen{(x,x'),(y,y')}\leq&\ \inf_{\substack{\rep{f}\colon X\to\R\\ f\otimes \car\in\DzLocB{\mssm^\otym{},\T^\otym{}}}} \rep{f}(x)-\rep{f}(y) + \inf_{\substack{\rep{g}\colon X'\to\R\\ \car\otimes g \in\DzLocB{\mssm^\otym{},\T^\otym{}}}} \rep{g}(x')-\rep{g}(y')
\\
\leq&\ \inf_{f \in\DzLocB{\mssm,\T}} f(x)-f(y) + \inf_{f'\in\DzLocB{\mssm',\T'}} g(x')-g(y')
\\
=&\ \mssd_\mssm(x,y)+ \mssd_{\mssm'}(x',y')\comma
\end{align*}
where the second inequality holds in light of~\eqref{eq:TensorDzLocBT}.
\end{proof}
\end{lemma}

Now, let~$(X,\T,\mssd,\mssm)$ and~$(X',\T',\mssd',\mssm')$ be extended metric-topological measure spaces in the sense of Definition~\ref{d:AES}.
We denote by
\begin{align*}
\wslo[\mssd,\mssd']{f^\otym{}}^{\otym{}}(x,x')\eqdef \sqrt{\wslo[\mssd]{f^\otym{}_{\mbfx,1}}^2(x)+\wslo[\mssd']{f^\otym{}_{\mbfx,2}}^2(x')} \comma \qquad f^\otym{}\colon X^\otym{}\to \R\comma
\end{align*}
the \emph{Cartesian gradient} in~\cite[p.~1477]{AmbGigSav14b}, by
\begin{align}\label{eq:CartesianGrad}
\slo[c,\mssd,\mssd']{f^\otym{}}^{\otym{}}(x,x')\eqdef \sqrt{\slo[\mssd]{f^\otym{}_{\mbfx,1}}^2(x)+\slo[\mssd']{f^\otym{}_{\mbfx,2}}^2(x')} \comma \qquad f^\otym{}\colon X^\otym{}\to \R \comma
\end{align}
the \emph{Cartesian slope} in~\cite[Eqn.~(3.3)]{AmbPinSpe15}, and by~$\slo[*,c,\mssd,\mssd']{f^\otym{}}^{\otym{}}$ the minimal relaxed gradient associated to the Cheeger energy
\begin{align*}
\Ch[*,c,\mssm^\otym{}](f^\otym{})\eqdef \inf\set{\liminf_n \int_{X^\tym{2}} \tparen{\slo[c,\mssd,\mssd']{f^\otym{}_n}^{\otym{}}}^2 \diff\mssm^\otym{}}=\int_{X^\tym{2}} \tparen{\slo[*,c,\mssd,\mssd']{f^\otym{}}^{\otym{}}}^2 \diff\mssm^\otym{} \comma
\end{align*}
where the infimum is taken over all sequences~$\seq{f^\otym{}_n}_n\subset \Lip(\mssd^\otym{})$ with~$L^2(\mssm^\otym{})\text{-}\nlim f_n^\otym{}=f^\otym{}$.

Recall that, by e.g.~\cite[Thm.~2.2(ii)]{AmbPinSpe15}, for every~$f^\otym{}\in\Lip_{bs}(\mssd^\otym{})\subset L^2(\mssm^\otym{})$ there exists a sequence of functions~$\seq{f^\otym{}_n}_n\subset \Lip_{bs}(\mssd^\otym{})$ such that
\begin{align}\label{eq:APS}
L^2(\mssm^\otym{})\text{-}\nlim f^\otym{}_n= f^\otym{} \qquad \text{and} \qquad \slo[*,c,\mssd,\mssd']{f^\otym{}}^{\otym{}}= L^2(\mssm^\otym{})\text{-}\nlim \slo[c,\mssd,\mssd']{f^\otym{}_n}^{\otym{}}\fstop
\end{align}

\subsection{Tensorization of the Rademacher property}\label{sss:TensorizationConseq}
In this section we show the tensorization of the Rademacher property.
In the literature, it has been addressed in the case when~$\mcE=\Ch[\mssd,\mssm]$ is the Cheeger energy, in connection with the tensorization of the Cheeger energy under the assumptions of
\begin{itemize}
\item infinitesimal Hilbertianity (Dfn.~\ref{d:IH}), measure doubling and a weak $(1,2)$-Poincar\'e inequality (see Dfn.~\ref{d:DP} below) in~\cite{AmbPinSpe15};
\item under the Riemannian Curvature-Dimension condition~$\RCD(K,\infty)$ in~\cite{AmbGigSav14b};
\item more recently under the \emph{infinitesimal quasi-Hilbertianity} in~\cite{EriRajSou22}.
\end{itemize}
Here, we discuss the case of general quasi-regular strongly local Dirichlet spaces~$(\mbbX,\mcE)$, without any geometric assumption.
In particular, we never require a two-sided comparison of~$(\mcE,\dom)$ with~$\Ch[*,\mssd,\mssm]$ implicit in the definition of infinitesimal (quasi-)Hilbertianity in~\cite{EriRajSou22}.

Similarly to the discussion on transfer by weight in~\S\ref{s:Transfer}, our assumptions are weaker than those in~\cite{AmbGigSav14b,AmbPinSpe15,EriRajSou22}, since the Cheeger energy~$\Ch[*,\mssd,\mssm]$ of a metric measure space is always quasi-regular by~\cite[Thm.~4.1]{Sav14} or~\cite[Prop.~3.21]{LzDSSuz20}, cf.~ Remark~\ref{r:ComparisonAGS-AES}.

\begin{theorem}[Tensorization of $(\mathsf{Rad})$]\label{t:TensorRad}
Let~$\mbbX\eqdef (X,\mssd,\mssm)$ be a metric measure space, and~$(\mcE,\dom)$ be a quasi-regular strongly local Dirichlet form on~$\mbbX$ admitting carr\'e du champ operator~$\cdc$ and satisfying~$(\Rad{\mssd}{\mssm})$.
Further let~$(\mbbX',\mcE')$ be satisfying the same assumptions as~$(\mbbX,\mcE)$.

Then, their product space~$(\mbbX^\otym{},\mcE^\otym{})$ satisfies~$(\Rad{\mssd^\otym{}}{\mssm^\otym{}})$.
\begin{proof}
Since~$(X,\mssd,\mssm)$ and~$(X',\mssd',\mssm')$ are metric measure spaces in the sense of Definition~\ref{d:MMSp}, their product space~$(X^\otym{},\mssd^\otym{},\mssm^\otym{})$ is so as well.
Since~$\T_\mssd=\T$, all Lipschitz functions in this proof are continuous (in particular: Borel) for the relative distances/topologies, and we thus omit the specification of measure-representatives of such functions.

Let~$\seq{G_k}_k$, resp.~$\seq{G_k'}_k$, be a monotone exhaustion of~$X$, resp.~$X'$ consisting of well-separated (see Rmk.~\ref{r:CutOff}) bounded open sets.
Set~$G_k^\otym{}\eqdef G_k\times G_k'$, and note that~$G_k^\otym{}$ is a monotone exhaustion of~$X^\otym{}$ consisting of well-separated bounded open sets.
Similarly to Remark~\ref{r:CutOff}, for each~$k$ there exists~$\varrho_k^\otym{}\in\bLip(\mssd^\otym{})$ satisfying condition~\ref{i:p:LocalityProbab:2} in Proposition~\ref{p:LocalityProbab}, that is such that~$\car_{G_k^\otym{}}\leq\varrho_k^\otym{}\leq \car_{{G_{k+1}^\otym{}}^c}$ and
\[
\Li[\mssd^\otym{}]{\varrho_k^\otym{}}\leq c_k\eqdef \mssd^\otym{}(G_k^\otym{}, {G_{k+1}^\otym{}}^\complement)^{-1}<\infty\fstop
\]

Throughout the proof, let~$\mbfx\eqdef (x,x')\in X^\otym{}$.

\paragraph{Step I: $\bLip(\mssd^\otym{})\subset \dom^\otym{}$} 
Without loss of generality, we show that $\bLipu(\mssd^\otym{})\subset \dom^\otym{}$.
Let~$f^\otym{}\in \bLipu(\mssd^\otym{})$ and set~$M\eqdef \sup_{X^\otym{}} \abs{f^\otym{}}$.
Note that, for each fixed~$\mbfx\eqdef (x,x')\in X^\otym{}$,
\begin{equation}\label{eq:t:Tensor:1}
\begin{aligned}
f_{k,\mbfx,1}\eqdef&\ (f^\otym{}\varrho_k^\otym{})_{\mbfx,1}\in \bLip(\mssd) \comma \quad \Li[\mssd]{f_{\mbfx,k,1}}\leq Mc_k+1\comma
\\
f_{k,\mbfx,2}\eqdef&\ (f^\otym{}\varrho_k^\otym{})_{\mbfx,2}\in \bLip(\mssd') \comma \quad \Li[\mssd']{f_{\mbfx,k,2}}\leq Mc_k+1 \fstop
\end{aligned}
\end{equation}
By~$(\Rad{\mssd}{\mssm})$, resp.~$(\Rad{\mssd'}{\mssm'})$, we have
\begin{equation}\label{eq:t:Tensor:2}
\begin{aligned}
f_{k,\mbfx,1}\in&\ \dotloc{\dom}\comma & \cdc(f_{k,\mbfx,1})\leq&\ \Li[\mssd]{f_{k,\mbfx,1}}^2 \as{\mssm}\comma
\\
f_{k,\mbfx,2}\in&\ \dotloc{\dom'}\comma & \cdc'(f_{k,\mbfx,2})\leq&\ \Li[\mssd']{f_{k,\mbfx,2}}^2 \as{\mssm'} \fstop
\end{aligned}
\end{equation}
Furthermore, by definition of~$\varrho_k^\otym{}$ we have~$f_{k,\mbfx,1}\equiv 0$ everywhere on~$G_{k+1}^\complement$, resp.~$f_{k,\mbfx,2}\equiv 0$ everywhere on~${G_{k+1}'}^\complement$.
Combing this fact with~\eqref{eq:t:Tensor:1} and~\eqref{eq:t:Tensor:2}, we conclude from Corollary~\ref{c:BH} that
\begin{equation}\label{eq:t:Tensor:3}
\begin{aligned}
\cdc(f_{k,\mbfx,1})\leq&
\begin{cases}
(Mc_k+1)^2 & \as{\mssm}\ \text{on } G_{k+1}
\\
0 & \as{\mssm}\ \text{on } {G_{k+1}}^\complement
\end{cases}\comma
\\
\cdc'(f_{k,\mbfx,2})\leq&
\begin{cases}
(Mc_k+1)^2 & \as{\mssm'}\ \text{on } G'_{k+1}
\\
0 & \as{\mssm'}\ \text{on } {G'_{k+1}}^\complement
\end{cases} \fstop
\end{aligned}
\end{equation}
Thus, since $G_k$, resp.~$G_k'$, has finite $\mssm$-, resp.~$\mssm'$-measure for all~$k$, \eqref{eq:t:Tensor:3} implies that, for each~$k$ and every~$\mbfx\in X^\otym{}$,
\begin{equation}\label{eq:t:Tensor:4}
\begin{aligned}
f_{k,\mbfx,1}\in&\ \dom\comma & \norm{f_{k,\mbfx,1}}_\dom^2\leq&\ \mssm G_{k+1} \tparen{(Mc_k+1)^2+ M^2}\comma
\\
f_{k,\mbfx,2}\in&\ \dom'\comma & \norm{f_{k,\mbfx,2}}_{\dom'}^2\leq&\ \mssm' G_{k+1}' \tparen{(Mc_k+1)^2+ M^2}\fstop
\end{aligned}
\end{equation}
As a consequence,
\begin{equation*}
\begin{aligned}
\norm{x'\mapsto f_{k,\mbfx,1}}_{L^2(\mssm',\dom)}^2\leq \mssm^\otym{} G_{k+1}^\otym{} \tparen{(Mc_k+1)^2+ M^2}\comma
\\
\norm{x'\mapsto f_{k,\mbfx,1}}_{L^2(\mssm',\dom)}^2\leq \mssm^\otym{} G_{k+1}^\otym{} \tparen{(Mc_k+1)^2+ M^2}\comma
\end{aligned}
\end{equation*}
which finally shows that~$f^\otym{}\varrho_k^\otym{}\in \dom^\otym{}$ for each~$k$ by definition~\eqref{eq:ProductDomain} of~$\dom^\otym{}$.
Since~$\seq{G_k^\otym{}}_k$ is an open exhaustion of~$X^\otym{}$, the sequence~$\seq{G_k^\otym{}, f^\otym{}\varrho_k^\otym{}}_k$ witnesses that~$f^\otym{}\in\dotloc{\dom^\otym{}}$.

\paragraph{Step II: sharp estimate}
We now proceed to estimate of~$\cdc^\otym{}(\emparg)$ by~$\Li[\mssd^\otym{}]{\emparg}$.
Fix~$f^\otym{}\in \bLipu(\mssd^\otym{})$.
Since~$f^\otym{}\in\dotloc{\dom^\otym{}}$ by Step~I, by Corollary~\ref{c:BH} we have 
\begin{equation}\label{eq:t:Tensor:6}
\cdc^\otym{}(f^\otym{})(\mbfx)= \cdc^\otym{}(f^\otym{}\varrho_k^\otym{})(\mbfx) \quad \text{on } G_k^\otym{}\fstop
\end{equation}
Thus, it suffices to show that, for each~$k\in\N$,
\begin{equation}\label{eq:t:Tensor:7}
\cdc^\otym{}(f^\otym{}\varrho_k^\otym{})\leq 1 \as{\mssm^\otym{}}\ \text{on }G_k^\otym{}\fstop
\end{equation}

Let~$k$ be fixed.
Since~$f^\otym{}\varrho_k^\otym{}$ is $\mssd^\otym{}$-Lipschitz with $\mssd^\otym{}$-bounded support, we can find a sequence~$\tseq{f^\otym{}_{k,n}}_n$ satisfying~\eqref{eq:APS} with~$f^\otym{}\varrho_k^\otym{}$ in place of~$f^\otym{}$.
By Step~I we further have~$f^\otym{}_{k,n}\in \dotloc{\dom^\otym{}}$ for each~$n$, thus we may compute~$\cdc^\otym{}(f^\otym{}_{k,n})$ pointwise.

By~\eqref{eq:t:Tensor:5}, and applying Proposition~\ref{p:IneqCdC} on~$X$ and~$X'$,
\begin{align}
\nonumber
\cdc^\otym{}(f^\otym{}_{k,n})\leq&\ \cdc\tparen{(f^\otym{}_{k,n})_{\emparg,1}} + \cdc'\tparen{(f^\otym{}_{k,n})_{\emparg,2}} \leq \twslo[\mssd]{(f^\otym{}_{k,n})_{\emparg,1}}^2+\twslo[\mssd']{(f^\otym{}_{k,n})_{\emparg,2}}^2
\\
\label{eq:t:Tensor:9}
\leq&\ \tslo[\mssd]{(f^\otym{}_{k,n})_{\emparg,1}}^2+\tslo[\mssd']{(f^\otym{}_{k,n})_{\emparg,2}}^2 \defeq \paren{\tslo[c,\mssd,\mssd']{f^\otym{}_{k,n}}^\otym{}}^2
\comma
\end{align}
where the inequality~\eqref{eq:t:Tensor:9} holds since the slope of a Lipschitz function is a weak upper gradient.
Since the right-hand side of~\eqref{eq:t:Tensor:9} converges in~$L^2(\mssm^\otym{})$ by~\eqref{eq:APS}, integrating~\eqref{eq:t:Tensor:9} w.r.t.~$\mssm^\otym{}$ shows that~$\tseq{f^\otym{}_{k,n}}_n$ is uniformly bounded in~$\dom^\otym{}$.
Thus, up to possibly choosing a suitable non-relabeled subsequence, we may and will assume with no loss of generality that~$\tseq{f^\otym{}_{k,n}}_n$ is $\dom^\otym{}$-weakly convergent, in which case its $\dom^\otym{}$-weak limit is~$f^\otym{}\varrho^\otym{}_k$ in light of its $L^2(\mssm^\otym{})$-convergence to~$f^\otym{}\varrho^\otym{}_k$ and~\cite[Lem.~I.2.12]{MaRoe92}.

Fix now~$h^\otym{}\in{\domb^\otym{}}^+$. Then, by e.g.~\cite[Lem.~3.3(ii)]{AriHin05},
and by~\eqref{eq:t:Tensor:9},
\begin{align}\label{eq:t:Tensor:10}
\int h^\otym{}\, \cdc^\otym{}(f^\otym{}\varrho_k^\otym{}) \diff\mssm^\otym{}\leq&\ \nliminf \int h^\otym{}\, \cdc^\otym{}(f^\otym{}_{k,n})\diff\mssm^\otym{}
\leq \nliminf \int h^\otym{} \paren{\tslo[c,\mssd,\mssd']{f^\otym{}_{k,n}}^\otym{}}^2\diff\mssm^\otym{} \fstop
\end{align}
Since~$h^\otym{}\in L^\infty(\mssm^\otym{})$, by~\eqref{eq:APS} there exists
\begin{align}\label{eq:t:Tensor:12}
\nlim \int h^\otym{} \paren{\tslo[c,\mssd,\mssd']{f^\otym{}_{k,n}}^\otym{}}^2 \diff\mssm^\otym{} = \int h^\otym{}\paren{\tslo[*,c,\mssd,\mssd']{(f^\otym{}\varrho_k^\otym{})}^\otym{}}^2 \diff\mssm^\otym{}
\fstop
\end{align}
Combining~\eqref{eq:t:Tensor:10} with~\eqref{eq:t:Tensor:12} we thus have
\begin{align*}
\int h^\otym{}\, \cdc^\otym{}(f^\otym{}\varrho_k^\otym{}) \diff\mssm^\otym{} \leq& \int h^\otym{}\paren{\tslo[*,c,\mssd,\mssd']{(f^\otym{}\varrho_k^\otym{})}^\otym{}}^2 \diff\mssm^\otym{}
\comma
\end{align*}
whence, by arbitrariness of~$h^\otym{}$, we conclude that
\begin{align*}
\cdc^\otym{}(f^\otym{}\varrho_k^\otym{})\leq&\ \paren{\tslo[*,c,\mssd,\mssd']{(f^\otym{}\varrho_k^\otym{})}^\otym{}}^2
= \tslo[*,\mssd^\otym{}]{(f^\otym{}\varrho_k^\otym{})}^2
\leq \tslo[\mssd^\otym{}]{(f^\otym{}\varrho_k^\otym{})}^2
\comma
\end{align*}
where the equality is shown in~\cite[Thm.~3.2]{AmbPinSpe15}, and the last inequality holds again by definition of the objects involved.
Since~$\varrho^\otym{}_k\equiv \car$ on~$G^\otym{}_k$, by locality of~$\slo[\mssd^\otym{}]{\emparg}$ and locality of~$\cdc^\otym{}$ as in Corollary~\ref{c:BH}, we finally obtain
\begin{align*}
\cdc^\otym{}(f^\otym{}\varrho_k^\otym{})\car_{G_k^\otym{}}\leq \slo[\mssd^\otym{}]{f^\otym{}} \car_{G_k^\otym{}}\leq \Li[\mssd^\otym{}]{f^\otym{}} \leq 1
\end{align*}
by the assumption that~$f^\otym{}\in \bLipu(\mssd^\otym{})$, which concludes~\eqref{eq:t:Tensor:7} and therefore the assertion of the theorem.
\end{proof}
\end{theorem}

\subsection{Tensorization of the Sobolev-to-Lipschitz property}\label{sss:TensorSL}
In this section we show a slightly weaker form of tensorization of the Sobolev-to-Lipschitz property.
Precisely, we show that the Sobolev-to-Lipschitz property on the factors implies the \emph{continuous}-Sobolev-to-Lipschitz property on their product space.

\begin{theorem}\label{t:TensorSL}
Let~$\mbbX\eqdef (X,\mssd,\mssm)$ be a metric measure space (Dfn.~\ref{d:MMSp}), and~$(\mcE,\dom)$ be a quasi-regular strongly local Dirichlet form on~$\mbbX$ satisfying~$(\SL{\mssm}{\mssd})$.
Further let~$(\mbbX',\mcE')$ be satisfying the same assumptions as~$(\mbbX,\mcE)$.
Then, their product space~$(\mbbX^\otym{},\mcE^\otym{})$ satisfies~$(\cSL{\T^\otym{}}{\mssm^\otym{}}{\mssd^\otym{}})$.
\end{theorem}

Before discussing a proof of Theorem~\ref{t:TensorSL}, let us explain why both its statement and the proof presented below are not dual to the tensorization of the Rademacher property in Theorem~\ref{t:TensorRad}.

Indeed, the proof of Theorem~\ref{t:TensorRad} crucially relies on Proposition~\ref{p:IneqCdC}.
A dual proof for the tensorization of the Sobolev-to-Lipschitz property would then rely on the chain of inequalities in~\eqref{eq:SLChain2}.
Now, the validity of the first inequality in~\eqref{eq:SLChain2} is implied by ---in essence, equivalent to--- the $\T$-upper regularity of the form~$(\mcE,\dom)$, as discussed in (the proof of) Proposition~\ref{p:IneqCdC2}.
The validity of the last equality in~\eqref{eq:SLChain2}, that is,~\eqref{eq:SLChain3}, was already discussed for \emph{Lipschitz}~$f$ in Remark~\ref{r:Validity1} under further assumptions on the metric measure structure, such as, e.g., measure doubling and a weak $(1,2)$-Poincar\'e inequality.
Deducing the same equality for~$f\in\DzLocB{\mssm,\T}$ from the equality for~$f\in\bLip(\mssd)$ is equivalent to having already shown the continuous-Sobolev-to-Lipschitz property~$(\cSL{\T}{\mssm}{\mssd})$.
As a consequence, it seems difficult to establish the chain of inequalities in~\eqref{eq:SLChain2} without assuming further structural properties of the product metric measure structure.

\subsubsection{Semigroups, irreducibility, and maximal functions}\label{sss:Semigroups}
Let us collect here some definitions and simple facts about them.
Throughout this section, let~$(\mbbX,\mcE)$ and~$(\mbbX',\mcE')$ be any Dirichlet spaces.

\paragraph{Semigroups}
We denote by~$\seq{T_t}_{t\geq 0}$, $T_t\colon L^2(\mssm)\to L^2(\mssm)$, the (strongly continuous contraction) $L^2(\mssm)$-semigroup associated to~$(\mcE,\dom)$.
We denote again by~$T_t\colon L^\infty(\mssm)\to L^\infty(\mssm)$ the extension of~$T_t\colon L^2(\mssm)\to L^2(\mssm)$ to a weakly* continuous semigroup on~$L^\infty(\mssm)$.

We denote by~$T_t^\otym{}$ the semigroup(s) associated to~$(\mbbX^\otym{},\mcE^\otym{})$.
We note that~$T_t^\otym{}=T_t\otimes T_t'$, that is, in particular
\begin{equation}\label{eq:TensorSemigroup}
T_t^\otym{}(f\otimes f')(x,x')=(T_t f)(x) (T_t' f')(x') \comma \qquad t\geq 0\comma \quad f\in L^\infty(\mssm)\comma f'\in L^\infty(\mssm')\fstop
\end{equation}

\paragraph{Invariance and irreducibility}
A set~$A\in\A$ is \emph{$\mcE$-invariant} if
\[
T_t(\car_A f)=\car_A T_tf\comma \qquad f\in L^2(\mssm) \fstop
\]

For a characterization of~$\mcE$-invariance in the present setting, see e.g.~\cite[Prop.~2.38]{LzDSSuz21}.
We say that a Dirichlet space~$(\mbbX,\mcE)$ is \emph{irreducible} if every~$\mcE$-invariant~$A\in\A$ satisfies either~$\mssm A=0$ or~$\mssm A^\complement=0$.

\begin{lemma}\label{SLIrreducibility}
Let~$(\mbbX,\mcE)$ be a quasi-regular strongly local Dirichlet form on~$\mbbX$, and~$\mssd\colon X^\tym{2}\to [0,\infty)$ be a distance (not: extended) on~$X$.
If~$(\mbbX,\mcE)$ satisfies~$(\SL{\mssm}{\mssd})$, then it is irreducible.

\begin{proof}
Argue by contradiction that~$(\mbbX,\mcE)$ is \emph{not} irreducible but~$(\SL{\mssm}{\mssd})$ holds.
Since~$(\mcE,\dom)$ is not irreducible, there exists an $\mcE$-invariant~$A\in\A$ with~$\mssm A>0$ and~$\mssm A^\complement>0$.
By $\mcE$-invariance of~$A$ and~\cite[Lem.~2.39]{LzDSSuz21} we have~$\car_A\in\dotloc{\dom}$ and~$\sq{\car_A}\equiv 0$.
Hence, by~$(\SL{\mssm}{\mssd})$, the function~$\car_A$ has a $\mssd$-Lipschitz representative~$\widerep{\car_A}$ with~$\Li[\mssd]{\widerep{\car_A}}=0$.
Since~$\mssd$ is a distance (not: extended),~$\widerep{\car_A}$ is a constant, which is a contradiction since~$\mssm A,\mssm A^\complement>0$.
\end{proof}
\end{lemma}

See~\cite[\S2.6.2]{LzDSSuz21} for relations between $\mcE$-invariance and metric notions (in particular, $\mssd$-accessibility) in the non-irreducible case.

\paragraph{Maximal functions}
We recall the following adaptation to our setting of results in~\cite{AriHin05}. Set
\begin{align*}
\Dz{\mu,A}_{\loc,r}\eqdef \set{f\in\DzLoc{\mu}: f=0 \as{\mssm} \text{~on~} A\comm \abs{f}\leq r \as{\mssm}}\subset \DzLocB{\mu}\comm \qquad r>0\fstop
\end{align*}

\begin{proposition}[{\cite[Prop.~4.14]{LzDSSuz21}}]\label{p:Hino}
Let~$(\mbbX,\mcE)$ be a quasi-regular strongly local Di\-richlet space.
For each~$A\in\A$ there exists an $\mssm$-a.e.\ unique $\A$-measurable function~$\hr{\mssm,A}\colon X\rar [0,\infty]$ so that~$\hr{\mssm,A}\wedge r$ is the $\mssm$-a.e.\ maximal element of~$\Dz{\mssm,A}_{\loc, r}$.
\end{proposition}

We call the function~$\hr{\mu,A}$ constructed in Proposition~\ref{p:Hino} the \emph{maximal function} of~$A\in\A$.
Note that~$\hr{\mssm,A}$ is generally not an element of~$\dotloc{L^\infty(\mssm)}$, unless~$(\mbbX,\mcE)$ is irreducible.

In proving the tensorization of the Sobolev-to-Lipschitz property, we shall need the following specialization to our setting of a very general result by T.~Ariyoshi and M.~Hino.

\begin{theorem}[{\cite[Thm.~5.2]{AriHin05}, irreducible case}]\label{t:AriHino}
Let~$(\mbbX,\mcE)$ be an irreducible quasi-regular strongly local Di\-richlet space.
Then, for every~$A\in\A$ with~$0<\mssm A<\infty$,
\begin{align*}
\nu\text{-}\lim_{t\downarrow 0} \tparen{-2t\log T_t\car_A} = \hr{\mssm, A}^2
\end{align*}
for every probability measure~$\nu$ on~$\A$ equivalent to~$\mssm^\otym{}$.
\end{theorem}

\subsubsection{Tensorization of the Sobolev-to-Lipschitz property}

Let us now show a form of tensorization of the Sobolev-to-Lipschitz property for metric measure spaces.
We recall that, on every such space~$(X,\mssd,\mssm)$, the topology~$\T$ is always assumed to be the one induced by~$\mssd$, and no other topology is considered.

We start with two preparatory lemmas.
\begin{lemma}\label{l:IntrinsicIneq}
Let $(\mbbX,\mcE)$, resp.\ $(\mbbX',\mcE')$, be a quasi-regular strongly local Dirichlet spaces with~$\mbbX$, resp.~$\mbbX'$, satisfying~\ref{ass:Luzin}.
Further let~$\mssd\colon X^\tym{2}\to [0,\infty]$, resp.~$\mssd'\colon X^{\prime\tym{2}}\to [0,\infty]$, be a separable jointly $\T$-, resp.~$\T'$-, continuous extended pseudo-distance on~$X$, resp.~$X'$, and assume~$(\dSL{\mssd}{\mssm})$, resp.~$(\dSL{\mssd'}{\mssm'})$.
Then,
\begin{equation}\label{eq:l:IntrinsicIneq:0}
\mssd_{\mssm^\otym{}}\leq \mssd_\mssm\oplus \mssd_{\mssm'} \quad \text{everywhere on~$X^\otym{}$}\comma
\end{equation}
and the topology induced by~$\mssd_{\mssm^\otym{}}$ on~$X^\otym{}$ is separable.
\begin{proof}
Firstly, note that by~$(\dSL{\mssd}{\mssm})$,~$(\dSL{\mssd'}{\mssm'})$ and joint $\T^\otym{}$-continuity of~$\mssd\oplus\mssd'$, the extended pseudo-distance~$\mssd_\mssm\oplus \mssd_{\mssm'}$ too is jointly $\T^\otym{}$-continuous.
Now, let~$\Omega^\otym{}\in\A^\otym{}$ be a set of full $\mssm^\otym{}$-measure on which~\eqref{eq:l:TensorIntrinsicD:0} holds.
By~$\T^\otym{}$-density of~$\Omega^\otym{}$ in~$X^\otym{}$, for every~$\mbfx,\mbfy\in X^\otym{}$ there exists~$\seq{\mbfx_n}_n,\seq{\mbfy_n}_n\subset \Omega^\otym{}$ with~$\T^\otym{}$-$\nlim \mbfx_n=\mbfx$ and~$\T^\otym{}$-$\nlim \mbfy_n=\mbfy$.
Thus, by joint $\T^\otym{}$-lower semi-continuity of~$\mssd_{\mssm^\otym{}}$ and joint $\T^\otym{}$-continuity of~$\mssd_\mssm \oplus \mssd_{\mssm'}$,
\begin{align*}
\mssd_{\mssm^\otym{}}(\mbfx,\mbfy) \leq \nliminf \mssd_{\mssm^\otym{}}(\mbfx_n,\mbfy_n) \leq \nliminf (\mssd_\mssm \oplus \mssd_{\mssm'})(\mbfx_n,\mbfy_n) = (\mssd_\mssm\oplus \mssd_{\mssm'})(\mbfx,\mbfy) \fstop
\end{align*}

\paragraph{Separability of~$\mssd_{\mssm^\otym{}}$}
Since~$(X,\mssd)$, resp.~$(X',\mssd')$, is separable, so is~$(X,\mssd_\mssm)$, resp.~$(X',\mssd_{\mssm'})$, by~$(\dSL{\mssd}{\mssm})$, resp.~$(\dSL{\mssd'}{\mssm'})$.
Thus,~$(X^\otym{},\mssd_\mssm\oplus \mssd_{\mssm'})$ is separable as well, and~$(X,\mssd_{\mssm^\otym{}})$ is separable as a consequence of~\eqref{eq:l:IntrinsicIneq:0}.
\end{proof}
\end{lemma}

Let us recall the following results in~\cite{LzDSSuz21}.

\begin{lemma}\label{l:Lemma}
Let $(\mbbX,\mcE)$ be a quasi-regular strongly local Dirichlet space with~$\mbbX$ satisfying~\ref{ass:Luzin}, and assume that the topology generated on~$X$ by~$\mssd_\mssm$ is separable.
Then,
\begin{enumerate}[$(i)$]
\item\label{i:l:Lemma:1} $(\mbbX,\mcE)$ satisfies~$(\Rad{\mssd_\mssm}{\mssm})$;
\item\label{i:l:Lemma:2} $\mssd_\mssm(\emparg, A)$ is $\A$-measurable for every~$A\in\A$;
\item\label{i:l:Lemma:3} $\mssd_\mssm(\emparg, A)\leq \hr{\mssm,A}$ $\mssm$-a.e. for every~$A\in\A$.
\end{enumerate}

\begin{proof}
Assertion~\ref{i:l:Lemma:1} is~\cite[Cor.~3.15]{LzDSSuz21}.
In order to show~\ref{i:l:Lemma:2}, it suffices to show that~$\mssd_\mssm(\emparg,x)$ is $\A$-measurable for every~$x\in X$.
Indeed, this implies the $\A$-measurability of~$\mssd_\mssm(\emparg, A)$ for every~$A\in\A$ as in the proof of~\cite[Rmk.~3.12(i)]{LzDSSuz21}.
Since~$\mssd_\mssm(\emparg,x)$ is $\T$-lower semi-continuous for every~$x\in X$, it is also $\A$-measurable, since~$\A\supset\Bo{\T}$.
This concludes the proof of~\ref{i:l:Lemma:2}.
Assertion~\ref{i:l:Lemma:3} follows from~\ref{i:l:Lemma:1} and~\ref{i:l:Lemma:2} by~\cite[Lem.~4.16]{LzDSSuz21}.
\end{proof}
\end{lemma}

\begin{proof}[Proof of Theorem~\ref{t:TensorSL}]
We prove~$(\dSL{\mssd^\otym{}}{\mssm^\otym{}})$ and conclude the assertion by~\eqref{eq:EquivalenceRadStoL}.

\paragraph{Step I: comparison with the maximal function} By~$(\SL{\mssm}{\mssd})$, resp.~$(\SL{\mssm'}{\mssd'})$, the Dirichlet spaces~$(\mbbX,\mcE)$ and~$(\mbbX',\mcE')$ are both irreducible.
Fix~$A\in\T$ and~$A'\in\T'$ with~$\mssm A, \mssm'A'\in (0,\infty)$ to be chosen later, and set~$A^\otym{}\eqdef A\times A'$.
Let~$\nu$, resp.~$\nu'$, be any probability measure on~$(X,\A)$, resp.~$(X',\A')$, equivalent to~$\mssm$, resp.~$\mssm'$, and observe that~$\nu^\otym{}\eqdef \nu\otimes\nu'$ is a probability measure on~$\A^\otym{}$ equivalent to~$\mssm^\otym{}$.

On the one hand, by~\eqref{eq:TensorSemigroup} and by Theorem~\ref{t:AriHino} applied to~$(\mcE,\dom)$, resp.~$(\mcE',\dom')$,
\begin{align}\label{eq:t:SLTensor:1}
\nu^\otym{}\text{-}\lim_{t\to 0} \paren{-2t\log T_t^\otym{}\car_{A^\otym{}}}= \nu\text{-}\lim_{t\to 0} \paren{-2t\log T_t\car_{A}} + \nu'\text{-}\lim_{t\to 0} \paren{-2t\log T_t'\car_{A'}} 
= \hr{\mssm,A}^2+\hr{\mssm',A'}^2 \fstop
\end{align}
Since~$(X,\mssd,\mssm)$, resp.~$(X',\mssd',\mssm')$, is a metric measure space,~$(\ScL{\mssm}{\T}{\mssd})$ coincides with~$(\SL{\mssm}{\mssd})$, resp.~$(\ScL{\mssm'}{\T'}{\mssd'})$ with~$(\SL{\mssm'}{\mssd'})$.
Furthermore, since~$A$ is $\T$-open, resp.~$A'$ is $\T'$-open, we have~$A\cap \cl_\T\inter_\T A=A$ and analogously for~$A'$ in place of~$A$.
Thus, combining~$(\SL{\mssm}{\mssd})$, resp.~$(\SL{\mssm'}{\mssd'})$, with~\cite[Lem.~4.19, Rem.~4.20(d)]{LzDSSuz21}, it follows from~\eqref{eq:t:SLTensor:1} that
\begin{align}\label{eq:t:SLTensor:2}
\nu^\otym{}\text{-}\lim_{t\to 0} \paren{-2t\log T_t^\otym{}\car_{A^\otym{}}} \leq \mssd(\emparg, A)^2+\mssd'(\emparg,A')^2 \as{\mssm^\otym{}}\fstop
\end{align}

On the other hand, by Theorem~\ref{t:AriHino} applied to~$(\mcE^\otym{},\dom^\otym{})$,
\begin{align}\label{eq:t:SLTensor:3}
\nu^\otym{}\text{-}\lim_{t\to 0} \paren{-2t\log T_t^\otym{}\car_{A^\otym{}}}= \hr{\mssm^\otym{},A^\otym{}}^2 \as{\mssm^\otym{}} \fstop
\end{align}

Combining~\eqref{eq:t:SLTensor:2} and~\eqref{eq:t:SLTensor:3}, for~$\mssm^\otym{}$-a.e.~$\mbfx\eqdef (x,x')$,
\begin{align*}
\hr{\mssm^\otym{},A^\otym{}}(\mbfx)^2 \leq&\ \mssd(x,A)^2+\mssd'(x',A')^2
\\
\defeq &\ \inf_{y\in A} \mssd(x,y)^2+\inf_{y'\in A'} \mssd'(x',y')^2
= \inf_{(y,y')\in A^\otym{}} \mssd(x,y)^2+\inf_{(y,y')\in A^\otym{}} \mssd'(x',y')^2
\\
\leq&\ \inf_{(y,y')\in A^\otym{}} \tparen{\mssd(x,y)^2+\mssd'(x',y')^2}
= \inf_{\mbfy\in A^\otym{}} \mssd^\otym{}(\mbfx,\mbfy)^2
\\
=&\ \paren{\inf_{\mbfy\in A^\otym{}} \mssd^\otym{}(\mbfx,\mbfy)}^2 = \mssd^\otym{}(\mbfx,A^\otym{})^2 \comma
\end{align*}
that is
\begin{align}\label{eq:t:SLTensor:4}
\hr{\mssm^\otym{},A^\otym{}} \leq \mssd^\otym{}(\emparg, A) \as{\mssm^\otym{}}
\end{align}

\paragraph{Step II: comparison with the intrinsic distance}
By Lemma~\ref{l:IntrinsicIneq} the topology induced by $\mssd_{\mssm^\otym{}}$ is separable.
Thus, by Lemma~\ref{l:Lemma}\ref{i:l:Lemma:3} applied to~$(\mbbX^\otym{},\mcE^\otym{})$,
\begin{align}\label{eq:t:SLTensor:5}
\mssd_{\mssm^\otym{}}(\emparg, A^\otym{})\leq \hr{\mssm^\otym{},A^\otym{}} \as{\mssm^\otym{}} \fstop
\end{align}

Combining~\eqref{eq:t:SLTensor:4} and~\eqref{eq:t:SLTensor:5} we thus have, for every~$A^\otym{}\eqdef A\times A'$ with~$A\in\T$,~$A'\in\T'$,
\begin{align}\label{eq:t:SLTensor:6}
\mssd_{\mssm^\otym{}}(\emparg, A^\otym{})\leq \mssd^\otym{}(\emparg, A^\otym{})\as{\mssm^\otym{}} \fstop
\end{align}
Now, fix~$\mbfx\eqdef (x,x')\in X^\otym{}$, choose~$A^\otym{}=A^\otym{}_r\eqdef B^{\mssd}_r(x)\times B^{\mssd'}_r(x')$ in~\eqref{eq:t:SLTensor:6}, and note that
\begin{align}\label{eq:t:SLTensor:7}
\inf_{r>0} \diam_{\mssd_\otym{}}(A^\otym{})=0\fstop
\end{align}
Then, by Lemma~\ref{l:SupDistanceShrinking}
\begin{align}\label{eq:t:SLTensor:8}
\sup_{r>0}\mssd_{\mssm^\otym{}}(\emparg, A^\otym{}_r)\leq \sup_{r>0} \mssd^\otym{}(\emparg, A^\otym{}_r) = \mssd^\otym{}(\emparg,\mbfx) \as{\mssm^\otym{}} \fstop
\end{align}
Let~$\Omega^\otym{}\in \A^\otym{}$ be a set of full $\mssm^\otym{}$-measure on which~\eqref{eq:t:SLTensor:8} holds, and fix~$\mbfy\in \Omega^\otym{}$.
For each~$r>0$, further let~$\mbfx_r\in A^\otym{}_r$ be satisfying~$\mssd_{\mssm^\otym{}}(\mbfy,\mbfx_r)\leq \eps+\mssd_{\mssm^\otym{}}(\mbfy,A^\otym{}_r)$.
Since~$\mbfx_r\in A^\otym{}_r$, there exists~$\T^\otym{}$-$\lim_{r\downarrow 0} \mbfx_r=\mbfx$ by~\eqref{eq:t:SLTensor:7}.
Thus, for every~$\mbfy\in\Omega^\otym{}$,
\begin{align*}
\mssd_{\mssm^\otym{}}(\mbfy,\mbfx)-\eps\leq \liminf_{r\downarrow 0} \mssd_{\mssm^\otym{}}(\mbfy,\mbfx_r) -\eps \leq \liminf_{r\downarrow 0} \mssd_{\mssm^\otym{}}(\mbfy,A^\otym{}_r) = \sup_{r>0} \mssd_{\mssm^\otym{}}(\mbfy,A^\otym{}_r) \leq \mssd^\otym{}(\mbfy,\mbfx)\comma
\end{align*}
where the first and last inequality hold by $\T^\otym{}$-lower semi-continuity of~$\mssd_{\mssm^\otym{}}$ and~\eqref{eq:t:SLTensor:8} respectively.
Since~$\eps>0$ was arbitrary, we conclude
\begin{align}\label{eq:t:SLTensor:9}
\mssd_{\mssm^\otym{}}(\mbfy,\mbfx)\leq \mssd^\otym{}(\mbfy,\mbfx)\comma \qquad \mbfy\in\Omega^\otym{}\comma \mbfx\in X^\otym{} \fstop
\end{align}
Since~$(X,\mssd,\mssm)$, resp.~$(X',\mssd',\mssm')$ is a metric measure space in the sense of Definition~\ref{d:MMSp}, the set~$\Omega^\otym{}$ is~$\T^\otym{}$-dense in~$X^\otym{}$.
For each fixed~$\mbfy_0\in X^\otym{}$ we may thus find a sequence~$\seq{\mbfy_n}_n\subset \Omega^\otym{}$ with~$\T^\otym{}$-$\nlim \mbfy_n=\mbfy_0$.
Respectively by $\T^\otym{}$-lower semi-continuity of~$\mssd_{\mssm^\otym{}}$, the inequality~\eqref{eq:t:SLTensor:9}, and by $\T^\otym{}$-continuity of~$\mssd^\otym{}$,
\[
\mssd_{\mssm^\otym{}}(\mbfy_0,\mbfx) \leq \nliminf \mssd_{\mssm^\otym{}}(\mbfy_n,\mbfx) \leq \nliminf \mssd^\otym{}(\mbfy_n,\mbfx)= \mssd^\otym{}(\mbfy_0,\mbfx) \comma
\]
which concludes the proof by arbitrariness of~$\mbfx$, and~$\mbfy_0$.
\end{proof}

\subsubsection{Tensorization of the intrinsic distance}
As a consequence of the tensorization of the Sobolev-to-Lipschitz property we further obtain the following identification of the intrinsic distance of product forms.

\begin{corollary}\label{c:IntrinsicDTensor}
Let~$\mbbX\eqdef (X,\mssd,\mssm)$ be a metric measure space (Dfn.~\ref{d:MMSp}), and~$(\mcE,\dom)$ be a quasi-regular strongly local Dirichlet form on~$\mbbX$ satisfying~$(\Rad{\mssd}{\mssm})$ and~$(\SL{\mssm}{\mssd})$.
Further let~$(\mbbX',\mcE')$ be satisfying the same assumptions as~$(\mbbX,\mcE)$.
Then, 
\[
\mssd_{\mssm^\otym{}}=\mssd^\otym{}=\mssd_\mssm\otimes\mssd_{\mssm'} \fstop
\]
\begin{proof}
By Theorem~\ref{t:TensorRad} we conclude~$(\Rad{\mssd^\otym{}}{\mssm^\otym{}})$, hence~$\dRad{\mssd^\otym{}}{\mssm^\otym{}}$ by~\eqref{eq:EquivalenceRadStoL}.
By Theorem~\ref{t:TensorSL} we conclude~$(\dSL{\mssd^\otym{}}{\mssm^\otym{}})$.
Thus, the first equality holds.
By~$(\Rad{\mssd}{\mssm})$, resp.~$(\Rad{\mssd'}{\mssm'})$, and~\eqref{eq:EquivalenceRadStoL} we conclude~$(\dRad{\mssd}{\mssm})$, resp.~$(\dRad{\mssd'}{\mssm'})$.
Together with the assumption of~$(\dSL{\mssd}{\mssm})$ and~$(\dSL{\mssd'}{\mssm'})$, this implies~$\mssd=\mssd_\mssm$ and~$\mssd'=\mssd_{\mssm'}$, which shows the second equality.
\end{proof}
\end{corollary}

\subsubsection{Doubling-and-Poincar\'e spaces}\label{sss:DandP}
Let~$(X,\mssd,\mssm)$ be a metric measure space in the sense of Definition~\ref{d:MMSp}.
In this short section we show that, under measure doubling and a weak Poincar\'e inequality for the factors, the Sobolev-to-Lipschitz property of the Cheeger energy tensorizes.
We start by recalling the necessary definitions.

\begin{definition}[Measure doubling and Poincar\'e inequality]\label{d:DP}
We say that~$(X,\mssd,\mssm)$ is (\emph{measure}) \emph{doubling} if there exists a constant~$C>0$ such that
\begin{equation}\tag{$\Dou{\mssd}{\mssm}$}\label{eq:Doubling}
\mssm B_{2r}(x) \leq C\, \mssm B_r(x)\comma \qquad x\in X\comma r>0\fstop
\end{equation}

We say that~$(X,\mssd,\mssm)$ satisfies a \emph{weak $(1,2)$-Poincar\'e inequality} if, for some constants~$c,\lambda>0$,
\begin{equation}\tag{$\Poi{\mssd}{\mssm}$}\label{eq:Poincare}
\bint_{B_r(x)}  \abs{f-f_{x,r}} \diff\mssm \leq c r \paren{\bint_{B_{\lambda r}(x)}\slo[\mssd]{f}^2\diff \mssm}^{1/2}\comma \qquad f\in \Lip_\loc(\mssd)\comma \quad x\in X\comma r>0\fstop
\end{equation}
(Here,~$\bint_A f\diff\mssm$ denotes the averaged integral of~$f$ over a Borel set~$A$, and~$f_{x,r}\eqdef \bint_{B_r(x)} f\diff\mssm$.)

We write~$(\DP{\mssd}{\mssm})$ to indicate that~$(X,\mssd,\mssm)$ satisfies both~$(\Dou{\mssd}{\mssm})$ and~$(\Poi{\mssd}{\mssm})$.
\end{definition}

On a metric measure space satisfying~$(\DP{\mssd}{\mssm})$ we have the following self-improvement of the continuous-Sobolev-to-Lipschitz property to the Sobolev-to-Lipschitz property.

\begin{lemma}[Self-improvement of~$(\cSL{\T}{\mssm}{\mssd})$ to~$(\SL{\mssm}{\mssd})$]\label{l:SLcSL}
Let~$\mbbX\eqdef (X,\mssd,\mssm)$ be an infinitesimally Hilbertian metric measure space satisfying~$(\DP{\mssd}{\mssm})$. 
Then,~$(\mbbX,\Ch[w,\mssd,\mssm])$ satisfies~$(\SL{\mssm}{\mssd})$ if and only if it satisfies~$(\cSL{\T}{\mssd}{\mssm})$.
\begin{proof}
One implication holds by definition, cf.~\eqref{eq:EquivalenceRadStoL}.
For the reverse implication it suffices to note that every $f\in\DzLocB{\mssm}$ admits a $\T$-continuous (in fact $\mssd$-Lipschitz) representative, by e.g.~\cite{HajKos00}, also cf.~\cite[Prop.~2.11]{AmbPinSpe15}.
\end{proof}
\end{lemma}

In light of this self-improvement, for metric measure spaces satisfying~$(\DP{\mssd}{\mssm})$ the tensorization of the Sobolev-to-Lipschitz property takes the following strong form.

\begin{corollary} 
Let~$\mbbX\eqdef (X,\mssd,\mssm)$, resp.~$\mbbX'\eqdef(X',\mssd',\mssm')$, be an infinitesimally Hilbertian metric measure space satisfying~$(\DP{\mssd}{\mssm})$, resp.~$(\DP{\mssd'}{\mssm'})$.
If~$(\mbbX,\Ch[w,\mssd,\mssm])$, resp.~$(\mbbX',\Ch[w,\mssd',\mssm'])$, satisfies~$(\cSL{\T}{\mssm}{\mssd})$, resp.~$(\cSL{\T'}{\mssm'}{\mssd'})$, then~$\Ch[w,\mssd^\otym{},\mssm^\otym{}]$ satisfies~$(\SL{\mssm^\otym{}}{\mssd^\otym{}})$.

\begin{proof}
Note that~$(\DP{\mssd}{\mssm})$ and~$(\DP{\mssd'}{\mssm'})$ together imply~$(\DP{\mssd^\otym{}}{\mssm^\otym{}})$.
The assertion holds combining Theorem~\ref{t:TensorSL} with Lemma~\ref{l:SLcSL}.
\end{proof}
\end{corollary}

\subsubsection{Tensorization of the Varadhan short-time asymptotics}
As a further consequence of the tensorization of both the Rademacher and the Sobolev-to-Lipschitz property, we further obtain the tensorization of Varadhan's short-time asymptotics for the heat flow.

Recall the notation for semigroups in~\S\ref{sss:Semigroups}. 
Whenever it exists, we denote by~$p_t(x,y)$ the \emph{heat kernel} on~$X$, i.e.\ the integral kernel of~$T_t$.
The kernel~$p_t^\otym{}(\mbfx,\mbfy)$ on~$X^\otym{}$ is defined analogously.

\begin{corollary}
Let~$\mbbX\eqdef (X,\mssd,\mssm)$ be an infinitesimally Hilbertian metric measure space satisfying~$(\DP{\mssd}{\mssm})$, and~\ref{ass:Polish}, with~$\Ch[\mssd,\mssm]$ satisfying~$(\SL{\mssm}{\mssd})$.
Further let~$\mbbX'$ be satisfying the same assumptions as~$\mbbX$.
Then, for every~$\mbfx\eqdef(x,x'), \mbfy\eqdef (y,y')\in X^\otym{}$,
\[
\lim_{t\downarrow 0} \tparen{-2t \log p_t^\otym{}(\mbfx,\mbfy)} =\mssd_{\mssm^\otym{}}(\mbfx,\mbfy)^2= \mssd^\otym{}(\mbfx,\mbfy)^2 = \mssd_\mssm(x,y)^2+\mssd_{\mssm'}(x',y')^2\fstop
\]
\end{corollary}

\begin{proof}
It suffices to verify the assumptions in~\cite[Thm.~4.1]{Ram01} and~\cite[Thm.~0.1]{Stu95b}.
Firstly,~$\mbbX^\otym{}$ satisfies~\ref{ass:Polish} since so do its factors.
On every space satisfying~\ref{ass:Polish}, our definition~\eqref{eq:IntrinsicD} of intrinsic distance coincides with the one in~\cite[p.~282]{Ram01} by~\cite[Prop.~2.31]{LzDSSuz20}.
Secondly, the Cheeger energy~$\Ch[\mssd^\otym{},\mssm^\otym{}]$ is strongly regular on~$\mbbX^\otym{}$ in light of Corollary~\ref{c:IntrinsicDTensor}.
Finally, the validity of~$(\DP{\mssd^\otym{}}{\mssm^\otym{}})$ follows from that of~$(\DP{\mssd}{\mssm})$ and~$(\DP{\mssd'}{\mssm'})$.
The inequality $\geq$ for the first equality then follows from~\cite[Thm.~4.1]{Ram01}.
The inequality $\leq$ is a consequence of the upper heat-kernel estimate~\cite[Thm.~0.1]{Stu95b}.
The second and third equalities follow from Corollary~\ref{c:IntrinsicDTensor}.
\end{proof}

\section{Direct integrals}
In this section we study the validity of the Rademacher and Sobolev-to-Lipschitz properties for direct integrals of Dirichlet forms as introduced in~\cite{LzDS20}, to which we refer the reader for a complete discussion; see also~\cite{LzDSWir21} for additional results, and~\cite{Kuw21} for a different but related decomposition.

We briefly recall the main definitions. For the sake of simplicity, we confine ourselves to the case of probability spaces. We expect that the result in this section can be adapted to the case of $\sigma$-finite measures by means of Propositions~\ref{p:Locality} and~\ref{p:LocalityProbab}.

Throughout this section, let~$\mbbX$ be satisfying~\ref{ass:Luzin} with~$\mssm$ a \emph{probability} measure, and~$(Z,\mcZ,\nu)$ be a countably generated probability space.
A \emph{disintegration} of~$\mssm$ over~$(Z,\mcZ,\nu)$ is a family~$\seq{\mssm_\zeta}_{\zeta\in Z}$ of non-zero measures on~$(X,\A)$ so that~$\zeta\mapsto \mssm_\zeta A$ is $\nu$-measurable for every~$A\in\A$ and
\begin{equation}\label{eq:Disintegration}
\mssm A= \int_Z \mssm_\zeta A \diff\nu(\zeta) \comma \qquad A\in\A\fstop
\end{equation}
A disintegration is \emph{separated} if there exists a family of pairwise disjoint sets~$\seq{A_\zeta}_{\zeta\in Z}\subset \A^\mssm$ so that~$A_\zeta$ is $\mssm_\zeta$-conegligible for $\nu$-a.e.~$\zeta\in Z$.

Henceforth, let~$\seq{\mssm_\zeta}_{\zeta\in Z}$ be a separated disintegration of~$\mssm$ over~$(Z,\mcZ,\nu)$ as above.
For a fixed Borel $\mssm$-representative~$\rep f$ of~$f\in L^2(\mssm)$, write~$f_\zeta\eqdef \tclass[\mssm_\zeta]{\rep f}$.
As it turns out, the choice of the $\mssm$-representative~$\rep f$ of~$f$ is immaterial, and the assignment~$\iota\colon f\mapsto (\zeta\mapsto f_\zeta)$ defines a unitary isomorphism providing a representation of~$L^2(\mssm)$ as the direct integral of the Hilbert spaces $L^2(\mssm_\zeta)$, see~\cite[Prop.~2.25(i)]{LzDS20}, viz.
\[
L^2(\mssm) \cong \dint{Z} L^2(\mssm_\zeta) \fstop
\]
Again since~$\seq{\mssm_\zeta}_{\zeta\in Z}$ is separated, such isomorphism is additionally a Riesz homomorphism, i.e.\ satisfying
\begin{equation}\label{eq:OrderIsoL1}
f\geq 0 \as{\mssm} \quad \iff \quad f_\zeta \geq 0 \as{\mssm_\zeta} \quad \forallae{\nu} \zeta\in Z \fstop
\end{equation}

\begin{lemma}\label{l:OrderPreserving}
Let~$f\in L^\infty(\mssm)\subset L^2(\mssm)$. Then,
\[
\norm{f}_{L^\infty(\mssm)}=\nu\text{-}\esssup\norm{f_\zeta}_{L^\infty(\mssm_\zeta)} \fstop
\]
\end{lemma}

\begin{proof}
It follows from~\eqref{eq:OrderIsoL1} that if~$\abs{f}\leq a$ $\mssm$-a.e.\ for some constant~$a>0$, then~$\abs{f_\zeta}\leq a$ $\mssm_\zeta$-a.e.\ for $\nu$-a.e.~$\zeta\in Z$.
On the one hand, as a consequence,~$\norm{f_\zeta}_{L^\infty(\mssm_\zeta)}\leq \norm{f}_{L^\infty(\mssm)}$ for $\nu$-a.e.~$\zeta\in Z$, which shows the inequality~`$\geq$' in the assertion.
On the other hand, for every~$a<\norm{f}_{L^\infty(\mssm)}$ there exists a set~$A_a\in\A$ with~$\mssm A>0$ and~$\abs{f}\geq a$ $\mssm$-a.e.\ on~$A_a$.
Thus, by~\eqref{eq:Disintegration} and since the disintegration is separated, there exists a set~$B\in\mcZ$ of positive $\nu$-measure so that~$\mssm_\zeta (A_a\cap A_\zeta)>0$ for every~$\zeta\in B$, and~$\abs{f_\zeta}\geq a$ $\mssm_\zeta$-a.e.\ on~$A_a\cap A_\zeta$ again by~\eqref{eq:OrderIsoL1}.
Since~$\mssm_\zeta( A_a\cap A_\zeta)>0$ we have~$\norm{f_\zeta}_{L^\infty(\mssm_\zeta)}>a$, and since~$\nu B>0$ we conclude that~$\nu$-$\esssup \norm{f_\zeta}_{L^\infty(\mssm_\zeta)}>a$.
The inequality~`$\leq$' in the assertion follows by arbitrariness of~$a<\norm{f}_{L^\infty(\mssm)}$.
\end{proof}

The next definition of direct integrals of Dirichlet forms is taken from~\cite{LzDS20}.
We present here a simplified definition for probability spaces, and refer the reader to~\cite[\S2]{LzDS20} for a detailed discussion.

\begin{definition}[Direct integrals of Dirichlet forms,~{\cite[Dfn.s~2.11,~2.26, and~2.31]{LzDS20}}]
For each~$\zeta\in Z$ let~$(\mcE_\zeta,\dom_\zeta)$ be a quasi-regular strongly local Dirichlet form on~$L^2(\mssm_\zeta)$, and assume that~$\zeta\mapsto \dom_\zeta$ is a $\nu$-measurable field of Hilbert subspaces of~$\zeta\mapsto L^2(\mssm_\zeta)$.
The \emph{direct integral of Dirichlet forms} of the field~$\zeta\mapsto (\mcE_\zeta,\dom_\zeta)$ is the Dirichlet form
\begin{align}\label{eq:d:DirInt:0}
\begin{aligned}
\dom\eqdef&\ \set{f\in L^2(\mssm)\comma \int_Z \paren{\mcE_\zeta(f_\zeta) + \norm{f_\zeta}_{L^2(\mssm_\zeta)}^2} \diff\nu(\zeta) <\infty}\comma
\\
\mcE(f,g)\eqdef &\ \int_Z \mcE_\zeta(f_\zeta,g_\zeta)\diff\nu(\zeta) \comma \qquad f,g\in\dom \fstop
\end{aligned}
\end{align}
\end{definition}

In the rest of this section, we will work under the following standing assumption.

\begin{assumption}\label{ass:Superposition}
Let~$(\mbbX,\mcE)$ be a \emph{conservative} quasi-regular strongly local Dirichlet space, with~$\mbbX$ satisfying~\ref{ass:Luzin} and~$\mssm$ a \emph{probability} measure.
We assume that there exists a countably generated probability space~$(Z,\mcZ,\nu)$ with the following property:
for every~$\zeta\in Z$ there exists a quasi-regular strongly local Dirichlet space~$(\mbbX_\zeta,\mcE_\zeta)$ with~$\mbbX_\zeta=\mbbX$ for~$\nu$-a.e.~$\zeta\in Z$, and $(\mcE,\dom)$ is the direct-integral Dirichlet form of~$\zeta\mapsto (\mcE_\zeta,\dom_\zeta)$.
\end{assumption}

Everywhere in the following, without loss of generality up to removing a $\nu$-negligible set from~$Z$, we may and will assume that~$\mbbX_\zeta=\mbbX$ for \emph{every}~$\zeta\in Z$.
Whereas we require~$\mbbX_\zeta=\mbbX$, we stress that we do not require~$(\mbbX,\mcE_\zeta)$ to satisfy~\ref{ass:Luzin}, that is, we do \emph{not} require that~$\supp[\mssm_\zeta]=X$.

\begin{lemma}
Let~$(\mbbX,\mcE)$ be satisfying Assumption~\ref{ass:Superposition}. Then,
\begin{align}\label{eq:CdCDirInt:3}
f \in \DzLocB{\mssm} \iff f_\zeta\in \DzLocB{\mssm_\zeta} \quad \forallae{\nu} \zeta\in Z\fstop
\end{align}

\begin{proof}
Since~$(\mcE,\dom)$ is conservative by assumption, and since~$\mssm$ is a finite measure, we have~$\car\in \dom$, hence~$\car\in\dom_\zeta$ for every (without loss of generality, as opposed to: $\nu$-a.e.)~$\zeta\in Z$.
By Lemma~\ref{l:ModerateDomain} we have that
\begin{align}\label{eq:CdCDirInt:1}
\DzLocB{\mssm}=\DzB{\mssm} \qquad \text{and} \qquad \DzLocB{\mcE_\zeta \ \mssm_\zeta}=\DzB{\mcE_\zeta \ \mssm_\zeta} \comma \quad \zeta\in Z\fstop
\end{align}
By~\cite[\S{V.3}, Exercise~3.2(2), p.~216]{BouHir91} the form~$(\mcE,\dom)$ admits carr\'e du champ operator
\begin{equation*}
\cdc(f)(x)=\int \cdc_\zeta(f_\zeta)(x) \diff\nu(\zeta) \comma \qquad f\in \domb \fstop
\end{equation*}
By Lemma~\ref{l:OrderPreserving} we conclude
\begin{equation}\label{eq:CdCDirInt:2}
\norm{\cdc(f)}_{L^\infty(\mssm)}=\nu\text{-}\esssup \norm{\cdc_\zeta(f_\zeta)}_{L^\infty(\mssm_\zeta)}\comma \qquad f\in \domb \fstop \fstop
\end{equation}
Combining~\eqref{eq:d:DirInt:0} and~\eqref{eq:CdCDirInt:2} we have therefore that
\begin{align*}
\DzB{\mssm}= \set{f\in L^2(\mssm)\comma f_\zeta \in \DzB{\mssm_\zeta} \forallae{\nu} \zeta\in Z} \fstop
\end{align*}
By~\eqref{eq:CdCDirInt:1} we finally conclude that
\begin{align}
f \in \DzLocB{\mssm} \iff f_\zeta\in \DzLocB{\mssm_\zeta} \forallae{\nu} \zeta\in Z\fstop &\qedhere
\end{align}
\end{proof}
\end{lemma}

\begin{theorem}\label{t:RadDInt}
Let~$(\mbbX,\mcE)$ be satisfying Assumption~\ref{ass:Superposition}, and $\mssd\colon X^\tym{2}\to [0,\infty]$ be an extended pseudo-distance.
Then, the following are equivalent:
\begin{enumerate}[$(i)$]
\item $(\mbbX,\mcE)$ satisfies~$(\Rad{\mssd}{\mssm})$;
\item $(\mbbX,\mcE_\zeta)$ satisfies~$(\Rad{\mssd}{\mssm_\zeta})$ for~$\nu$-a.e.~$\zeta\in Z$.
\end{enumerate}

\begin{proof}
Fix~$\rep f\in \bLipu(\mssd,\A)$.
Assume first~$(\Rad{\mssd}{\mssm_\zeta})$ for $\nu$-a.e.~$\zeta\in Z$.
Then,~$\cdc_\zeta(f_\zeta)\leq 1$ $\mssm_\zeta$-a.e.\ for $\nu$-a.e.~$\zeta\in Z$ by the assumption, and combining this fact with~\eqref{eq:CdCDirInt:3} proves~$(\Rad{\mssd}{\mssm})$.

Vice versa, assume~$(\Rad{\mssd}{\mssm})$. Then~$\cdc(f)\leq 1$ $\mssm$-a.e.\ by assumption, and combining this fact with~\eqref{eq:CdCDirInt:3} proves~$(\Rad{\mssd}{\mssm_\zeta})$ for $\nu$-a.e.~$\zeta\in Z$.
\end{proof}
\end{theorem}

No statement analogous to Theorem~\ref{t:RadDInt} holds for the Sobolev-to-Lipschitz property.
This is discussed in some detail~\cite[Ex.~4.7]{LzDSSuz20} for the simplest possible direct-integral form, namely the direct sum of two forms each defined on a quasi-connected component of the resulting direct-sum space.
In~\cite[Ex.~4.7]{LzDSSuz20} however, the measures~$\mssm_\zeta$, $\zeta\in\set{\pm 1}$ are absolutely continuous w.r.t.~$\mssm$, and one might expect the situation to improve in the case when~$\mssm$ is properly `disintegrated' into $\mssm$-negligible fibers (as opposed to: split into $\mssm$-non-negligible components).
The next example shows that this is not the case, and again the Sobolev-to-Lipschitz property is not preserved by direct integration.

\begin{example}[cf., e.g.,~{\cite[Ex.~2.32]{LzDS20}}]
Let~$I\eqdef [0,1]$, and~$X\eqdef I^\tym{2}$ with standard topology, $2$-dimensional Euclidean distance~$\mssd_\tym{2}$, Borel $\sigma$-algebra, and endowed with the $2$-dimensional Lebesgue measure~$\mssm\eqdef \Leb^2$.
Consider the form
\begin{align*}
\dom\eqdef& \set{f\in L^2(X): f(\emparg,x_2)\in W^{1,2}(I) \quad \forallae{\Leb^1} x_2\in I}
\\
\mcE(f)\eqdef& \int_{I} \abs{\partial_1 f(x_1,x_2)}^2 \dLeb^2(x_1,x_2) \fstop
\end{align*}
The form~$(\mcE,\dom)$ is the direct integral over~$(Z,\nu)\eqdef (I,\Leb^1)$ of the forms
\[
\mcE_{\zeta}(f)\eqdef \int_I \abs{\diff f(\emparg,\zeta)}^2\dLeb^1\comma \qquad f\in \dom_\zeta\eqdef W^{1,2}(I) \comma
\]
with~$\mssm_\zeta\eqdef \Leb^1\otimes \delta_{\zeta}$.
In particular,~$\dom= L^2(Z; W^{1,2}(I))$.

Now, since~$\car\in\dom$ and~$\car_\zeta\in W^{1,2}(I)\defeq \dom_\zeta$ for every~$\zeta\in Z$, by~\eqref{eq:CdCDirInt:3} we have that~$\DzLocB{\mssm,\T}=\DzB{\mssm,\T}$, and analogously for~$\mssm_\zeta$ in place of~$\mssm$ for every~$\zeta\in Z$.
Thus, it suffices to show that there exists~$f\in \DzB{\mssm,\T}$ not admitting any $\mssd_{\tym{2}}$-Lipschitz $\mssm$-representative.
In fact, it is readily seen that $\zeta\mapsto \tparen{\mcE_\zeta, \dom_\zeta}$ is the ergodic decomposition of~$(\mcE,\dom)$ in the sense of~\cite{LzDS20}, hence
\[
\DzB{\mssm,\T}= \set{f\in \Cb(X) : f_\zeta\in \DzB{\mssm_\zeta,\T} \forallae{\nu} \zeta\in Z} \fstop
\]
In particular, every function of the form~$f=\car\otimes g$ with~$g\in \Cb(I)$ satisfies~$f\in \DzB{\mssm,\T}$ and~$\mcE(f)=0$.
Taking~$g\notin \bLipu(I,\mssd)$ shows that~$f$ is in general not $\mssd_{\tym{2}}$-Lipschitz.
\end{example}

A positive result in the spirit of Theorem~\ref{t:RadDInt} holds only for~$(\cSL{\T}{\mssm}{\mssd})$ under further assumptions, granting a consistent choice of representatives.

\begin{theorem}\label{t:cSLDInt}
Let~$(\mbbX,\mcE)$ be satisfying Assumption~\ref{ass:Superposition}, and $\mssd\colon X^\tym{2}\to [0,\infty]$ be an extended pseudo-distance.
Further assume that~$(\mbbX,\mcE_\zeta)$ satisfies~$(\cSL{\T}{\mssm_\zeta}{\mssd})$ and~$\supp[\mssm_\zeta]=X$ for every~$\zeta\in B$ for some~$B\in\mcZ$ with~$\nu B>0$.
Then, $(\mbbX,\mcE)$ satisfies~$(\cSL{\T}{\mssm}{\mssd})$.

\begin{proof}
Fix~$f\in \DzB{\mssm,\T}$.
Then~$f_\zeta\in \DzB{\mssm_\zeta,\T}$ for $\nu$-a.e.~$\zeta\in Z$ by~\eqref{eq:CdCDirInt:3}.
In particular, there exists~$\zeta\in B$ with~$f_\zeta\in \DzB{\mssm_\zeta,\T}$.
For such fixed~$\zeta$, by~$(\cSL{\T}{\mssm_{\zeta}}{\mssd})$, there exists a $\T$-continuous $\mssm_\zeta$-representative of~$f_\zeta$, additionally an element of~$\Lipu(\mssd,\T)$.
Since~$\mssm_\zeta$ has full $\T$-support, such representative is unique, again denoted by~$f_\zeta$, and we conclude that~$f_\zeta\equiv f$ everywhere on~$X$.
Thus,~$f\in \Lipu(\mssd,\T)$ and the proof is concluded.
\end{proof}
\end{theorem}

With a proof similar to the one of Theorem~\ref{t:cSLDInt}, one can show the following.
\begin{proposition}
Let~$(\mbbX,\mcE)$ be satisfying Assumption~\ref{ass:Superposition}, and $\mssd\colon X^\tym{2}\to [0,\infty]$ be an extended pseudo-distance.
Further assume that there exists~$\mssd\text{-}\supp[\mssm]=X$, and that~$(\mbbX,\mcE_\zeta)$ satisfies~$(\dcSL{\mssd}{\mssm_\zeta}{\mssd})$ and there exists~$\mssd\text{-}\supp[\mssm_\zeta]=X$ for every~$\zeta\in B$ for some~$B\in\mcZ$ with~$\nu B>0$.
Then, $(\mbbX,\mcE)$ satisfies~$(\dcSL{\mssd}{\mssm}{\mssd})$.
\end{proposition}

\begin{remark}
The assumption in Theorem~\ref{t:cSLDInt} that~$\mssm_\zeta$ have full $\T$-support for all~$\zeta$ in a set of positive $\nu$-measure may at first sight seem to contrast our standing assumption that~$\seq{\mssm_\zeta}_{\zeta\in Z}$ be separated.
This is however not the case, since~$\mssm_\zeta$ may be in general concentrated on sets much smaller than its support.
For instance, let~$Y$ be a locally compact Polish space, and denote by~$\Upsilon$ the configuration space over~$Y$, i.e.\ the space of all locally finite point measures on~$Y$, endowed with the vague topology and the corresponding Borel $\sigma$-algebra.
Further let~$\sigma$ be a Radon measure on~$Y$ and, for~$s\in\R_+$, let~$\pi_{s\sigma}$ be the Poisson measure on~$\Upsilon$ with intensity measure~$s\sigma$.
Finally, let~$\lambda$ be a diffuse Borel probability measure on~$\R_+$.
The mixed Poisson measure on~$\Upsilon$ with L\'evy measure~$\lambda$ and intensity measure~$\sigma$ is the probability measure~$\mu_{\lambda,\sigma}\eqdef \int_{\R_+} \pi_{s\sigma}\diff\lambda(s)$.
Then,~$\seq{\pi_{s\sigma}}_{s\in\R_+}$ is a \emph{separated} disintegration of~$\mu_{\lambda,\sigma}$ over~$(\R_+,\lambda)$, yet~$\mu_{\lambda,\sigma}$ and each~$\pi_{s\sigma}$ has \emph{full} topological support on~$\Upsilon$; for details on all these constructions see e.g.~\cite{LzDSSuz21}.
We refer the reader to~\cite[Ex.~3.13]{LzDS20} for further details and for the construction of a relevant direct integral of Dirichlet forms in this case.
\end{remark}

\end{document}